%% file: 8chromaticDoubleCritical.tex
\documentclass[a4paper, 10pt]{article}
\pdfoutput=1 
\usepackage[latin1]{inputenc}
\usepackage[T1]{fontenc} 

\usepackage[dvips]{graphicx}
\usepackage{color}

\usepackage{amsthm} 
\usepackage{amsmath}
\usepackage{amssymb}
\usepackage{latexsym}
\usepackage{url}

\theoremstyle{plain}
\newtheorem{theorem}{Theorem}[section]
\newtheorem{proposition}[theorem]{Proposition}
\newtheorem{observation}[theorem]{Observation}
\newtheorem{corollary}[theorem]{Corollary}
\newtheorem{lemma}[theorem]{Lemma}

\newtheorem{conjecture}[theorem]{Conjecture}
\newtheorem{problem}[theorem]{Problem}

\newcommand{\mc}[1]{\mathcal{ #1 }}

\usepackage[numbers]{natbib}
\usepackage{subfigure}
\title{Complete and almost complete minors in double-critical
  $8$-chromatic graphs}
\author{
   \small Anders Sune Pedersen \\
    \small Dept. of Mathematics and Computer Science \\
   \small University of Southern Denmark \\
  \small Campusvej 55, 5230 Odense M, Denmark \\
\small \texttt{asp@imada.sdu.dk} \vspace{5mm} \\
  \small MR Subject Classification: 05C15, 05C69}
\begin{document}
%\linenumbers
\maketitle
\begin{abstract}
  A connected $k$-chromatic graph $G$ is said to be {\it double-critical}
  if for all edges $uv$ of $G$ the graph $G - u - v$ is
  $(k-2)$-colourable. A longstanding conjecture of Erd\H{o}s and
  Lov\'asz states that the complete graphs are the only
  double-critical graphs. Kawarabayashi, Pedersen and Toft
  [\emph{Electron. J. Combin.}, 17(1): Research Paper 87, 2010] proved
  that every double-critical $k$-chromatic graph with $k \leq 7$
  contains a $K_k$ minor. It remains unknown whether an arbitrary
  double-critical $8$-chromatic graph contains a $K_8$ minor, but in
  this paper we prove that any double-critical $8$-chromatic contains
  a $K_8^-$ minor; here $K_8^-$ denotes the complete $8$-graph with
  one edge missing. In addition, we observe that any double-critical
  $8$-chromatic graph with minimum degree different from $10$ and $11$
  contains a $K_8$ minor.
\end{abstract}
\section{Introduction and motivation}
At the very center of the theory of graph colouring is Hadwiger's Conjecture which dates back to 1942. It states that every $k$-chromatic
graph\footnote{All graphs considered in this paper are undirected,
  simple, and finite. The reader is referred to
  Section~\ref{sec:Prelim} for basic graph-theoretic terminology and
  notation.} contains a $K_k$ minor.
\begin{conjecture}[Hadwiger~\cite{MR0012237}]
If $G$ is a $k$-chromatic graph, then $G$ contains a $K_k$ minor.
\end{conjecture}
Hadwiger~\cite{MR0012237} showed that the conjecture holds for $k \leq
4$, the case $k=4$ being the first non-trivial instance of the
conjecture. Later, several short and elegant proofs for the case $k=4$
were found; see, for instance, \cite{MR1411244}. The case $k=5$ was
studied independently by Wagner~\cite{MR1513158}, who proved that the
case $k=5$ is equivalent to the Four Colour Problem. In the early
1960s, Dirac~\cite{MR0139160} and Wagner~\cite{MR0121309},
independently, proved that every $5$-chromatic graph $G$ contains a
$K_5^-$ minor, that is, $G$ contains, as a minor, a complete $5$-graph
with at most one edge missing. The case $k=5$ of Hadwiger's Conjecture
was finally settled in the affirmative with Appel and Haken's proof of
the Four Colour Theorem~\cite{MR0543792, MR0543793} (an improved proof
was subsequently published in 1997 by Robertson et
al.~\cite{MR1441258}). In 1964, Dirac~\cite{MR0162241} proved that
every $6$-chromatic graph contains a $K_6^-$ minor (see
\cite[p. 257]{MR1373659} for a short version of Dirac's proof), and,
in 1993, Robertson, Seymour and Thomas~\cite{MR1238823} proved, using
the Four Colour Theorem, that every $6$-chromatic graph contains a
$K_6$ minor. Thus, Hadwiger's Conjecture has been settled in the
affirmative for each $k \leq 6$, but remains unsettled for all $k \geq
7$. In the early 1970s, Jakobsen~\cite{MR0295963, MR0340108,
  MR0323641} proved that for $k=7,8$, and $9$ every $k$-chromatic
graph contains, as a minor, $K_7^{--}$, $K_7^-$, and $K_7$,
respectively, and these results seem to be the best obtained so far in
support of Hadwiger's Conjecture for the cases $k=7,8$, and $9$. (Here
$K_7^{-}$ denotes the complete $7$-graph with one edge missing, while
$K_7^{--}$ denotes a complete $7$-graph with two edges missing. There
are two non-isomorphic complete $7$-graphs with two edges missing.)
The interested reader is referred to \cite{JensenToft95, MR1411244}
for a thorough survey of Hadwiger's Conjecture and related
conjectures.

Another longstanding conjecture in the theory of graph colouring is
the so-called Erd\H{o}s-Lov\'asz Tihany Conjecture which dates back to
1966. This conjecture states, in an interesting special case, that the
complete graphs are the only double-critical
graphs~\cite{TihanyProblem2}. A connected $k$-chromatic graph $G$ is
\emph{double-critical} if for all edges $uv$ of $G$ the graph $G - u -
v$ is $(k-2)$-colourable.
\begin{conjecture}[Erd\H{o}s \& Lov\'asz~\cite{TihanyProblem2}]\label{conj:DC}
  If $G$ is a double-critical $k$-chromatic graph, then $G$ is
  isomorphic to $K_k$.
\end{conjecture}
Conjecture~\ref{conj:DC}, which we call the \emph{Double-Critical
  Graph Conjecture}, is settled in the affirmative for all $k \leq 5$,
but remains unsettled for all $k \geq 6$
\cite{MR995391,MR882614,MR1221590}. As a relaxed version of the
Double-Critical Graph Conjecture the following conjecture was posed
in~\cite{KawarabayashiPedersenToftEJC2010}.
\begin{conjecture}[Kawarabayashi, Pedersen \&
  Toft~\cite{KawarabayashiPedersenToftEJC2010}]\label{conj:DCrelaxed}
  If $G$ is a double-critical $k$-chromatic graph, then $G$ contains a
  $K_k$ minor.
\end{conjecture}
Conjecture~\ref{conj:DCrelaxed} is, of course, also a relaxed version
of Hadwiger's Conjecture, and so we call it the \emph{Double-Critical
  Hadwiger Conjecture}; in~\cite{KawarabayashiPedersenToftEJC2010}, it
was settled in the affirmative for $k \in \{ 6, 7 \}$ (without use of
the Four Colour Theorem) but it remains open for all $k \geq 8$. Very
little seems to be known about complete minors in $8$-chromatic
graphs. The best result so far in the direction of proving the
Hadwiger Conjecture for $8$-chromatic graphs seems to be a theorem
published in 1970 by Jakobsen~\cite{MR0295963}; the theorem states
that every $8$-chromatic graph contains a $K_7^-$ minor. Corollary 7.3
in~\cite{KawarabayashiPedersenToftEJC2010} states that every
double-critical $k$-chromatic graph with $k \geq 7$ contains a $K_7$
minor. In this paper we prove that every double-critical $8$-chromatic
graph contains a $K_8^-$ minor. The proof of this result is
surprisingly complicated and uses a number of deep results by other
authors.
\section{Main results}
These are our main results.
 \begin{theorem}\label{th:mainTheorem2} 
   Every double-critical $8$-chromatic graph contains a $K_8^-$ minor.
\end{theorem}
\begin{corollary}\label{cor:mainCorollary}
  Every double-critical $k$-chromatic graph with $k \geq 8$ contains a
  $K_8^-$ minor.
\end{corollary}
In the case of minimum degree different from $10$ and $11$ we are able
to find `the edge missing in Theorem~\ref{th:mainTheorem2}'.
\begin{theorem}\label{th:mainTheorem}
  Every double-critical $8$-chromatic graph with minimum degree
  different from $10$ and $11$ contains a $K_8$ minor.
\end{theorem}
Our proofs of the above-mentioned results do not rely on the Four
Colour Theorem but they do rely on the following two deep results.
\begin{theorem}[(i) Song~\cite{MR2171368}; (ii) J\o
  rgensen~\cite{MR1283309}] \label{th:JorgensenAndSongOnGraphMinors}
Suppose $G$ is a graph on at least $8$ vertices.
\begin{itemize}
\item[\emph{(i)}]  If $G$ has more than \mbox{$\lceil (11n(G) -
    35)/2 \rceil$} edges, then $G$ contains a $K_8^-$ minor, and
\item[\emph{(ii)}] if $G$ has more than $6n(G) - 20$ edges, then $G$
  contains a $K_8$ minor.
\end{itemize}
\end{theorem}
\begin{proof}[Proof of Theorem~\ref{th:mainTheorem}.]
  Suppose $G$ is a double-critical $8$-chromatic graph with minimum
  degree $\delta(G)$. Then, according to
  Proposition~\ref{prop:ElemPropertiesOfDC}~(ii), $\delta(G) \geq
  9$. If $\delta(G) \geq 12$, then $|E(G)| \geq 6 n(G)$ and so, by
  Theorem~\ref{th:JorgensenAndSongOnGraphMinors}~(ii), $G \geq
  K_8$. If $\delta(G)=9$, then the desired result follows from
  Corollary~\ref{cor:minimumDegreeNine}.
\end{proof}
\begin{proof}[Proof of Theorem~\ref{th:mainTheorem2}.]
  Let $G$ denote a double-critical $8$-chromatic graph. By
  Theorem~\ref{th:mainTheorem}, we may assume $\delta(G) \geq 10$. If
  $\delta(G) \geq 11$, then $|E(G)| \geq 11n(G)/2$ and so, by
  Theorem~\ref{th:JorgensenAndSongOnGraphMinors}~(i), $G \geq K_8^-$. Suppose $\delta(G) = 10$, let $x$ denote a vertex of
  degree $10$ in $G$, and define $G_x := G[N(x)]$. Then, according to
  Observation~\ref{obs:432785932465}, $\Delta(\overline{G_x}) \leq
  3$. If $\Delta(\overline{G_x}) \leq 2$, then, by
  Proposition~\ref{prop:43287530893245}, $G \geq K_8$. If
  $\Delta(\overline{G_x}) = 3$ and $G_x$ contains at least one vertex
  of degree $9$, then, by Proposition~\ref{prop:importantProposition},
  $G \geq K_8^-$. If $\Delta(\overline{G_x}) = 3$ and $G_x$ contains
  no vertex of degree $9$, then, by
  Proposition~\ref{prop:importantProposition2}, $G \geq K_8^-$. This
  completes the proof.
\end{proof}
\begin{proof}[Proof of Corollary~\ref{cor:mainCorollary}]
  Let $G$ denote a double-critical $k$-chromatic graph with $k \geq
  8$. If $k=8$ or $\delta(G) \geq 11$, then the desired result follows
  from Theorem~\ref{th:mainTheorem2} or
  Theorem~\ref{th:JorgensenAndSongOnGraphMinors}~(i), respectively. Hence, by
  Proposition~\ref{prop:ElemPropertiesOfDC}~(ii), we may assume $k=9$ and
  $\delta(G)=10$; in this case we prove $G \geq K_8^-$ by an argument
  somewhat similar to the first part of the proof of
  Proposition~\ref{prop:83475342}. The details are omitted.
\end{proof}
\section{Preliminaries and notation} \label{sec:Prelim} We shall use
standard graph-theoretic terminology and notation as defined
in~\cite{BondyAndMurty2008, DiestelGT2006} with a few additions. Given
any graph $G$, $V(G)$ denotes the vertex set of $G$ and $E(G)$ denotes
the edge set, while $\overline{G}$ denotes the complement of $G$. The
\emph{order} of a graph $G$, that is, the number of vertices in $G$,
is denoted $n(G)$, and any graph on $n$ vertices is called an
\emph{$n$-graph}. A vertex of degree $k$ in a graph $G$ is said to be
a \emph{$k$-vertex} (of $G$). Given two graphs $H$ and $G$, the
\emph{complete join} of $G$ and $H$, denoted $G + H$, is the graph obtained
from two vertex-disjoint copies of $H$ and $G$ by joining each vertex
of the copy of $G$ to each vertex of the copy of $H$. For every positive integer
$k$ and graph $G$, $kG$ denotes the graph $\sum_{i=1}^k G$. Given any
edge-transitive graph $G$, any graph, which can be obtained from $G$
by removing one edge, is denoted $G^-$. The \emph{girth} of a graph
$G$ is the length of a shortest cycle in $G$; if $G$ is acyclic, then
the girth of $G$ is said to be infinite. Given any subset $X$ of the
vertex set $V(G)$ of a graph $G$, we let $G[X]$ denote the subgraph of $G$
induced by the vertices of $X$. The set of vertices of $G$ adjacent to
$v$ is called the \emph{neighbourhood of $v$} (in $G$), and it is
denoted $N_G(v)$ or $N(v)$. The set $N(v) \cup \{ v \}$ is called the
\emph{closed neighbourhood of $v$} (in $G$), and it is denoted
$N_G[v]$ or $N[v]$. The induced graph $G[N(v)]$ is referred to as the
\emph{neighbourhood graph of $v$} (w.r.t. $G$), and it is denoted
$G_v$. Given two graphs $G$ and $H$, we say that $H$ is a \emph{minor}
of $G$ (and that $G$ has an \emph{$H$ minor}) if there is a collection
$\{ V_h \mid h \in V(H) \}$ of non-empty, disjoint subsets of $V(G)$
such that the induced graph $G[V_h]$ is connected for each $h \in
V(H)$, and for any two adjacent vertices $h_1$ and $h_2$ in $H$ there
is at least one edge in $G$ joining some vertex of $V_{h_1}$ to some
vertex of $V_{h_2}$. The sets $V_h$ are called the \emph{branch sets}
of the minor $H$ of $G$. We may write $H \leq G$ or $G \geq H$, if $G$
contains an $H$ minor. In~\cite{KawarabayashiPedersenToftEJC2010}, a
number of basic results on double-critical graphs were determined. We
will make repeated use of these results and so, for ease of reference,
they are restated here. \\

In the remaining part of this section, we let $G$ denote a
non-complete double-critical $k$-chromatic graph with $k \geq
6$. Given any edge $xy \in E(G)$, define
\begin{eqnarray*}
  A(x,y) & := & N(x) \setminus N[y] \\
  B(x,y) & := & N(x) \cap N(y) \\
  C(x,y) & := & N(y) \setminus N[x] 
\end{eqnarray*}
\begin{proposition}[\cite{KawarabayashiPedersenToftEJC2010}]
\label{prop:ElemPropertiesOfDC} ~ 
\begin{itemize} \item[\emph{(i)}] The graph $G$ does not contain a
  complete $(k-1)$-graph as a subgraph;
\item[\emph{(ii)}] the graph $G$ has minimum degree at least $k+1$, and
\item[\emph{(iii)}] for all edges \mbox{$xy \in E(G)$} and all $(k-2)$-colourings of $G - x - y$,
  the set $B(x,y)$ of common neighbours of $x$ and $y$ in $G$ contains
  vertices from every colour class, in particular,
  $|B(x,y)| \geq k-2$.
\end{itemize}
\end{proposition}
\begin{proposition}\label{prop:antiMatchingBetweenAandB} If $G[A(x,y)]$ is a complete graph for some edge $xy \in E(G)$, then
  there is a matching of the vertices of $A(x,y)$ to the vertices of
  $B(x,y)$ in $\overline{G_x}$.
\end{proposition}
\begin{proof}
  Suppose $G[A(x,y)]$ is a complete graph for some edge $xy \in E(G)$,
  and let $G-x-y$ be coloured properly in the colours $1, 2, \ldots,
  k-3$, and $k-2$. The colours applied to $A(x,y)$ are all distinct,
  and so we may assume $A(x,y) = \{ a_1, \ldots, a_p \}$ where vertex
  $a_i$ is coloured $i$ for each $a_i \in A(x,y)$. According to
  Proposition~\ref{prop:ElemPropertiesOfDC}~(iii), each
  of the colours $1, 2, \ldots, k-3$, and $k-2$ appear at least once
  on a vertex of $B(x,y)$, say \mbox{$B(x,y) = \{ b_1, \ldots, b_q \}$}
  with vertex $b_i$ being coloured $i$ for each \mbox{$i \leq
    k-2$}. Also, $q \geq k-2$. Since $G[A(x,y) \cup \{ x \} ]$ is a
  complete graph, it follows from
  Proposition~\ref{prop:ElemPropertiesOfDC}~(i) that $p =
  |A(x,y)| \leq k-3$. Hence $p < q$, and $a_i$ and $b_i$ have the same
  colour for each $i \in [p]$, in particular, $\{ a_1 b_1, a_2 b_2,
  \ldots, a_p b_p \}$ is a matching of the vertices of $A(x,y)$ to
  vertices of $B(x,y)$ in $\overline{G_x}$.
\end{proof}
\begin{proposition}[\cite{KawarabayashiPedersenToftEJC2010}]
  If $A(x,y)$ is non-empty for some edge $xy \in E(G)$, then $\delta
  (G[A(x,y)]) \geq 1$, that is, the induced subgraph $G[A(x,y)]$ contains
  no isolated vertices. By symmetry, $\delta (G[C(x,y)]) \geq 1$, if
  $C(x,y)$ is non-empty.
\label{prop:noIsolatedInAxy}
\end{proposition}
Thus, by Proposition~\ref{prop:noIsolatedInAxy}, if $y$ is a vertex which
has degree $2$ in $\overline{G_x}$ then the two neighbours of $y$ in
$\overline{G_x}$ must be non-adjacent in $\overline{G_x}$.
\begin{proposition}[\cite{KawarabayashiPedersenToftEJC2010}]
\label{prop:kMinusTwoColourableNeighbourGraphGx}
\label{prop:structureOfGx}
\label{prop:doubleCriticalImpliesSixConnected}
 ~ \begin{itemize} 
\item[\emph{(i)}] For any vertex $x$ of $G$ not joined to all other vertices of $G$, $\chi(G_x) \leq k-3$;
\item[\emph{(ii)}] if $x$ is a vertex of degree $k+1$ in $G$, then the complement
  $\overline{G_x}$ consists of isolated vertices (possibly none) and
  cycles (at least one), where the length of each cycle is at least
  five, and 
\item[\emph{(iii)}] $G$ is $6$-connected.
\end{itemize}
\end{proposition}
\begin{proposition}[\cite{KawarabayashiPedersenToftEJC2010}]\label{prop:atLeast15vertices} There is no non-complete double-critical $8$-chromatic graph of
  order less than $15$.
\end{proposition}
% \begin{theorem}[\cite{KawarabayashiPedersenToftEJC2010}]
%   Every non-complete double-critical $k$-chromatic graph with $k \geq
%   6$ is $6$-connected.
% \label{th:doubleCriticalImpliesSixConnected}
% \end{theorem}
\section{Minimum degree $9$ and $K_8$
  minors} \label{sec:delta9contractionsToK_8}
\begin{proposition} \label{prop:83475342}
  If $G$ is a double-critical $8$-chromatic graph with
  a vertex $x$ of degree $9$, then $G_x \simeq
  \overline{C_8} + K_1$ or $G_x \simeq \overline{C_9}$.
\end{proposition}
\begin{proof}
  Suppose $G$ is a double-critical $8$-chromatic graph with a vertex
  $x$ of degree $9$. Now, according to
  Proposition~\ref{prop:structureOfGx}~(ii), $\overline{G_x}$ consists of
  isolated vertices and cycles (at least one cycle) of length at least
  $5$. Since $G_x$ consists of only nine vertices, it follows that
  $\overline{G_x}$ consists of exactly one cycle, which we denote
  $C_j$, and some isolated vertices. If $j \in \{ 5, 6 \}$, then $G[
  N[x]]$ is easily seen to contain $K_7$ as a subgraph, contrary to
  Proposition~\ref{prop:ElemPropertiesOfDC}~(i). Suppose $j =
  7$. Moreover, suppose that the vertex $x$ is not adjacent to all
  other vertices of $G$. Then, according to
  Proposition~\ref{prop:kMinusTwoColourableNeighbourGraphGx}~(i),
  $\chi(G_x) \leq 5$. However, the graph $G_x$, which is isomorphic to
  $\overline{C_7} + K_2$, is easily seen not be $5$-colourable. Thus,
  the vertex $x$ is adjacent to all other vertices of $G$, and so $G$
  is isomorphic to $\overline{C_7} + K_3$. However, the graph
  $\overline{C_7} + K_3$ is easily seen to be $7$-colourable, a
  contradiction. Thus, we must have $j \geq 8$, and so the desired
  result follows immediately.
\end{proof}
The proof of Proposition~\ref{prop:83475342} implies that any
double-critical $8$-chromatic graph with a vertex of degree $9$
contains $K_6^-$ as a subgraph.
\begin{corollary}\label{cor:minimumDegreeNine}
  Every double-critical $8$-chromatic graph with minimum
  degree $9$ contains a $K_8$ minor.
\end{corollary}
\begin{proof}
  Suppose $G$ is a double-critical $8$-chromatic graph with minimum
  degree $9$, and let $x$ denote a vertex of $G$ of degree $9$.
  Suppose that $G$ does not contain a $K_8$ minor. Then, according to
  Proposition~\ref{prop:atLeast15vertices}, there are at least $15$
  vertices in $G$, in particular, there is a vertex, which we shall
  call $z$, in $G - N[x]$. According to
  Proposition~\ref{prop:83475342}, there are two cases to consider:
  either $G_x \simeq \overline{C_8} + K_1$ or $G_x \simeq
  \overline{C_9}$.

  Suppose $G_x \simeq \overline{C_8} + K_1$, where $C_8 : v_0,v_1,v_2,
  \ldots, v_7$ and \mbox{$V(K_1) = \{ u \}$}. By
  Proposition~\ref{prop:doubleCriticalImpliesSixConnected}~(iii), $G$
  is $6$-connected, and so $G - u$ must be \mbox{$5$-connected}. Now,
  according to Menger's Theorem (see, for instance, \citep[Theorem
  9.1]{BondyAndMurty2008}), there is a collection $\mathcal{C}$ of
  five internally vertex-disjoint $(x,z)$-paths in $G-u$. Obviously,
  each path $P \in \mathcal{C}$ contains a vertex from $V(C_8)$, and
  we may assume that each of the paths $P \in \mathcal{C}$ contains
  exactly one vertex from $V(C_8)$. The fact that there are eight
  vertices in $V(C_8)$ and five vertex-disjoint $(x,z)$-paths in
  $\mathcal{C}$ going through $V(C_8)$ implies the existence of a pair
  of vertices $v_i$ and $v_{i+1}$ (modulo $8$) such that there is a
  $(v_i,z)$-path $Q_i$ and a $(v_{i+1},z)$-path $Q_{i+1}$ in $G-u$
  such that $Q_i$ and $Q_{i+1}$ are internally vertex-disjoint. We may
  assume $i=0$. Now, the $(v_0, v_1)$-path $Q_0 \cup Q_1$ in $G$ is
  contracted to an edge between $v_0$ and $v_1$. The resulting graph
  contains the graph $H \simeq \overline{C_8^-} + K_2$ as a subgraph,
  and $H$ can be contracted to $K_8$ by contracting the edges $v_2v_5$
  and $v_4v_7$. Thus, $G \geq K_8$. A similar argument shows that, if
  $G_x \simeq \overline{C_9}$, then $G \geq K_8$.
\end{proof}
\section{Minimum degree $10$ and $K_8$
  minors}\label{sec:delta10contractionsToK_8}
\begin{observation}\label{obs:432785932465} If $G$ is a double-critical $8$-chromatic graph with minimum
  degree $10$ and $\deg(x, G) = 10$, then $\Delta(\overline{G_x}) \leq
  3$.
\end{observation}
\begin{proof}
  Suppose $\Delta(\overline{G_x}) \geq 4$, and let $y$ denote a vertex
  which has degree $\geq 4$ in $\overline{G_x}$. Then $|A(x,y)| \geq
  4$ and, according to
  Proposition~\ref{prop:ElemPropertiesOfDC}~(iii), $|B(x,y)|
  \geq 6$. Thus, $\deg(x, G) \geq |A(x,y)| + |B(x,y)| + 1 \geq 11$,
  which contradicts the assumption $\deg(x, G)=10$.
\end{proof}
\begin{proposition}\label{prop:43287530893245}
  Suppose $G$ is a double-critical $8$-chromatic graph with minimum
  degree $10$, and suppose $G$ contains a vertex $x$ of degree $10$
  such that $\Delta(\overline{G_x}) \leq 2$. Then $G$ contains a $K_8$
  minor.
\end{proposition}
\begin{proof}
  If $\Delta(\overline{G_x}) = 0$, then $G_x \simeq K_{10}$, a
  contradiction. According to Proposition~\ref{prop:noIsolatedInAxy},
  no vertex of $\overline{G_x}$ has degree exactly $1$. Hence,
  $\Delta(\overline{G_x})=2$, and so the graph $\overline{G_x}$
  consists of cycles (at least one) and possibly some isolated
  vertices. If $\overline{G_x}$ has at least five isolated vertices,
  then it is easy to see that $G_x$ contains $K_7$ as a subgraph. If
  $\overline{G_x}$ has exactly four isolated vertices then either $G_x
  \simeq K_4 + 2 \overline{K_3}$ or $G_x \simeq K_4 +
  \overline{C_6}$. In the former case we obtain $G_x \geq K_7$ and in
  the latter case $G_x \supset K_7$. If $\overline{G_x}$ has exactly
  three isolated vertices, then either $G_x \simeq K_3 +
  \overline{C_3} + \overline{C_4}$ or $G_x \simeq K_3 +
  \overline{C_7}$. If $\overline{G_x}$ has exactly two isolated
  vertices, then $G_x$ is isomorphic to either $K_2 + \overline{K_3} +
  \overline{C_5}$, $K_2 + 2 \overline{C_4}$, or $K_2 +
  \overline{C_8}$. If $\overline{G_x}$ has exactly one isolated
  vertices, then $G_x$ is isomorphic to either $K_1 +
  3\overline{K_3}$, $K_1 + \overline{K_3} + \overline{C_6}$, $K_1 +
  \overline{C_4} + \overline{C_5}$, or $K_1 + \overline{C_9}$. If
    $\overline{G_x}$ has no isolated vertices, then $G_x$ is
    isomorphic to either $2 \overline{K_3} + \overline{C_4}$,
    $\overline{K_3} + \overline{C_7}$, $\overline{C_4} +
    \overline{C_6}$, $2 \overline{C_5}$, or $\overline{C_{10}}$. In
    each case it is easy to exhibit a $K_7$ minor in $G_x$, and so $G
    \geq K_8$.
\end{proof}
It may be true that if $G$ is a double-critical $8$-chromatic graph
with minimum degree $10$ and a vertex $x$ of degree $10$ such that
$G[N(x)]$ is $6$-regular then $G$ contains a $K_8$ minor. I was only
able to prove the desired result when $\overline{G[N(x)]}$ is not
isomorphic to any of the eight graphs $G_7, G_8, G_9, G_{12}, G_{13},
G_{16}, G_{17}$, and $G_{19}$ (see Appendix A). The graph denoted
$G_{17}$ is the Petersen graph. Given the symmetry of the Petersen
graph, it is particularly annoying not being able to settle the case
$\overline{G[N(x)]} \simeq G_{17}$.
\begin{problem}
  Prove that if $G$ is a double-critical $8$-chromatic graph with
  minimum degree $10$ and a vertex $x$ of degree $10$ such that
  $\overline{G_x}$ is the Petersen graph, then $G$ contains a $K_8$
  minor.
\end{problem}
\section{Minimum degree $10$ and $K_{8}^{-}$ minors}\label{sec:delta10contractionsToK_8^-}
In this section, we shall apply the following result of Mader.
\begin{theorem}[Mader~\cite{MR0229550}]\label{th:Mader} Every graph with minimum degree at least $5$ contains $K_ 6^-$ or
  the icosahedron graph as a minor. In particular, every graph with
  minimum degree at least $5$ and at most $11$ vertices contains a
  $K_6^-$ minor.
\end{theorem}
A proof of Theorem~\ref{th:Mader} may also be found
in~\citep[p. 373]{MR506522}.
\begin{proposition}\label{prop:importantProposition}
  Suppose $G$ is a double-critical $8$-chromatic graph with minimum
  degree $10$. If $G$ contains a vertex $x$ of degree $10$ such that
  $G_x$ contains at least one vertex of degree $9$ in $G_x$, then $G$
  contains a $K_{8}^{-}$ minor.
\end{proposition}
\begin{proof}
  Suppose $G$ is a double-critical $8$-chromatic graph with minimum
  degree $10$ such that a vertex, say $v$, has degree $9$ in
  $G_x$. According to Observation~\ref{obs:432785932465}, $\Delta
  (\overline{G_x}) \leq 3$ and so $\delta (G_x) = n(G_x) - 1 - \Delta(
  \overline{G_x}) \geq 6$. Thus, the graph $G_x - v$ has minimum
  degree at least $5$ and exactly $9$ vertices, and so it follows from
  Theorem~\ref{th:Mader} that $G_x - v$ contains a $K_6^-$ minor. Such
  a $K_6^-$ minor of $G_x - v$ along with the additional branch sets
  $\{ x \}$ and $\{ v \}$ constitute a $K_8^-$ minor of $G$.
\end{proof}
\begin{lemma}
  Suppose $G$ is a graph with a vertex $x$ of degree $10$ such that
  $\overline{G_x}$ is connected and cubic. Moreover, suppose that
  there is a vertex $z \in V(G) \setminus N_G[x]$ such that $G$
  contains at least six internally vertex-disjoint $(x,z)$-paths. Then
  $G$ contains a $K_8^-$ minor.\label{lem:cubic}
\end{lemma}
\begin{proof}
  Suppose $G$ is a $6$-connected graph with a vertex $x$ of degree
  $10$, where $\overline{G_x}$ is a connected cubic graph. There are
  exactly $21$ non-isomorphic cubic graphs of order $10$, see, for
  instance,~\cite{MR1692656}. These $21$ non-isomorphic cubic graphs
  of order $10$ are depicted in Appendix A; let these graphs be
  denoted as in Appendix A. If $\overline{G_x} \simeq G_i$, where $i
  \in [19] \setminus \{7, 8, 9, 12, 17 \}$, then the labelling of the
  vertices of the graph $G_i$ indicates how $\overline{G_i}$ may be
  contracted to $K_{7}^{-}$ or $K_7$. The vertices labelled $j \in
  [7]$ constitute the $j$th branch set of a $K_{7}^{-}$ minor or $K_7$
  minor. If the branch sets only constitute a $K_{7}^{-}$ minor, then
  it is because there is no edge between the branch sets of vertices
  labelled $1$ and $7$, respectively.
\begin{figure}[htbp!]
  \begin{center}
\mbox{
\subfigure[The graph $G_7$.]{\scalebox{0.26}{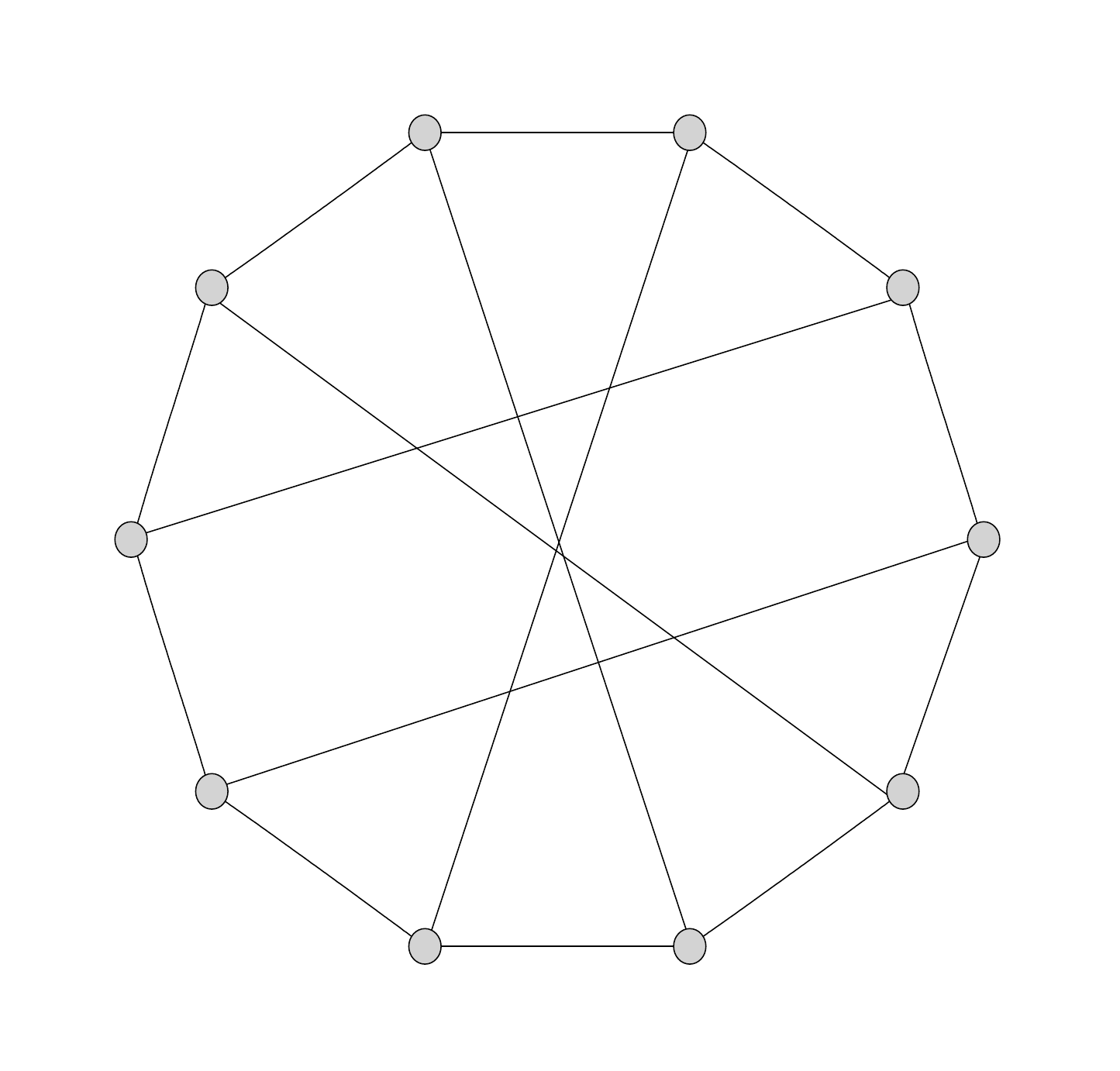}}
\subfigure[The graph $G_7'$.]{\scalebox{0.26}{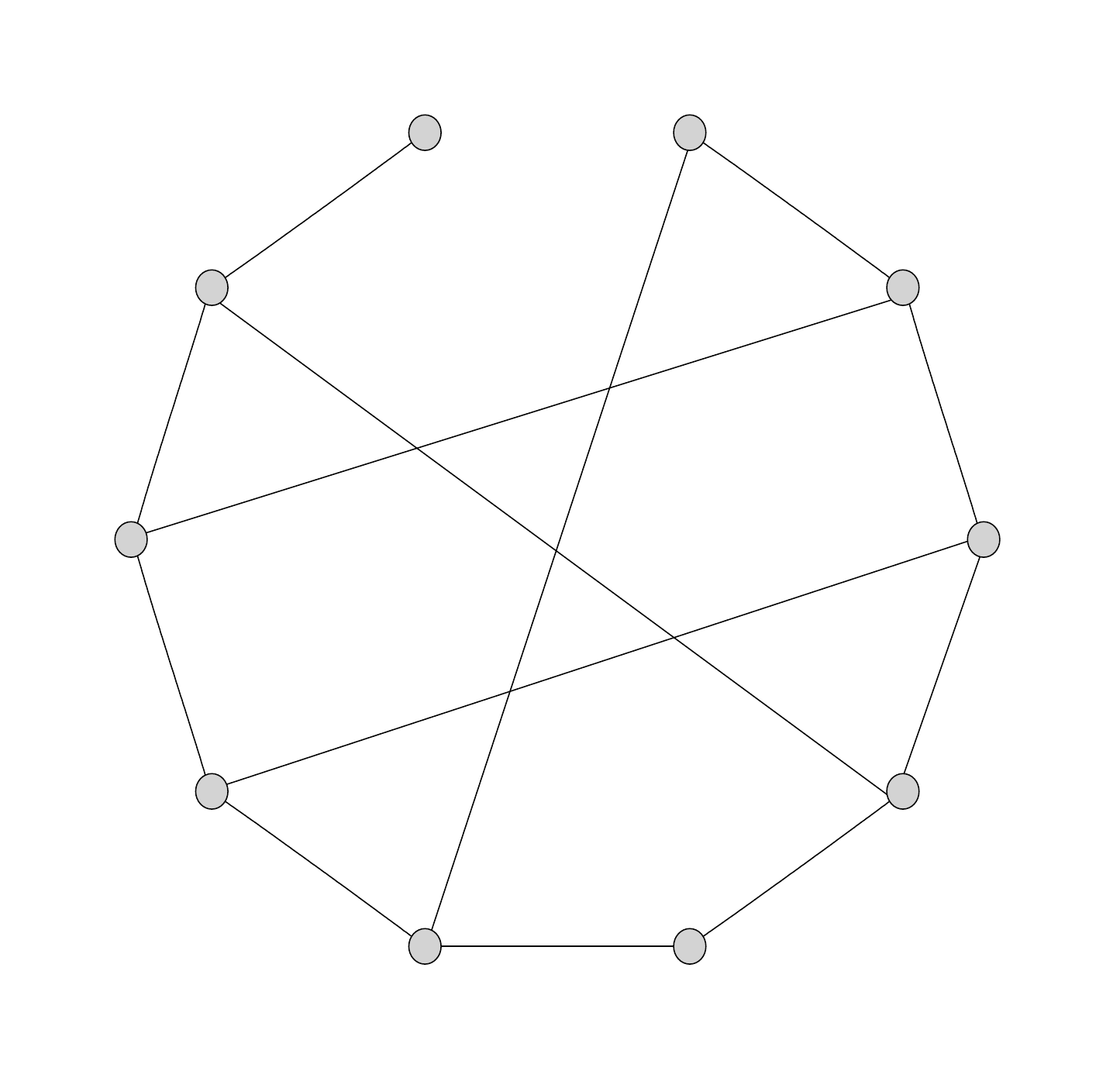}}  
\subfigure[The graph $G_8$.]{\scalebox{0.26}{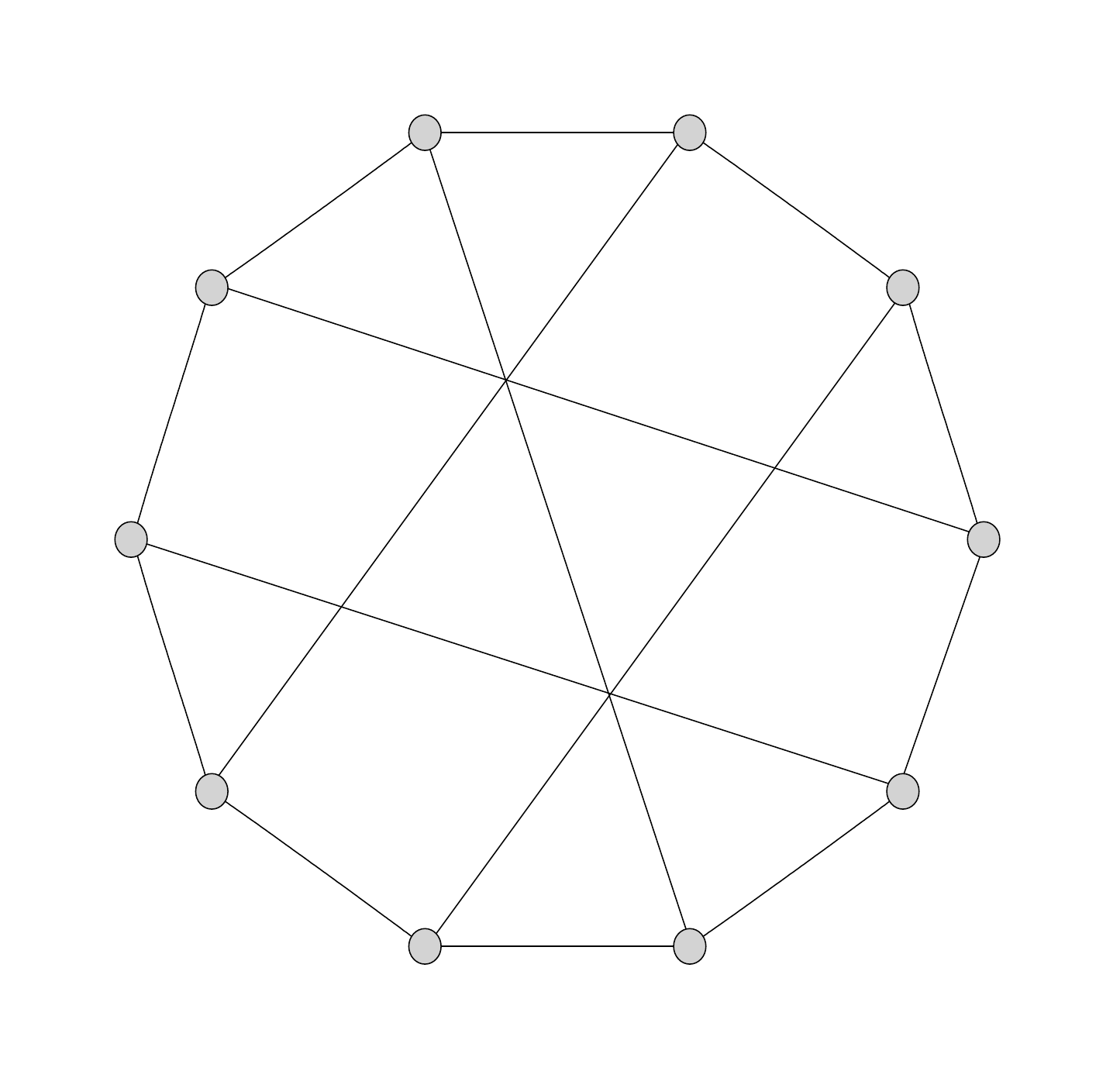}} }
\caption{The graphs $G_7$, $G_7'$, and $G_8$, which occur in the cases (i) and (ii) in the proof of Lemma~\ref{lem:cubic}.} 
\label{fig:G7G8}
 \end{center}
\end{figure}
In order to handle the cases $\overline{G_x} \simeq G_i$, where $i
\in \{7, 8, 9, 12, 17 \}$, we use the assumption that $V(G)
\setminus N_G[x]$ contains a vertex $z$ such that $G$ has a collection
$\mc{R}$ of at least six internally vertex-disjoint $(x,z)$-paths.
\begin{itemize}
\item[(i)] Suppose $\overline{G_x} \simeq G_7$ with the vertices of
  $\overline{G_x}$ labelled as shown in Figure~\ref{fig:G7G8}~(a). Let
  $\mc{S}$ denote the collection of the five $2$-sets $\{ v_1, v_6
  \}$, $\{ v_2, v_7 \}$, $\{ v_3, v_9 \}$, $\{ v_4, v_8 \}$ and $\{
  v_5, v_{10} \}$. Since the $2$-sets in $\mc{S}$ are pairwise
  disjoint and cover $N_G(x)$, it follows from the pigeonhole
  principle that at least two of the internally vertex-disjoint
  $(x,z)$-paths, say $Q_1$ and $Q_2$, of $\mc{R}$ go through the
  same $2$-set $S \in \mc{S}$. If $S = \{ v_i, v_j \} \in \mc{S}
  \setminus \{ \{ v_1, v_6 \} \}$, then, by contracting the $(v_i,
  v_j)$-path $( Q_1 \cup Q_2) - x$ into the edge $v_iv_j$, we obtain a
  graph which, as is readily verifiable, has a $K_7^-$ minor in the
  neighbourhood of $x$ and so $G \geq K_8^-$. Hence, we may assume
  that $\mc{R}$ contains no such two paths going through the same
  $2$-set of $\mc{S} \setminus \{ \{ v_1, v_6 \} \}$. Hence $S = \{ v_1,
  v_6 \}$ with say $Q_1$ and $Q_2$ going through $v_1$ and $v_6$,
  respectively. Since $|\mc{R}| \geq 6$, there is precisely one path
  going through each of the sets $S' \in \mc{S} \setminus \{ \{v_1, v_6
  \} \}$. By symmetry of $\overline{G_x}$, we may assume that there is an
  $(x,z)$-path $Q_3 \in \mc{R}$ going through the vertex $v_2$ of
  $N_G(x)$. Now, by contracting the $(v_2, z)$-path $Q_3 - x$ and the
  $(v_6, z)$-path $Q_2 - x$ into two edges, and then contracting the
  $(v_1,z)$-path $Q_1 - x$ into one vertex, we obtain a graph $G'$ in
  which the neighbourhood graph $G'[N_G(x)]$ of $x$ contains the
  complement of the $G_7'$, depicted in Figure~\ref{fig:G7G8}~(b), as
  a subgraph. The branch sets $\{ v'_1 \}$, $\{ v'_2 \}$, $\{ v'_3,
  v'_5 \}$, $\{ v'_4, v'_9 \}$, $\{ v'_6 \}$, $\{ v'_7, v'_{10} \}$,
  $\{ v'_8 \}$ constitute a $K_7^-$ minor in $\overline{G_7'}$ (there may
  be no edge between the branch sets $\{ v'_8 \}$ and $\{ v'_4, v'_9
  \}$), and so $G \geq K_8^-$.
\item[(ii)] Suppose $\overline{G_x} \simeq G_8$ with the vertices of
  $\overline{G_x}$ labelled as shown in Figure~\ref{fig:G7G8}~(c). In
  this case we contract a path $(P \cup Q) - x$, where $P,Q \in
  \mc{R}$, into an edge $e \in \{ v_1v_6, v_2v_8, v_3v_7, v_4v_{10},
  v_5 v_9 \}$, which is missing in $G_x$. By the symmetry of $G_x$, we
  need only consider the cases $e = v_1 v_6$ and $e = v_2 v_8$. If $e
  = v_1 v_6$, then the branch sets $\{ v_1, v_5 \}$, $\{ v_2 \}$, $\{
  v_3, v_9 \}$, $\{ v_4, v_7 \}$, $\{ v_6 \}$, $\{ v_8 \}$, and $\{
  v_{10} \}$ constitute a $K_7^-$ minor in the neighbourhood of
  $x$. If $e=v_2 v_8$, then the branch sets $\{ v_1, v_9 \}$, $\{ v_2
  \}$, $\{ v_3, v_6 \}$, $\{ v_4, v_7 \}$, $\{ v_5 \}$, $\{ v_8 \}$,
  and $\{ v_{10} \}$ constitute a $K_7^-$ minor in the neighbourhood
  of $x$. In both cases we obtain $G \geq K_8^-$.
\item[(iii)] Suppose $\overline{G_x} \simeq G_9$ with the vertices of
  $\overline{G_x}$ labelled as shown in Figure~\ref{fig:G9G12G17}~(a). Just as
  in case (ii), we contract a path $(P\cup Q) - x$, where $P,Q \in
  \mc{R}$, into an edge $e \in \{ v_1v_6, v_2v_{10}, v_3v_7, v_4v_8,
  v_5 v_9 \}$. By the symmetry of $G_x$, we need only consider $e \in
  \{ v_1 v_6, v_2 v_{10}, v_3 v_7, v_4 v_8 \}$. If $e = v_1 v_6$, then
  the branch sets $\{ v_1 \}$, $\{ v_2, v_5 \}$, $\{ v_3 \}$, $\{ v_4,
  v_9 \}$, $\{ v_6 \}$, $\{ v_7, v_{10} \}$, and $\{ v_8 \}$ constitute a
  $K_7^-$ minor in the neighbourhood of $x$. If $e = v_2 v_{10}$, then
  the branch sets $\{ v_1, v_8 \}$, $\{ v_2 \}$, $\{ v_3, v_5 \}$, $\{
  v_4 \}$, $\{ v_6, v_9 \}$ $\{ v_7 \}$, and $\{ v_{10} \}$ constitute a
  $K_7^-$ minor in the neighbourhood of $x$. If $e = v_3 v_7$, then
  the branch sets $\{ v_1, v_8 \}$, $\{ v_2, v_6 \}$, $\{ v_3 \}$, $\{
  v_4, v_{10} \}$, $\{ v_5 \}$, $\{ v_7 \}$, and $\{ v_9 \}$ constitute a
  $K_7^-$ minor in the neighbourhood of $x$. If $e = v_4 v_8$, then
  the branch sets $\{ v_1 \}$, $\{ v_2, v_5 \}$, $\{ v_3, v_9 \}$, $\{
  v_4 \}$, $\{ v_6 \}$, $\{ v_7, v_{10} \}$, and $\{ v_8 \}$ constitute a
  $K_7^-$ minor in the neighbourhood of $x$. In each case we obtain $G
  \geq K_8^-$.
\item[(iv)] Suppose $\overline{G_x} \simeq G_{12}$ with the vertices
  of $\overline{G_x}$ labelled as in Figure~\ref{fig:G9G12G17}~(b). Again, we
  contract a path $(P \cup Q)-x$, where $P,Q \in \mc{R}$, into an edge $e \in \{
  v_1v_6, v_2v_4, v_3v_7, v_5v_9, v_8v_{10} \}$. By the symmetry of $G_x$, we need only consider the cases $e
  \in \{ v_1 v_6, v_2 v_4, v_3 v_7 \}$. If $e = v_1 v_6$, then the
  branch sets $\{ v_1 \}$, $\{ v_2, v_7 \}$, $\{ v_3 \}$, $\{ v_4,
  v_{10} \}$, $\{ v_5, v_9 \}$, $\{ v_6 \}$, and $\{ v_8 \}$ constitute a
  $K_7^-$ minor in the neighbourhood of $x$. If $e = v_2 v_4$, then
  the branch sets $\{ v_1, v_5 \}$, $\{ v_2 \}$, $\{ v_3, v_8 \}$, $\{
  v_4 \}$, $\{ v_6 \}$, $\{ v_7, v_{10} \}$, and $\{ v_9 \}$ constitute a
  $K_7^-$ minor in the neighbourhood of $x$. If $e = v_3 v_7$, then
  the branch sets $\{ v_1, v_9 \}$, $\{ v_2, v_6 \}$, $\{ v_3 \}$, $\{
  v_4, v_8 \}$, $\{ v_5 \}$, $\{ v_7 \}$, and $\{ v_{10} \}$ constitute a
  $K_7^-$ minor in the neighbourhood of $x$. In each case we obtain $G
  \geq K_8^-$.
\item[(v)] Suppose $\overline{G_x} \simeq G_{17}$ with the vertices of
  $\overline{G_x}$ labelled as shown in
  Figure~\ref{fig:G9G12G17}~(c). The graph $G_{17}$ is the Petersen
  graph, and the complement of the Petersen graph does not contain a
  $K_7$ minor. However, we may repeat the trick used in the previous
  cases to obtain a $K_7^-$ minor. We contract a path $(P \cup Q)-x$,
  where $P,Q \in \mc{R}$, into an edge $e \in \{ v_iv_{i+5} \mid i \in
  [5] \}$. By the symmetry of $G_x$, we may assume $e=v_1v_6$. Now,
  the branch sets $\{ v_1 \}$, $\{ v_2, v_8 \}$, $\{ v_3 \}$, $\{ v_4,
  v_{10} \}$, $\{ v_5, v_9 \}$, $\{ v_6 \}$, and $\{ v_7 \}$
  constitute a $K_7^-$ minor in the neighbourhood of $x$. Thus, $G$
  contains a $K_8^-$ minor.
\end{itemize}
This completes the proof.
\end{proof}

\begin{figure}[htbp!]
  \begin{center}
\mbox{
\subfigure[The graph $G_9$.]{\scalebox{0.26}{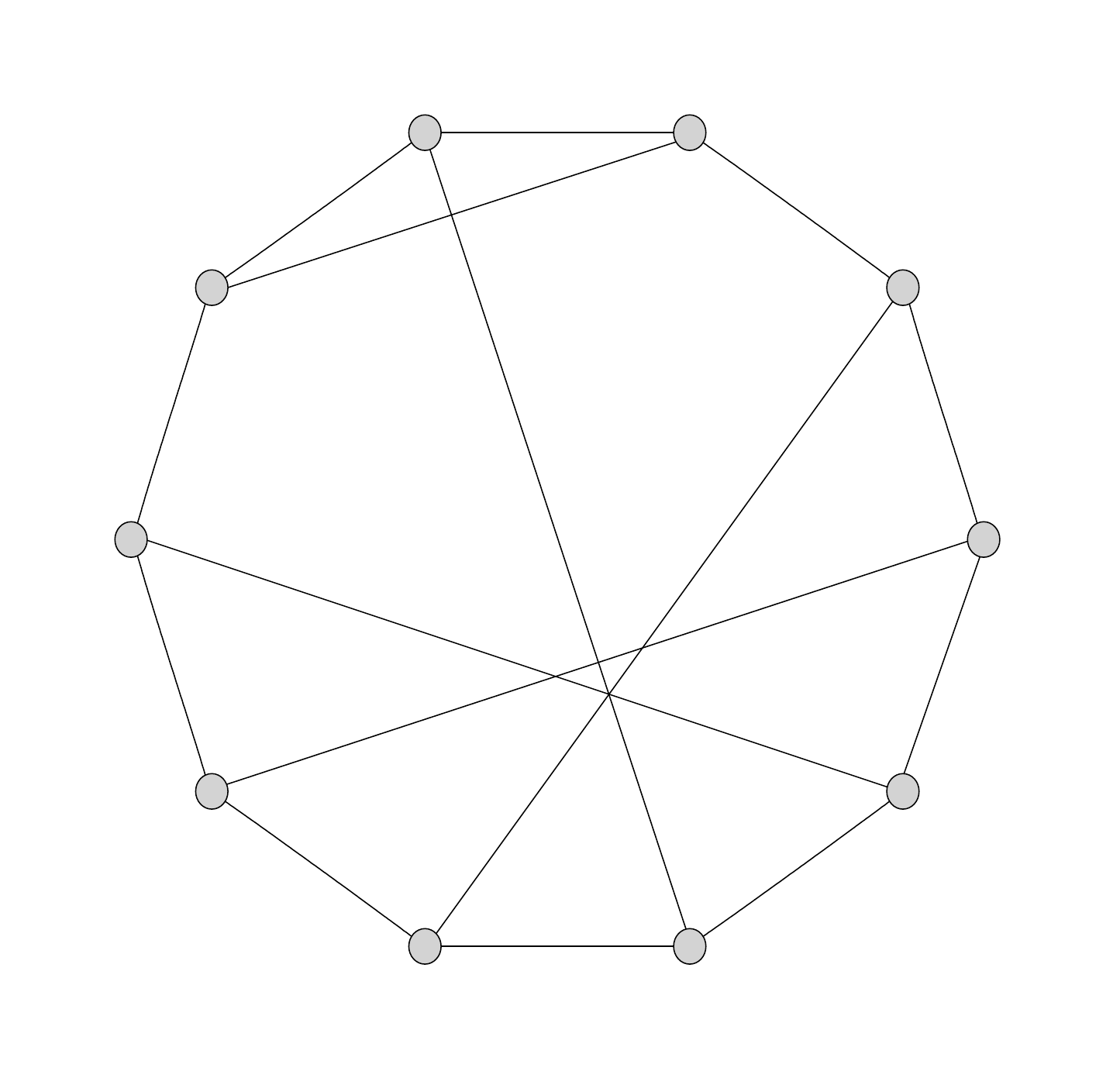}}
\subfigure[The graph $G_{12}$.]{\scalebox{0.26}{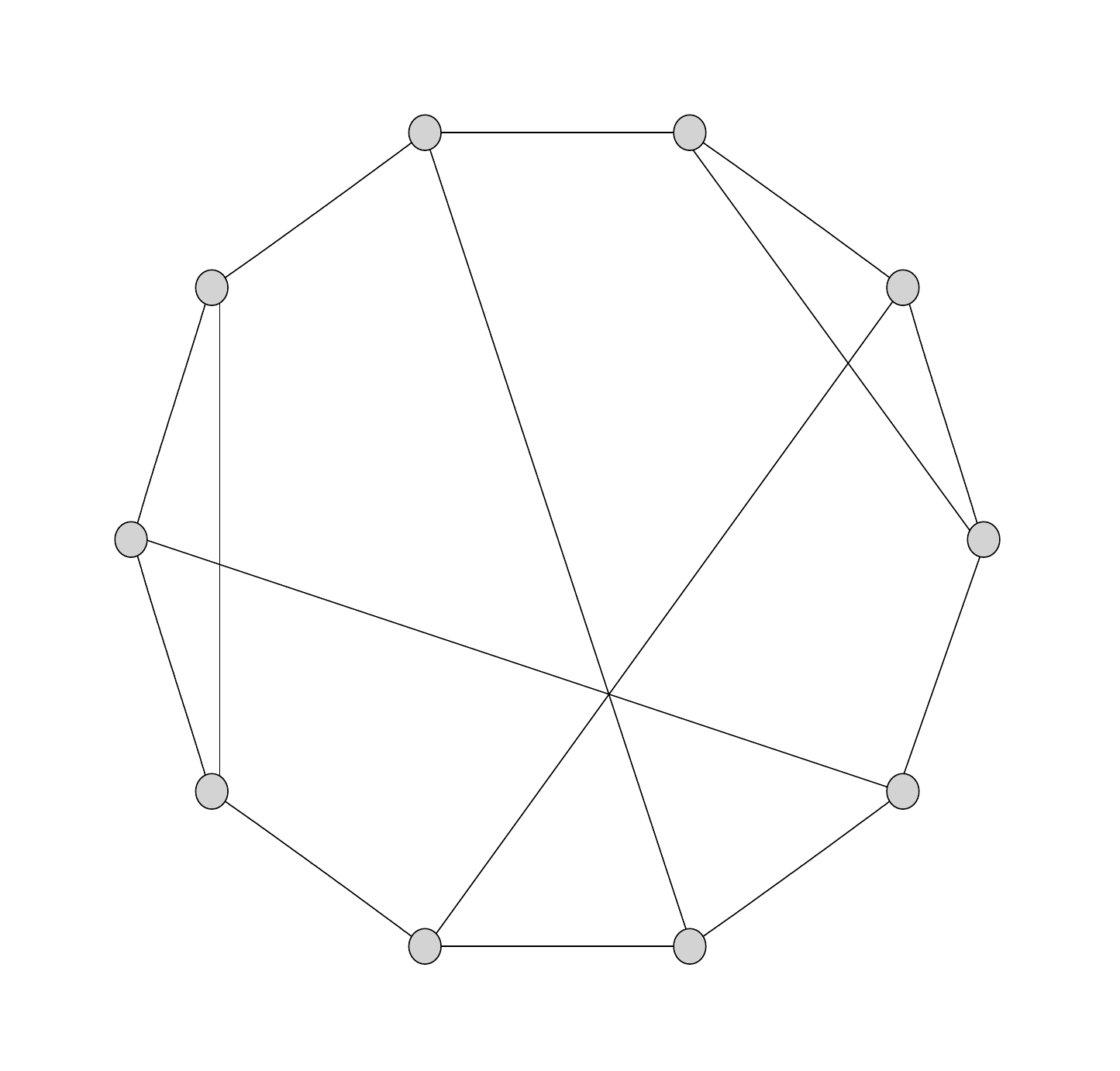}}
\subfigure[The graph $G_{17}$.]{\scalebox{0.26}{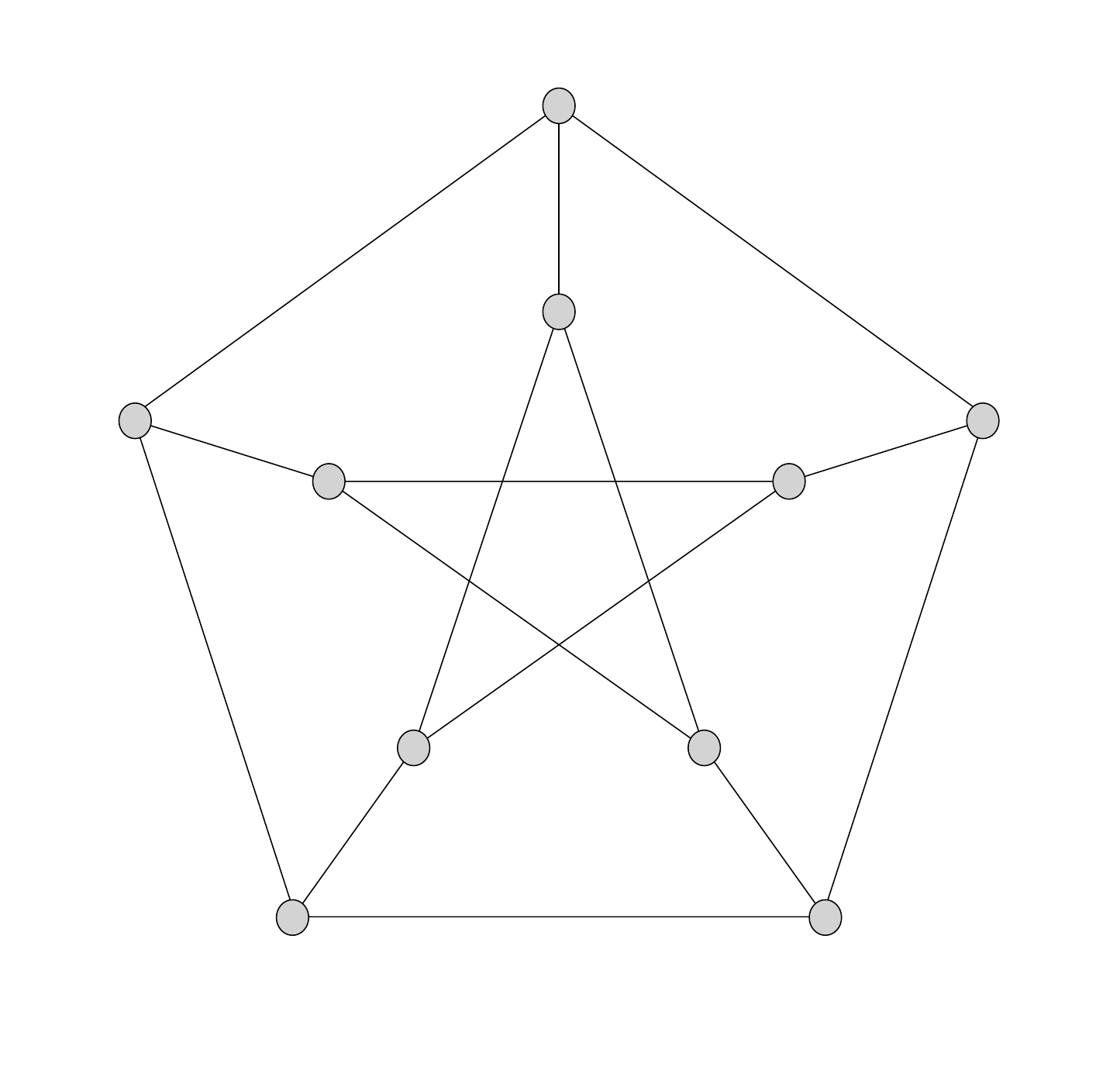}}}
\caption{The graphs $G_9$, $G_{12}$, and $G_{17}$, which occur in the
  cases (iii), (iv), and (v) in the proof of
  Lemma~\ref{lem:cubic}.}
\label{fig:G9G12G17}
 \end{center}
\end{figure}
Notice that in each of the cases (i-v) in the proof of
Lemma~\ref{lem:cubic} we used the regularity of $G_x$ and the six
internally vertex-disjoint $(x,z)$-paths of $G$, but we did not assume
$G$ to be double-critical. It may be possible to relax the assumptions
of Lemma~\ref{lem:cubic} and still maintain the conclusion. It may
even be that Lemma~\ref{lem:cubic} follows from an earlier result
similar in spirit to that of Theorem~\ref{th:Mader}.
\begin{proposition}\label{prop:importantProposition2}
  Suppose $G$ is a double-critical $8$-chromatic graph with minimum
  degree $10$. If $G$ contains a vertex $x$ of degree $10$ such that
  $G_x$ contains no vertex of degree $9$ in $G_x$, then $G$ contains a
  $K_{8}^{-}$ minor.
\end{proposition}
\begin{proof}
  Suppose $G$ is a double-critical $8$-chromatic graph with minimum
  degree $10$, and suppose $G$ contains a vertex $x$ of degree $10$
  such that $G_x$ contains no vertex of degree $9$ in $G_x$. Then it
  follows from Proposition~\ref{prop:noIsolatedInAxy} and
  Observation~\ref{obs:432785932465} that each vertex of
  $\overline{G_x}$ has degree $2$ or $3$.

  We first consider the case where $\overline{G_x}$ is
  disconnected. Since $\delta(\overline{G_x}) \geq 2$, it follows that
  any component of $\overline{G_x}$ contains at least three vertices.
  If $\overline{G_x}$ contains a component on three vertices, then
  this component is a $K_3$; this contradicts
  Proposition~\ref{prop:noIsolatedInAxy}. Hence, each component of
  $\overline{G_x}$ contains at least four vertices, and so, since
  $n(G_x)=10$, it follows that $\overline{G_x}$ contains precisely two
  components, say $D_1$ and $D_2$ with $n(D_1) \leq n(D_2)$. Suppose
  $n(D_1)=4$. The fact that $\delta(\overline{G_x}) \geq 2$ implies
  that $D_1$ must contain a $4$-cycle, and so it is easy to see that
  $D_1$ must be $C_4$, $K_{4}^{-}$ or $K_4$. This, however,
  contradicts Proposition~\ref{prop:noIsolatedInAxy}, and so we must
  have $n(D_1) = n(D_2) = 5$. Of course, if $G'$ is a subgraph of $G$,
  and $G'$ contains an $H$ minor, then $G$ contains an $H$
  minor. Thus, it suffices to consider the case where both $D_1$ and
  $D_2$ contain exactly one vertex of degree $2$, in which case both
  $D_1$ and $D_2$ is isomorphic to $K_4$ with exactly one edge
  subdivided. In this case it is very easy to find a $K_7$ minor in
  $G_x$.

  Suppose that $\overline{G_x}$ is connected, and let $D$ denote
  $\overline{G_x}$. By Proposition~\ref{prop:atLeast15vertices}, we
  may assume there is a vertex $z \in V(G) \setminus N_G[x]$, and, by
  Proposition~\ref{prop:doubleCriticalImpliesSixConnected}~(iii), there
  are six internally vertex-disjoint $(x,z)$-paths in $G$. If $D$ is
  cubic, then, according to Lemma~\ref{lem:cubic}, $G \geq
  K_8^-$. Suppose that $D$ is not cubic. We add edges (possibly none!)
  between non-adjacent $2$-vertices to $D$ to obtain $D'$, which
  contains no two non-adjacent $2$-vertices. If $D'$ is cubic, then
  $G' := G \setminus (E(D')\setminus E(D))$ satisfies the assumption
  of Lemma~\ref{lem:cubic}.  (The graph $D'$ is connected, cubic
  $10$-graph and the graph $G'$ has six internally vertex-disjoint
  $(x,z)$-paths, since $G$ has six internally vertex-disjoint
  $(x,z)$-paths, and these may be chosen so that they do not contain
  any edge of $E(G[N_G(x)])$.)  Thus, $G' \geq K_8^-$, which implies
  that the supergraph $G$ of $G'$ has a $K_8^-$ minor.

  Now, suppose $D'$ is not cubic. The graph $D'$ contains no two
  non-adjacent $2$-vertices. Moreover, $D'$ is a connected $10$-graph
  in which each vertex has degree $2$ or $3$. Thus, since the number
  of odd degree vertices of any graph is even it follows that $D'$
  contains exactly two $2$-vertices and these must be
  neighbours. There are exactly 23 connected $10$-graphs each with two
  $2$-vertices and eight $3$-vertices, where the two $2$-vertices are
  adjacent\footnote{According to the computer program {\tt geng}
    developed by Brendan McKay~\cite{nauty}, there are 113 connected
    graphs of order $10$ each with two $2$-vertices and eight
    $3$-vertices -- among these graphs exactly 23 have the property
    that the two $2$-vertices are adjacent. This latter fact has been
    determined, independently, by inspection done by the author and by
    a computer program developed by Marco Chiarandini.}. These graphs,
  denoted $J_i$ ($i \in [23]$), are depicted in Appendix B. For each
  $i \in [23]$, the labelling of the vertices of the graph $J_i$
  indicates how $\overline{J_i}$ may be contracted to $K_{7}^{-}$ or,
  even, $K_7$; the vertices labelled $j \in [7]$ constitute the $j$th
  branch set of a $K_{7}^{-}$- or $K_7$ minor. If the branch sets only
  constitute a $K_{7}^{-}$ minor, then it is because there is no edge
  between the branch sets labelled $1$ and $7$. This completes the
  proof.
\end{proof}
\section{More open problems}
The Double-Critical Graph Conjecture is still open for $6$-chromatic
graphs. To settle this instance of the conjecture in the affirmative,
it would, by Proposition~\ref{prop:ElemPropertiesOfDC}~(i),
suffice to prove that any double-critical $6$-chromatic graph contains
$K_5$ as a subgraph; however, we cannot even prove that such a graph
contains $K_4$ as a subgraph.
\begin{problem}[Matthias Kriesell\footnote{Private
    communication to the author, Odense, September, 2008.}]
  Prove that every double-critical $6$-chromatic graph contains $K_4$
  as a subgraph. \label{prob:Kriesell}
\end{problem}
In~\cite{KawarabayashiPedersenToftEJC2010}, it was proved that every
double-critical $6$-chromatic graph contains a $K_6$ minor; a stronger
result would be that every double-critical $6$-chromatic graph
contains a subdivision of $K_6$.
\begin{problem}\label{prob:K6subdiv}
Prove that every double-critical $6$-chromatic graph $G$ contains
  a subdivision of $K_6$.
\end{problem}
According to Observation~\ref{obs:34978kfjdghs},
Problem~\ref{prob:K6subdiv} has a positive solution if $G$ has minimum degree at most $7$.

Mader~\cite{MR1722261} proved a longstanding conjecture, known as
Dirac's Conjecture, which states that any graph $G$ with at least
three vertices and at least $3n(G) - 5$ edges contains a subdivision
of $K_5$. Thus, in particular, any double-critical
\mbox{$6$-chromatic} graph $G$ contains a subdivision of $K_5$.
\begin{observation}\label{obs:34978kfjdghs}
  Any double-critical $6$-chromatic graph with minimum degree at most
  $7$ contains a subdivision of $K_6$.
\end{observation}
\begin{proposition}[\cite{KawarabayashiPedersenToftEJC2010}] \label{prop:342alkfh39485}
  If $G$ is a non-complete double-critical $6$-chromatic graph, then
  $G$ contains at least $12$ vertices.
\end{proposition}
\begin{proof}[Proof of Observation~\ref{obs:34978kfjdghs}.]
  Let $G$ denote any double-critical $6$-chromatic graph with minimum
  degree at most $7$. If $\delta(G) \leq 6$, then, by
  Proposition~\ref{prop:ElemPropertiesOfDC}~(i), $G \simeq
  K_6$. Hence $\delta(G) = 7$. Let $x$ denote a vertex of degree $7$
  in $G$. The graph $G$ is non-complete, and so, by
  Proposition~\ref{prop:342alkfh39485}, $n(G) \geq 12$, in particular,
  $G - N[x]$ is non-empty. Let $z$ denote a vertex of
  $G-N[x]$. According to Corollary 6.1
  in~\cite{KawarabayashiPedersenToftEJC2010}, $\overline{G_x}$ is a
  $7$-cycle $C_7$ with, say, $C_7 : v_1, v_2, v_3, \ldots, v_7$. By
  Proposition~\ref{prop:doubleCriticalImpliesSixConnected}~(iii), $G$
  is $6$-connected, and so there is a collection $\mc{C} = \{ Q_1,
  Q_2, \ldots, Q_6 \}$ of six internally vertex $(x,z)$-paths in
  $G$. Choose the paths such that the sum of the lengths of the paths
  is minimum. Then each of the paths $Q_i \in \mc{C}$ contains exactly
  one vertex of $N(x)$. By the symmetry of $G_x$, we may, without loss
  of generality, assume that $V(Q_i) \cap V(G_x) = \{ v_i \}$ for each
  $i \in [6]$. Thus, in $G$, there is a $K_6$-subdivision $H$ with
  branch vertices $v_1, v_2, v_4,v_5, x$ and $z$.
\begin{figure}
\begin{center}
\scalebox{0.5}{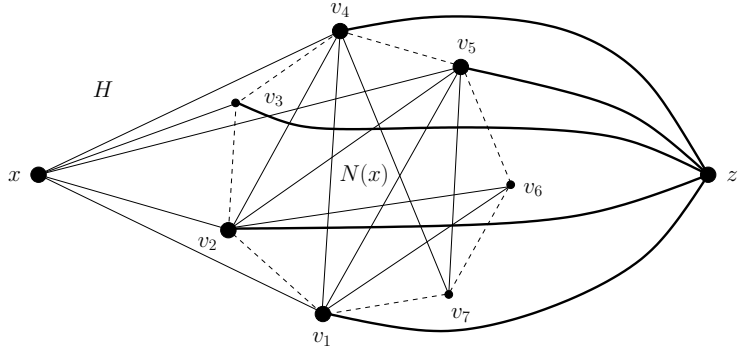}
\caption{The graph $H$ of $G$ is a subdivision of $K_6$. The six
  larger dots represent the branch vertices of $H$, while the
  smaller dots represent subdividing vertices. The filled straight
  lines represent edges in $H$, while the bold curves represent the
  paths $Q_1, Q_2, Q_3,Q_4$, and $Q_5$.}
\label{fig:K6subdivision}
\end{center}
\end{figure}
The paths in $H$ connecting the branch vertices of are as indicated in
Figure~\ref{fig:K6subdivision}. Note that the $(x,z)$-path in $H$ is
the union of the $(z,v_3)$-path $Q_3$ and the $(v_3, x)$-path $( \{
v_3, x \}, \{ v_3x \})$. Thus, $G$ contains a subdivision of $K_6$.
\end{proof}
The following conjecture, known as the \emph{$(k-1,1)$ Minor
  Conjecture}, is a well-known relaxed version of Hadwiger's
Conjecture.
\begin{conjecture}[Chartrand, Geller \& Hedetniemi~\cite{MR0285427};
  Woodall~\cite{MR1210067}] Every $k$-chromatic graph has either a
  $K_k$ minor or a $K_{\lfloor \frac{k+1}{2} \rfloor, \lceil
    \frac{k+1}{2} \rceil}$ minor.
\end{conjecture}
Kawarabayashi and Toft~\cite{MR2141662} proved that every
$7$-chromatic graph contains $K_7$ or $K_{4,4}$ as a minor -- thus,
settling the case $k=7$ of the $(k-1,1)$ Minor Conjecture. This result
has inspired the following problem.
\begin{problem}
Prove that every double-critical $8$-chromatic graph contains $K_8$ or $K_{4,5}$ as a minor.
\end{problem}
A natural generalisation of Problem~\ref{prob:Kriesell} would be to ask
for a linear function $f$ such that every double-critical $k$-chromatic
graph has a clique of order $f(k)$; if that problem is too hard it might be
worth considering the following problem.
\begin{problem}[Sergey Norin\footnote{Private
    communication to the author at Prague Midsummer Combinatorial Workshop XV,  July 27 - July 31, 2009.}]
  Prove that there a linear, strictly increasing function $f$ such that every
  double-critical $k$-chromatic graph has a complete minor of order
  $f(k)$.
\end{problem}
\section*{Acknowledgement}
I wish to thank Marco Chiarandini, Daniel Merkle, Friedrich Regen, and
Bjarne Toft for stimulating discussions on critical graphs and for
assistance in using certain computer programs, in particular, I must
thank Friedrich and Marco for developing certain computer programs for
sorting and displaying small graphs.
\section*{Appendix A}
This section contains drawings of all non-isomorphic cubic graphs
$G_i$ ($i \in [21]$) of order $10$ - the drawings are copies of
drawings found in~\cite{MR1849620}. Drawings of all non-isomorphic
cubic graphs of order at most $14$ be found in~\cite{MR1692656}.

For $i \in [19] \setminus \{7, 8, 9, 12, 17 \}$, the labelling of the
vertices of the graph $G_i$ indicates how $\overline{G_i}$ may be
contracted to $K_{7}^{-}$ or, even, $K_7$. The vertices labelled $j
\in [7]$ constitute the $j$th branch set of a $K_{7}^{-}$- or $K_7$
minor. If the branch sets only constitute a $K_{7}^{-}$ minor, then it
is because there is no edge between the branch sets of vertices
labelled $1$ and $7$, respectively.
\begin{center}
\scalebox{0.26}{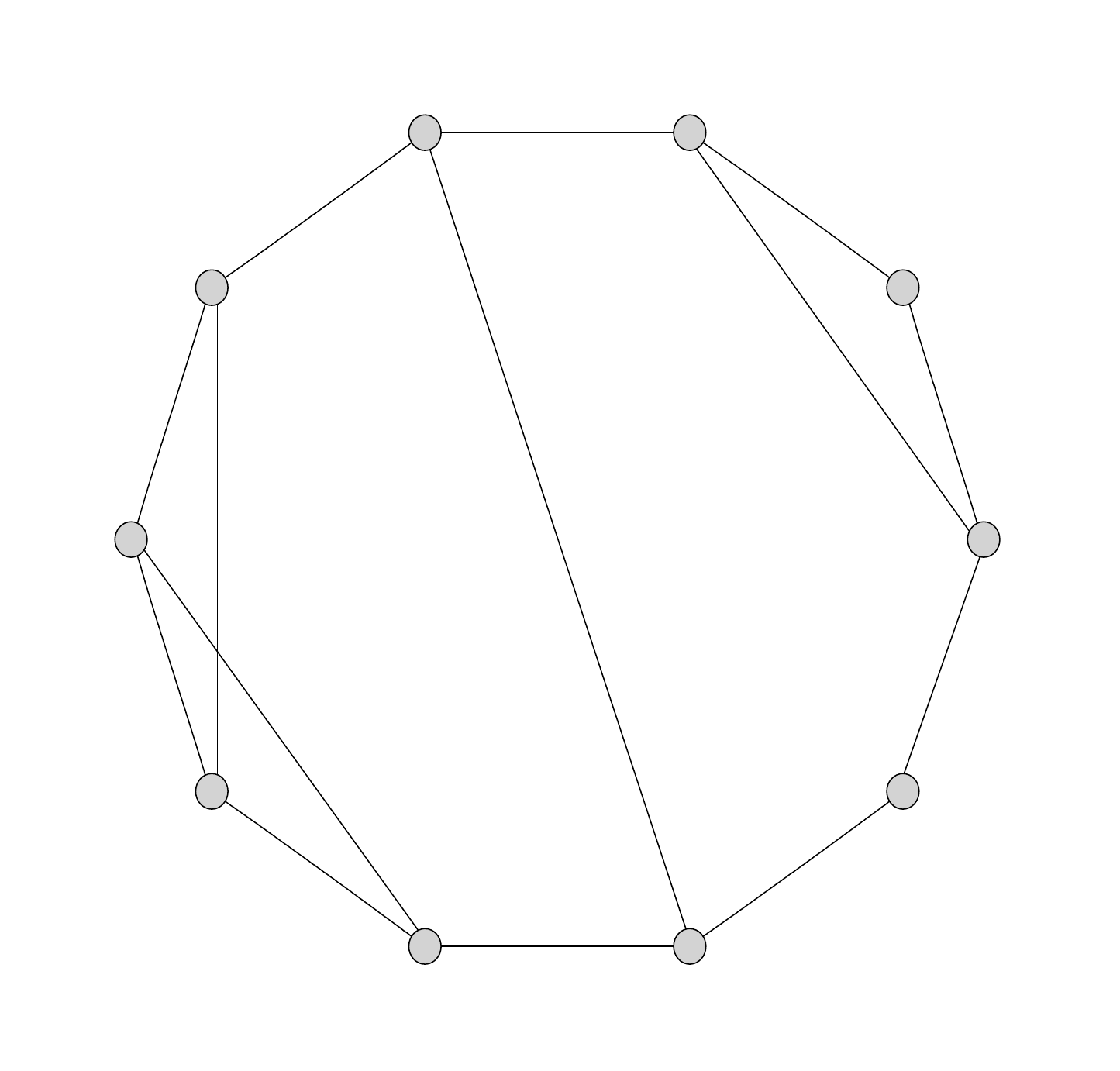}
\scalebox{0.26}{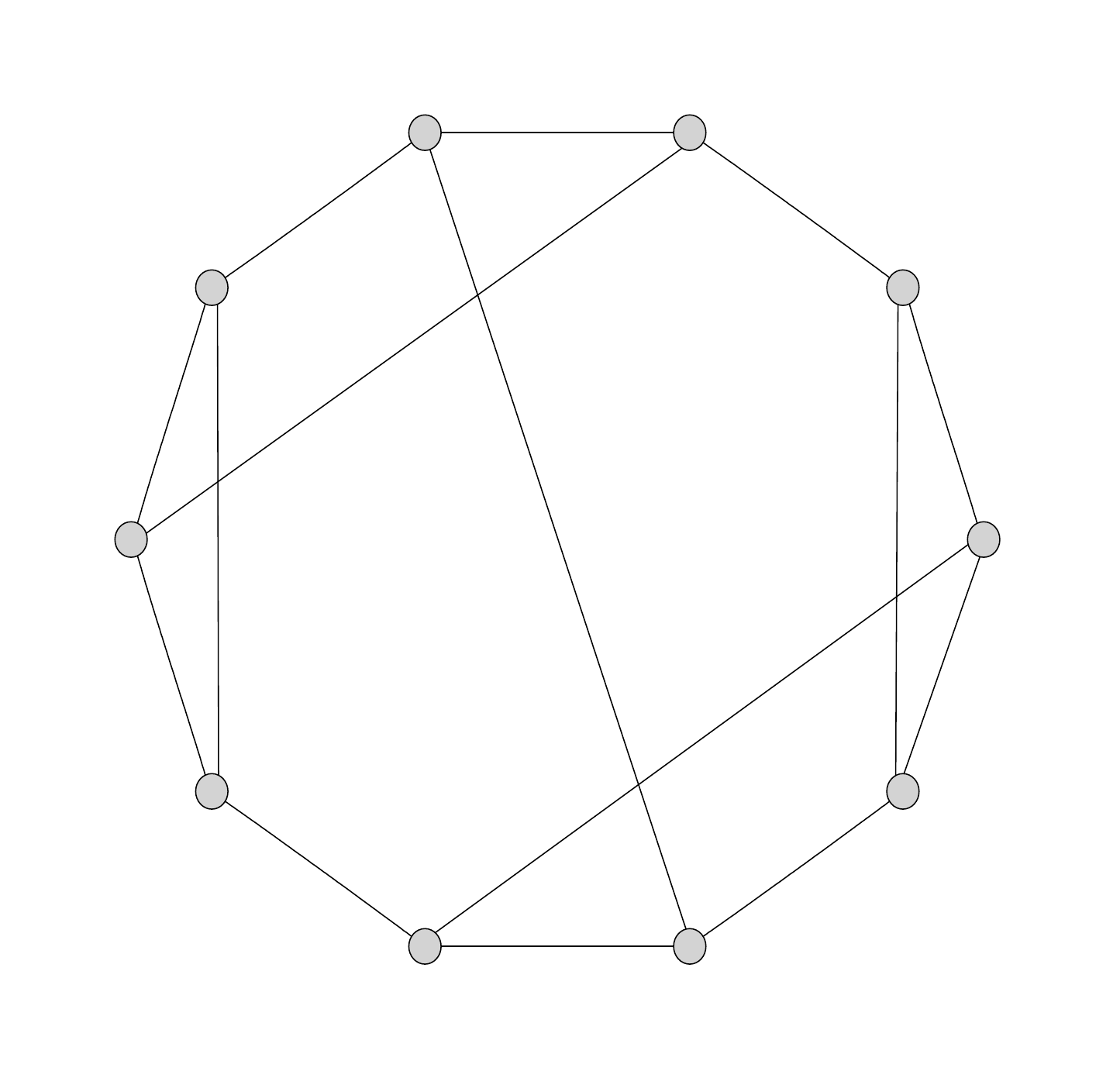}
\scalebox{0.26}{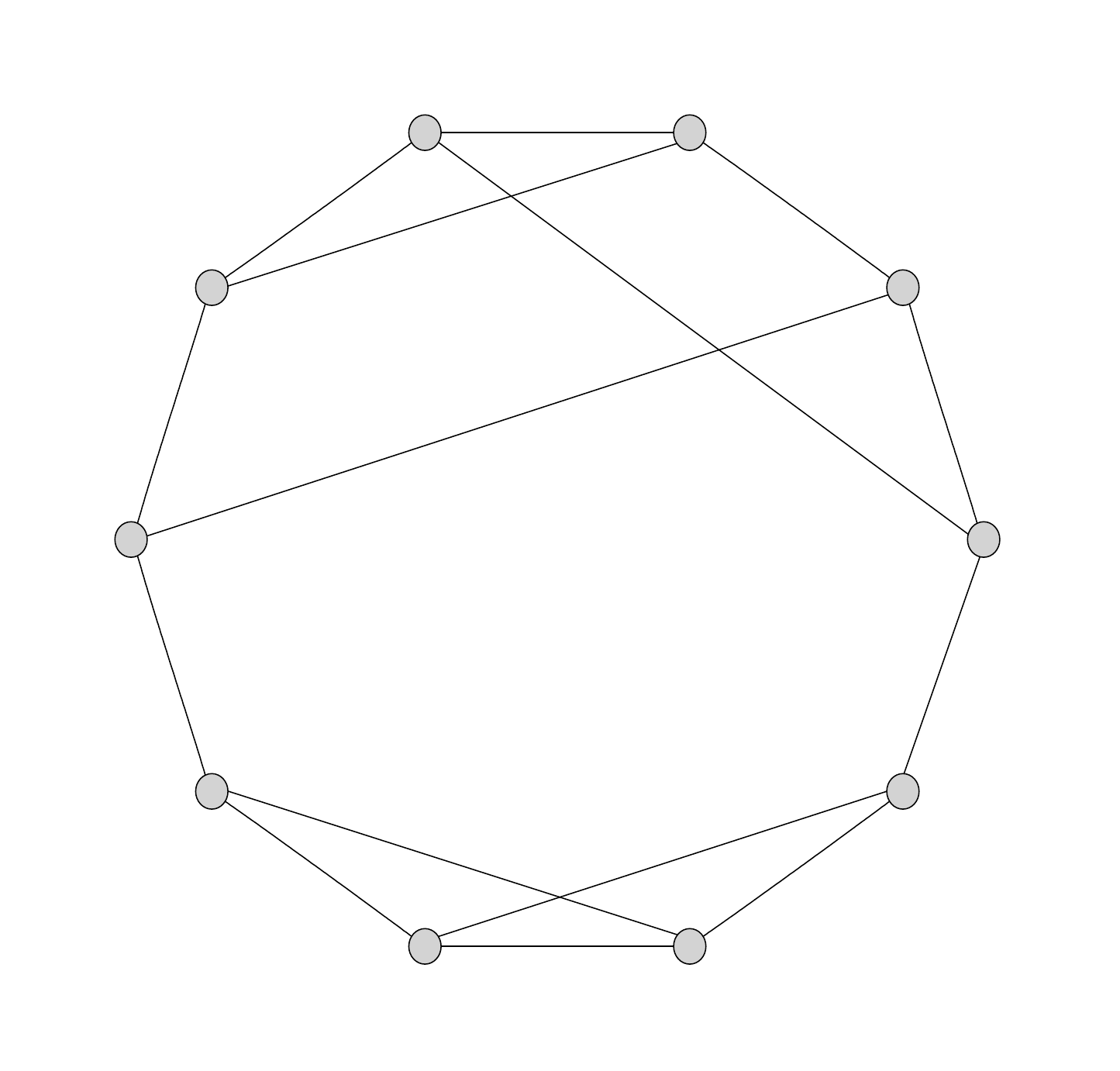} \\
 \scalebox{0.26}{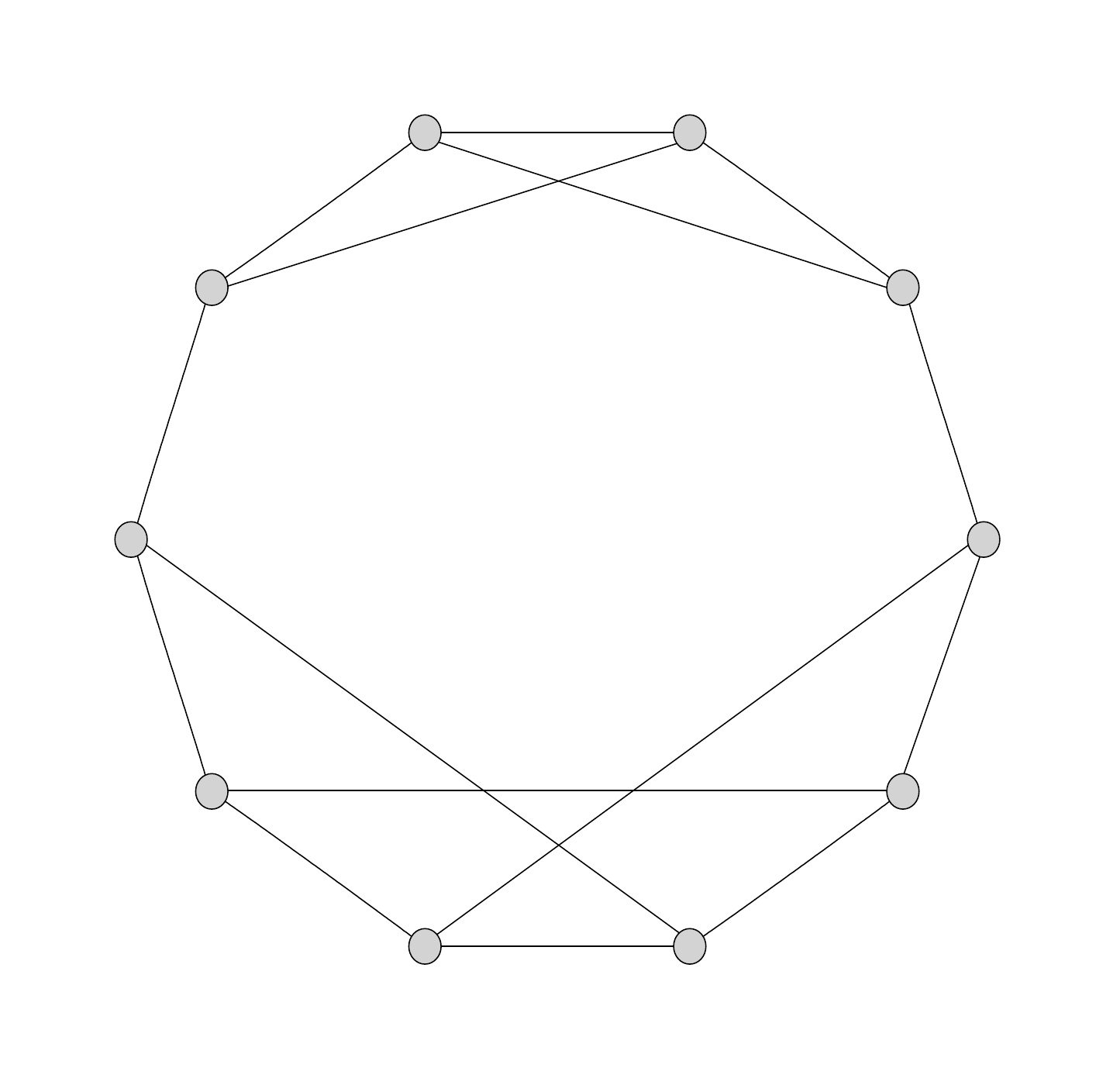}
 \scalebox{0.26}{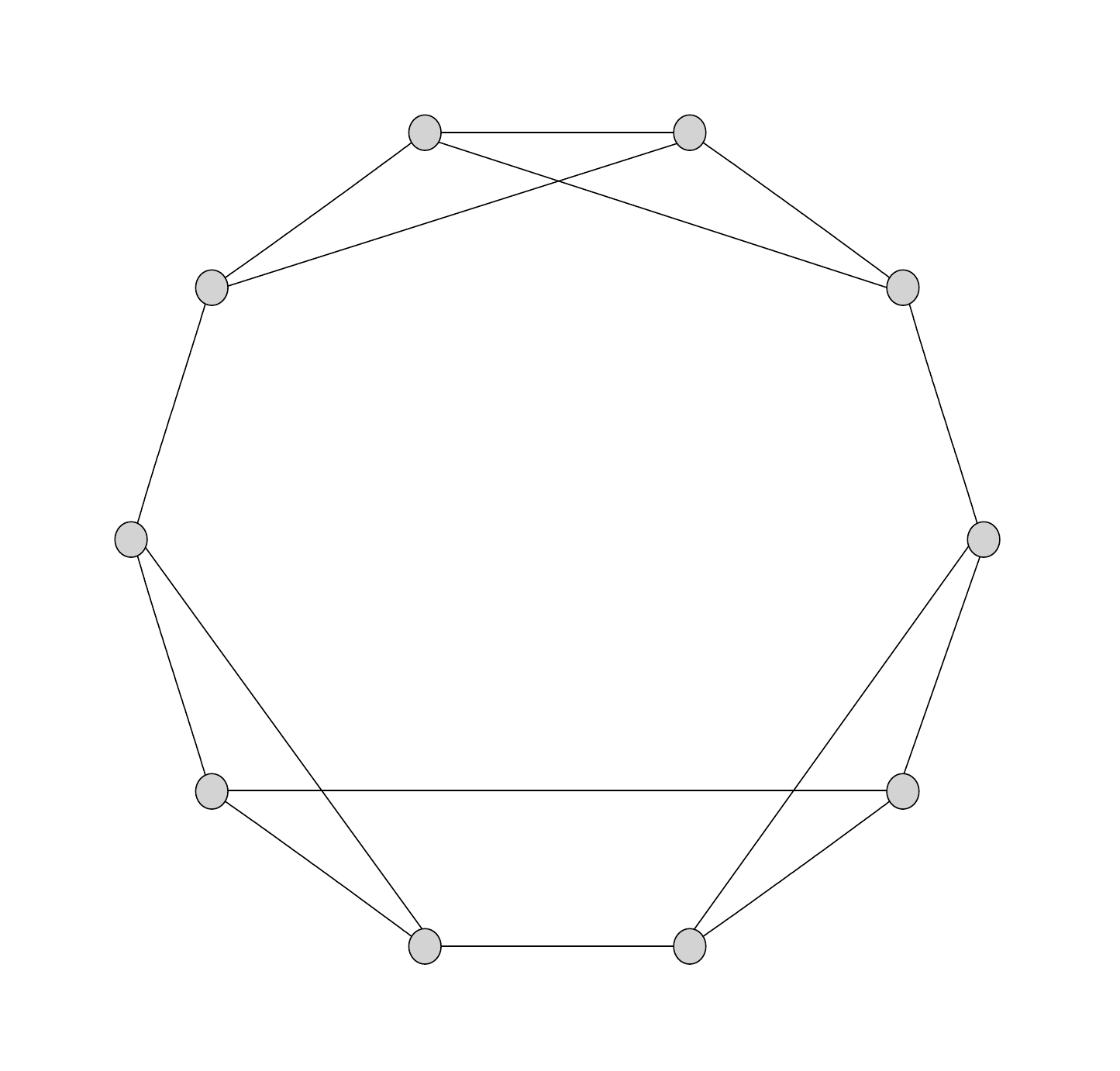}
 \scalebox{0.26}{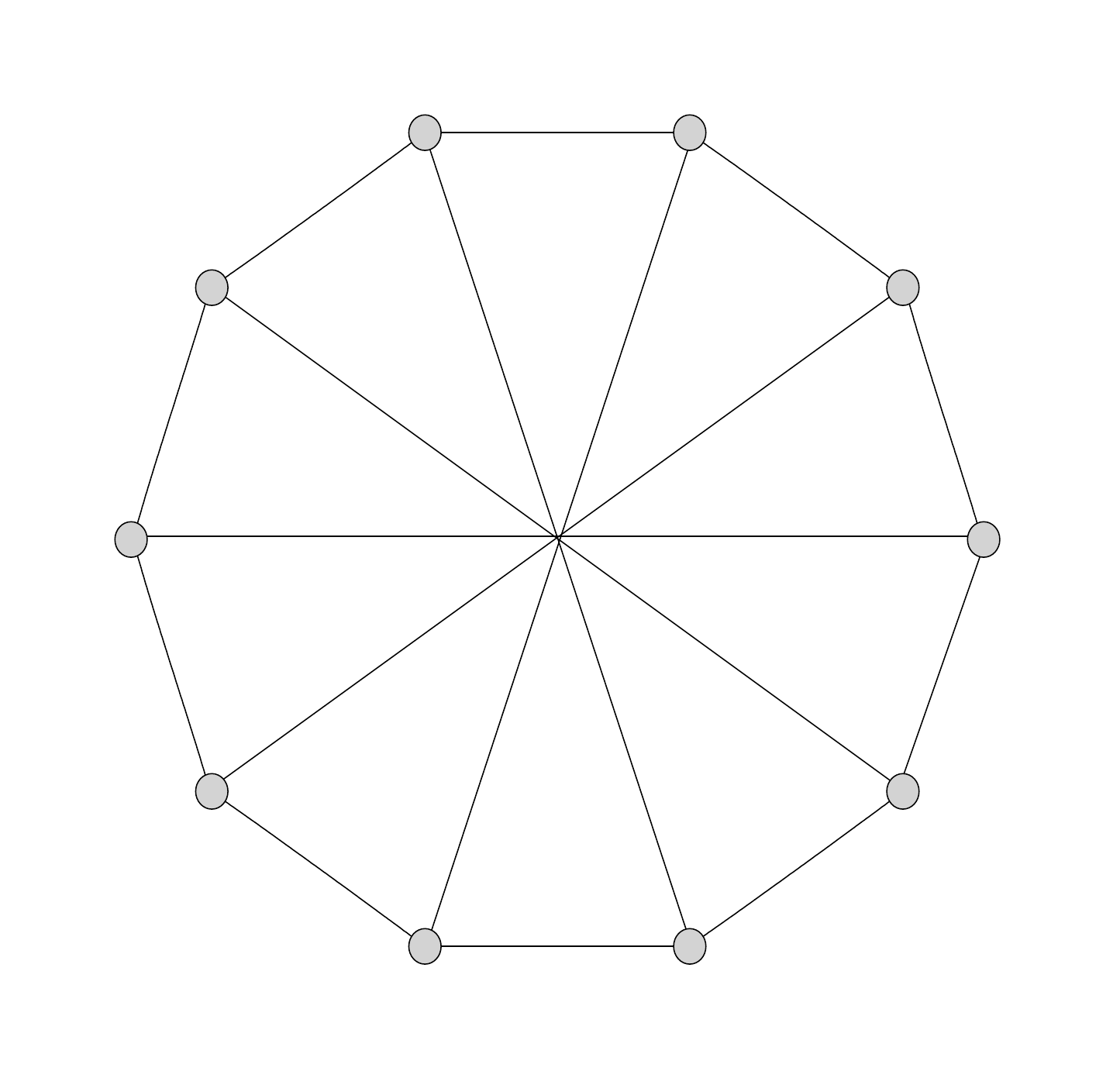} \\
\scalebox{0.26}{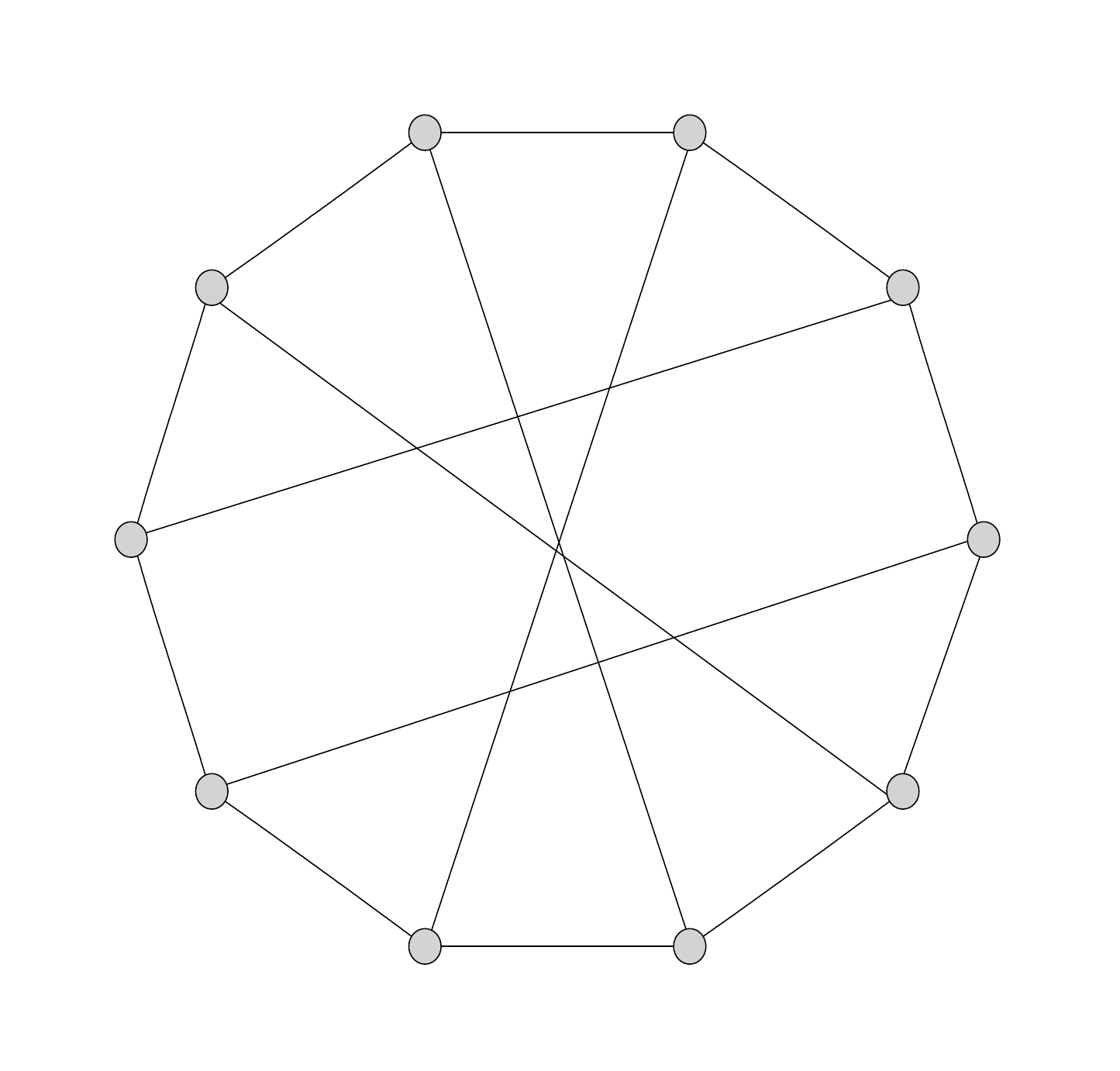}
 \scalebox{0.26}{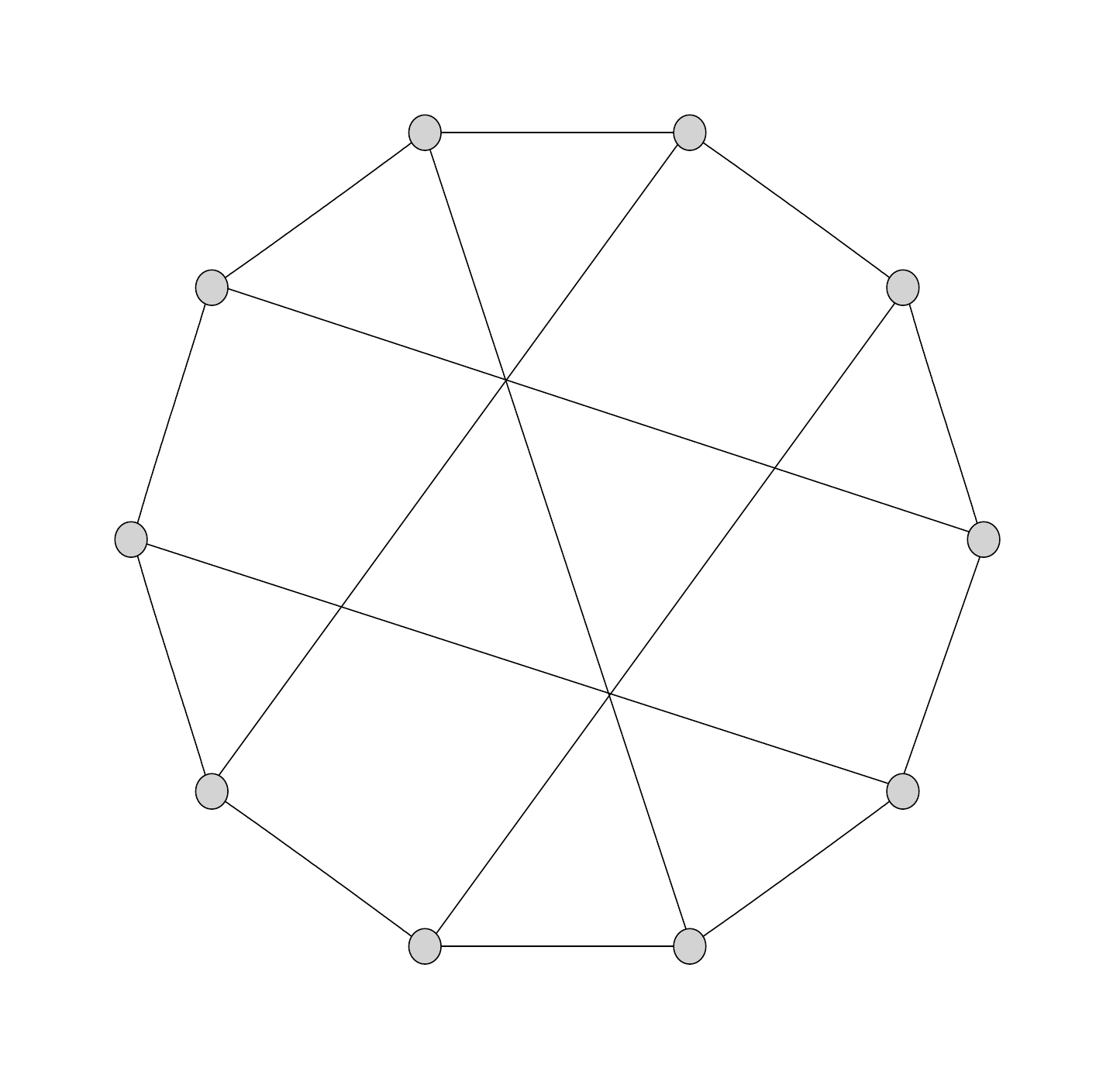}
 \scalebox{0.26}{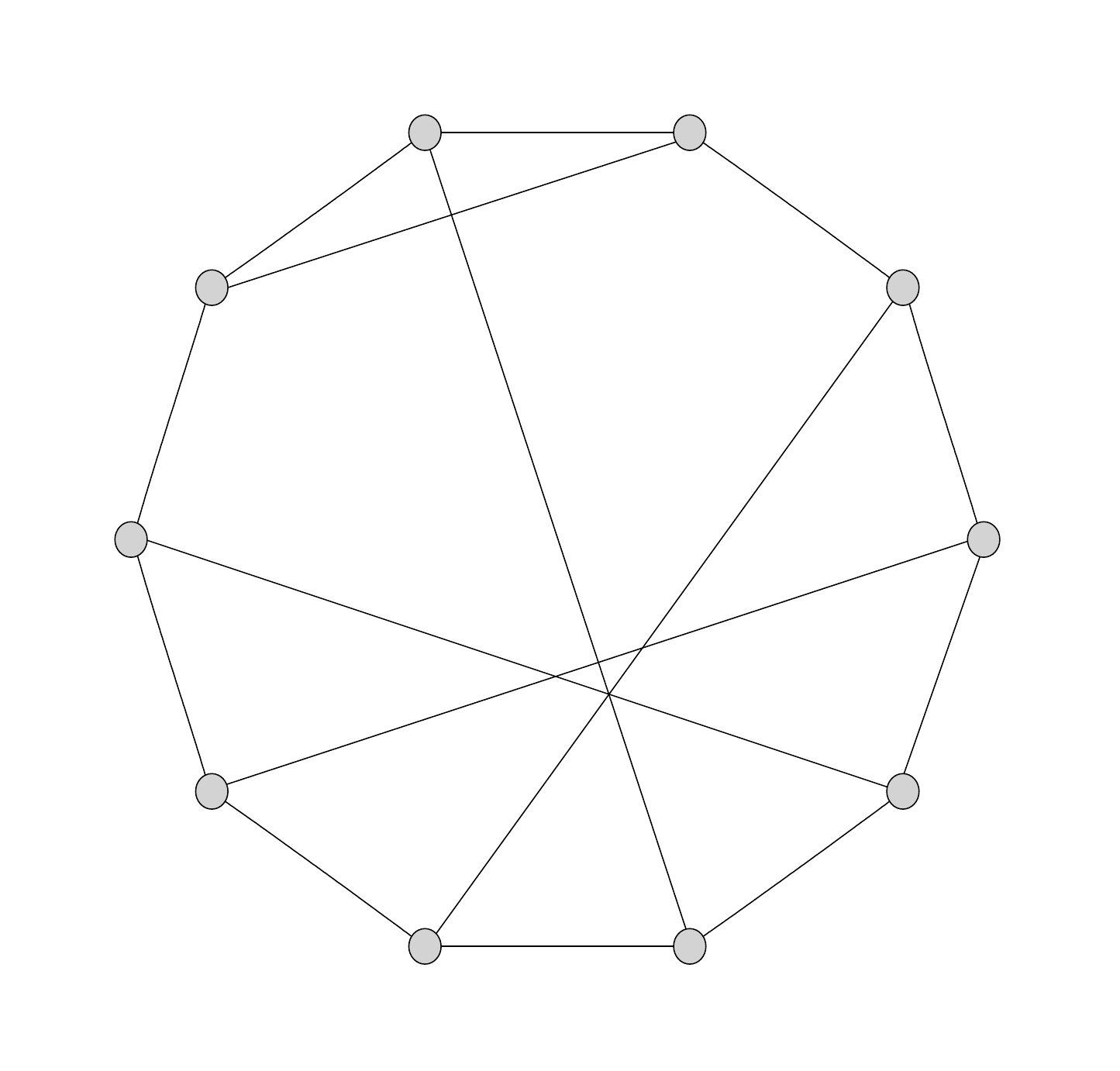} \\
 \scalebox{0.26}{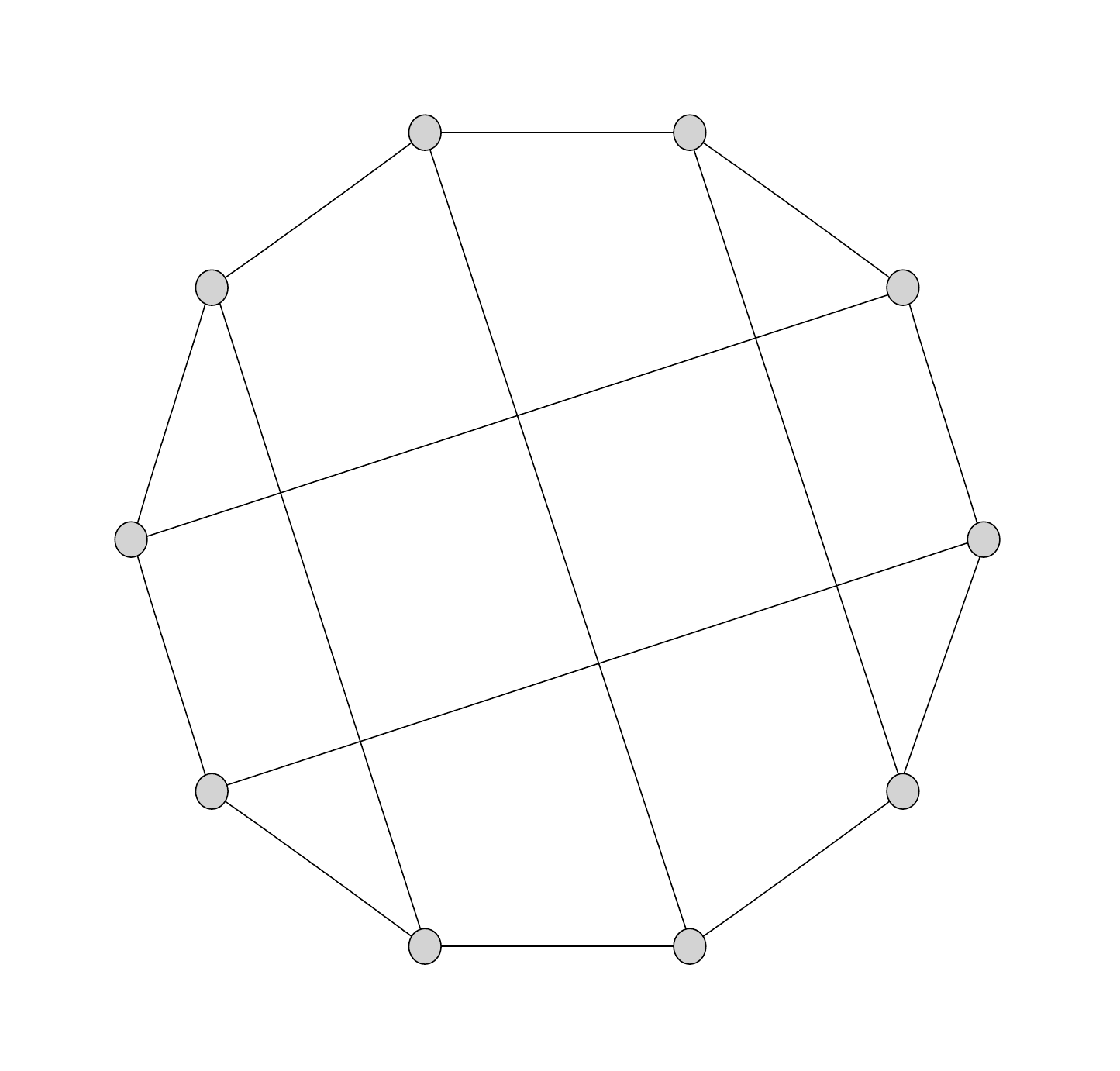}
 \scalebox{0.26}{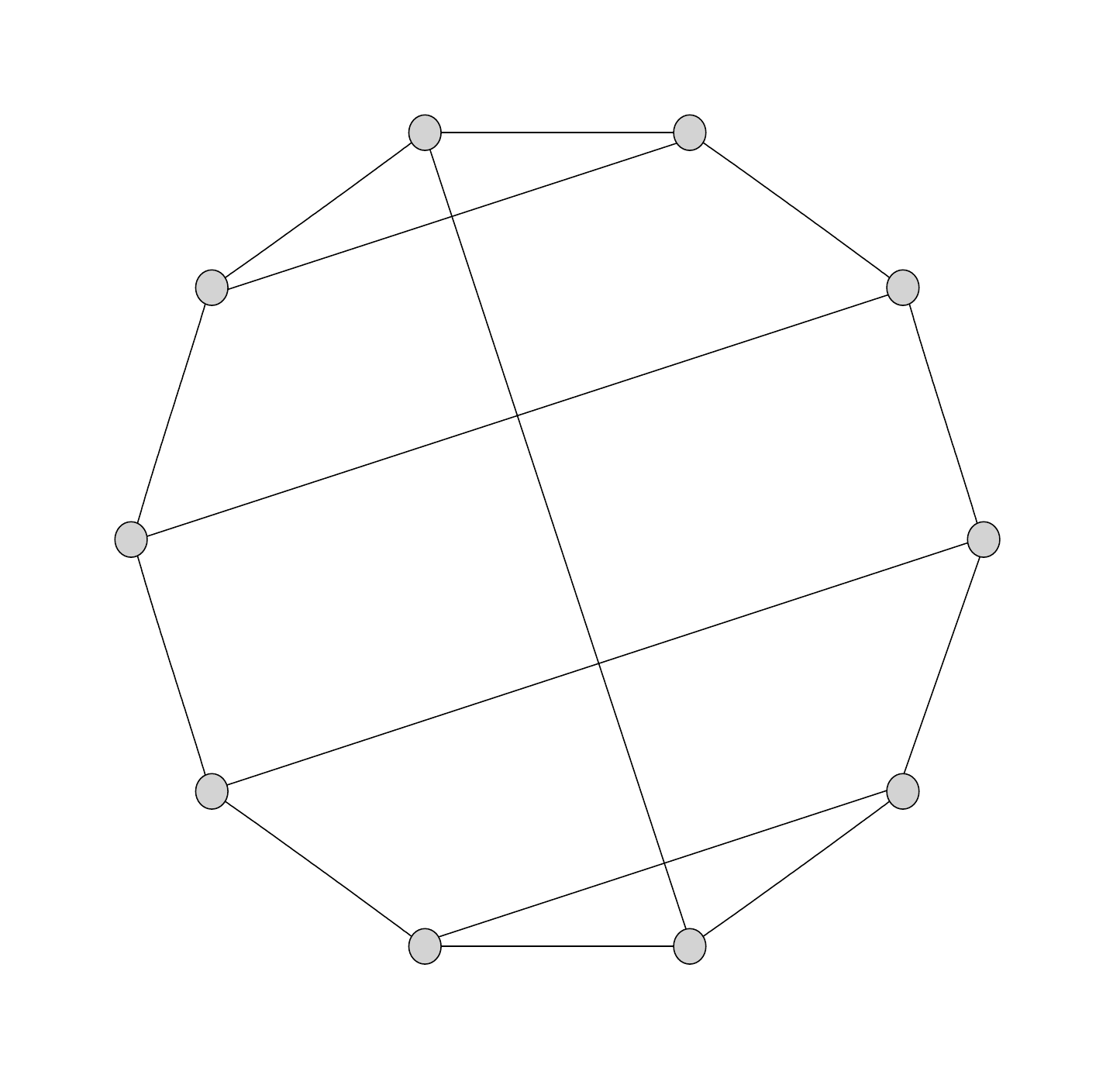}
 \scalebox{0.26}{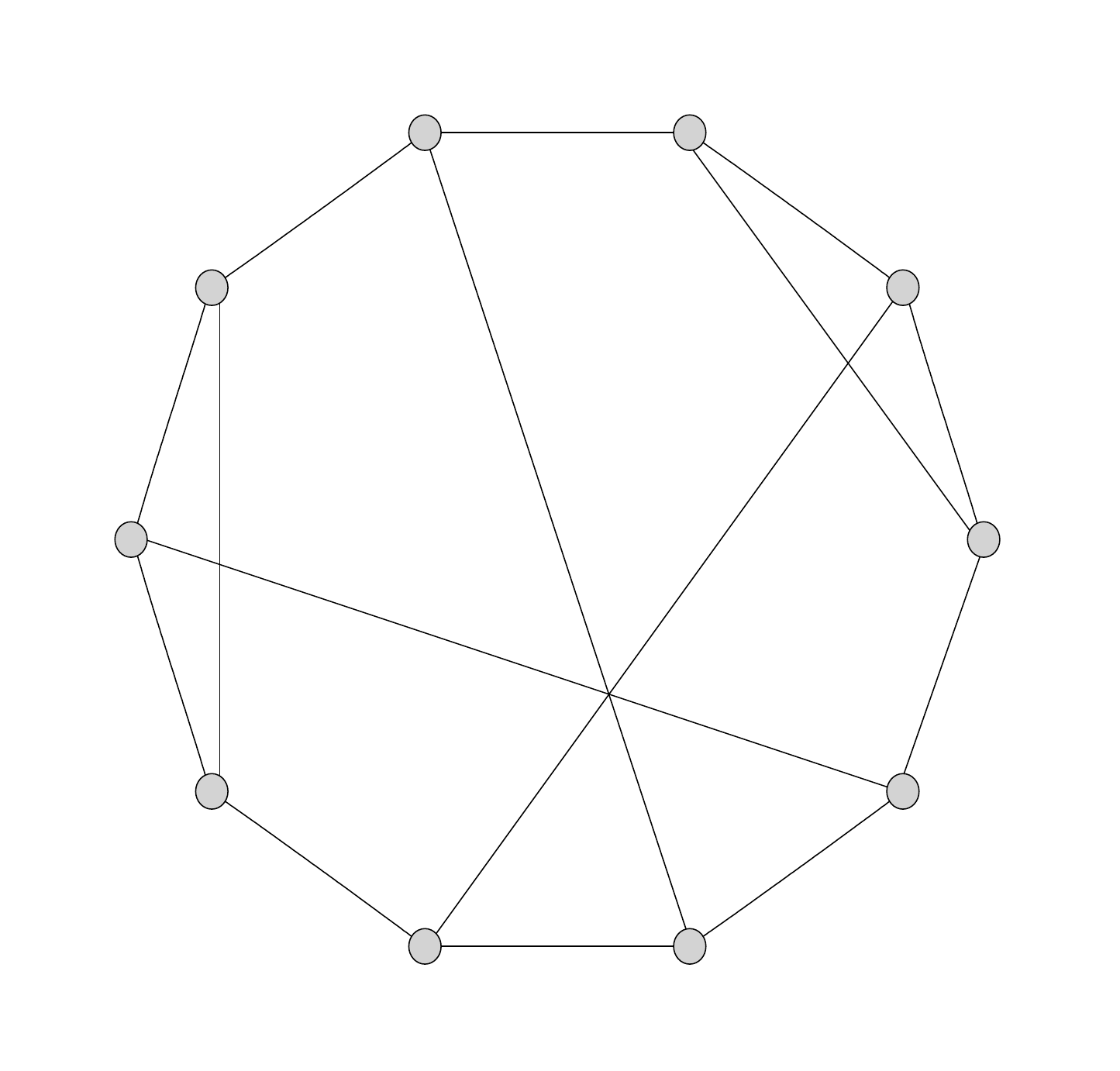} \\
 \scalebox{0.26}{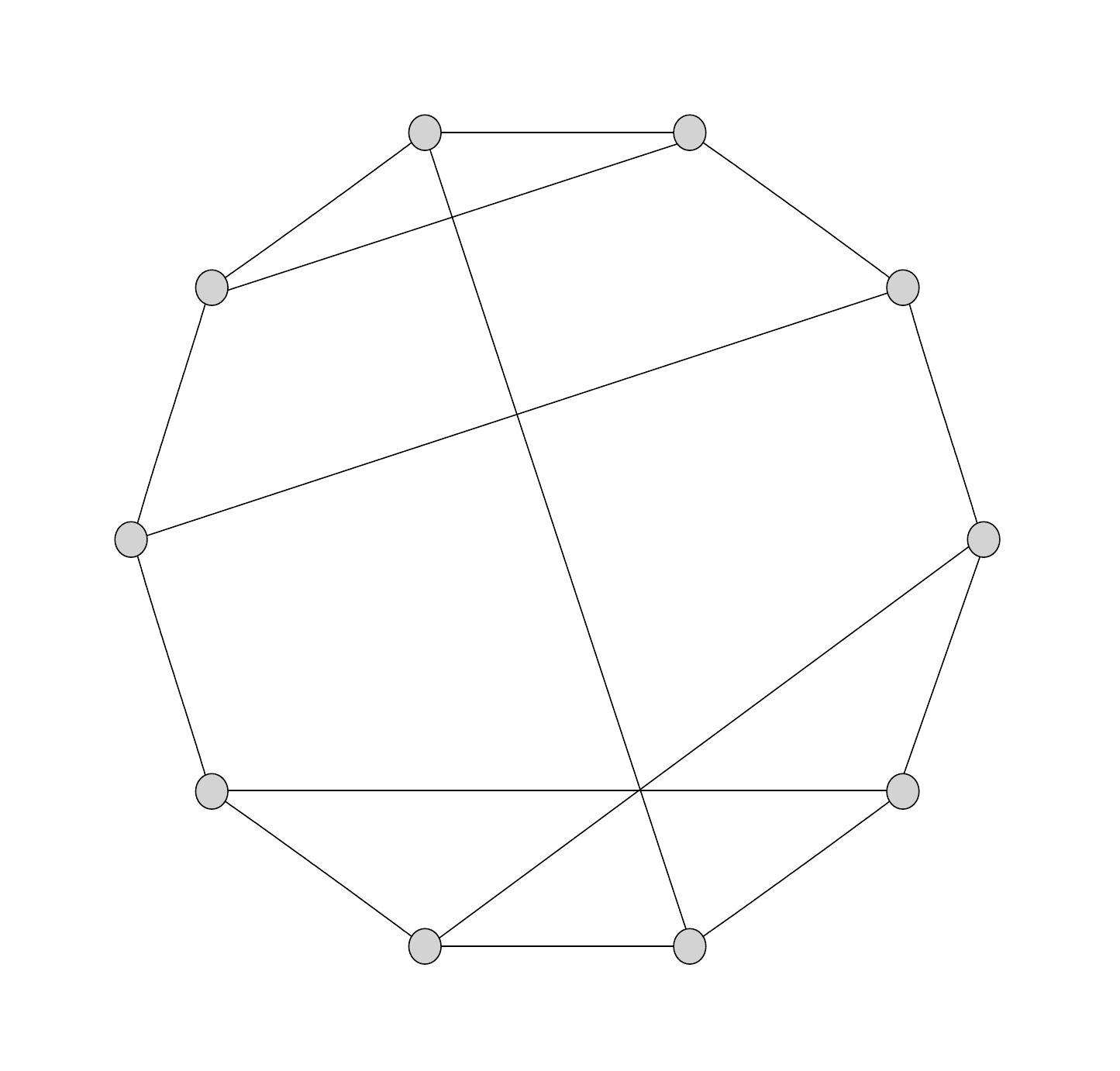}
 \scalebox{0.26}{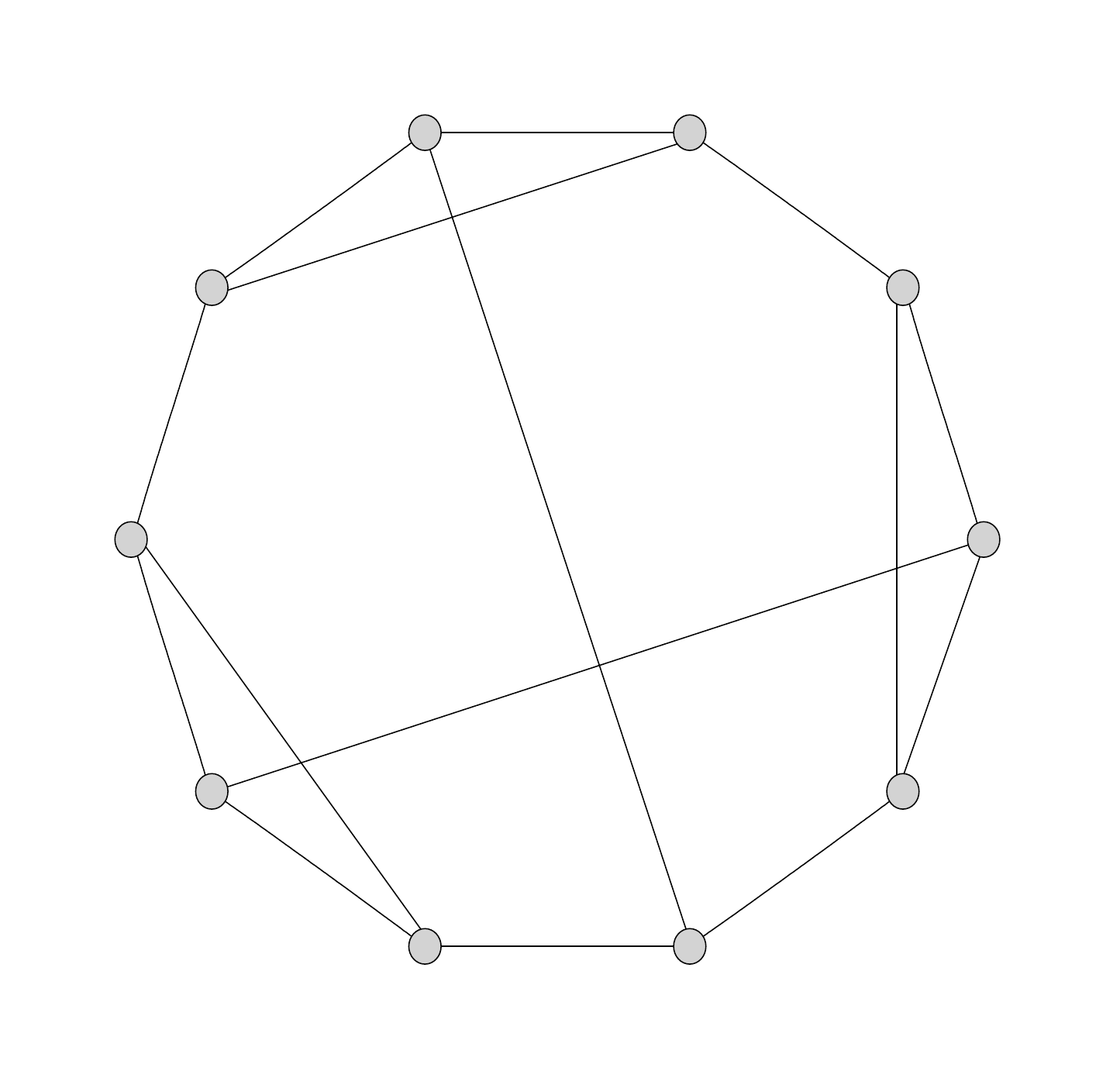}
 \scalebox{0.26}{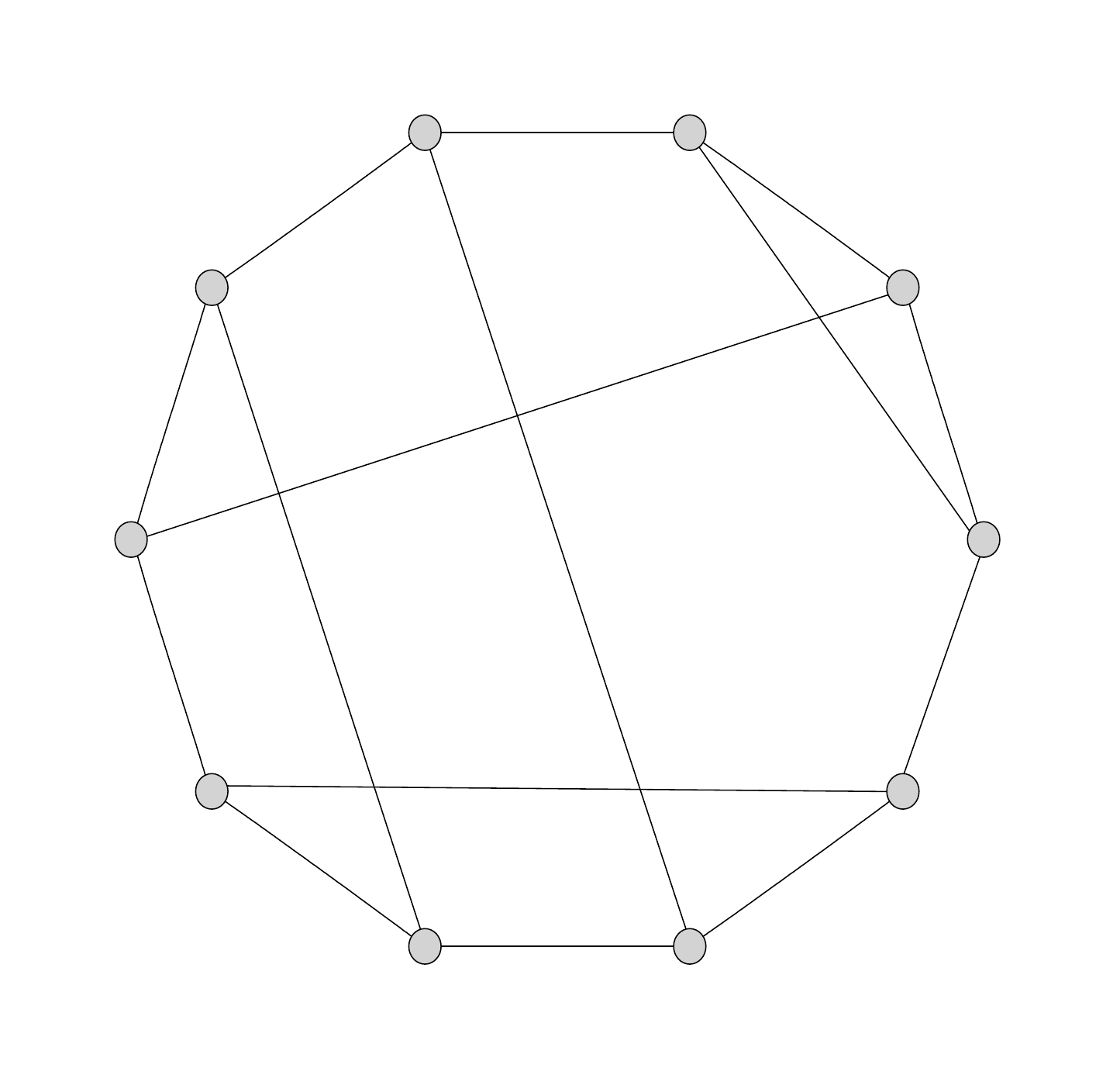} \\
 \scalebox{0.26}{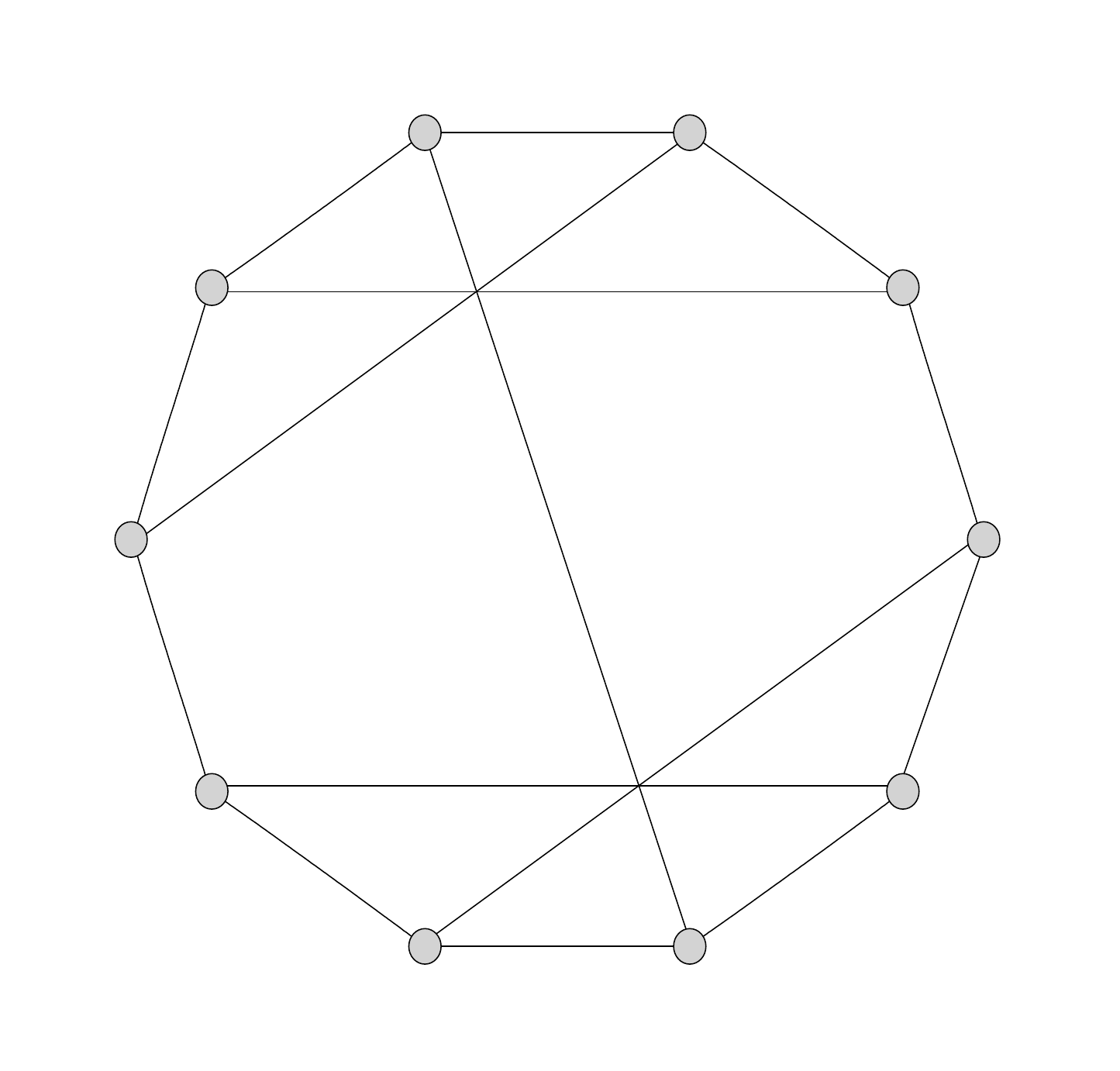} 
 \scalebox{0.26}{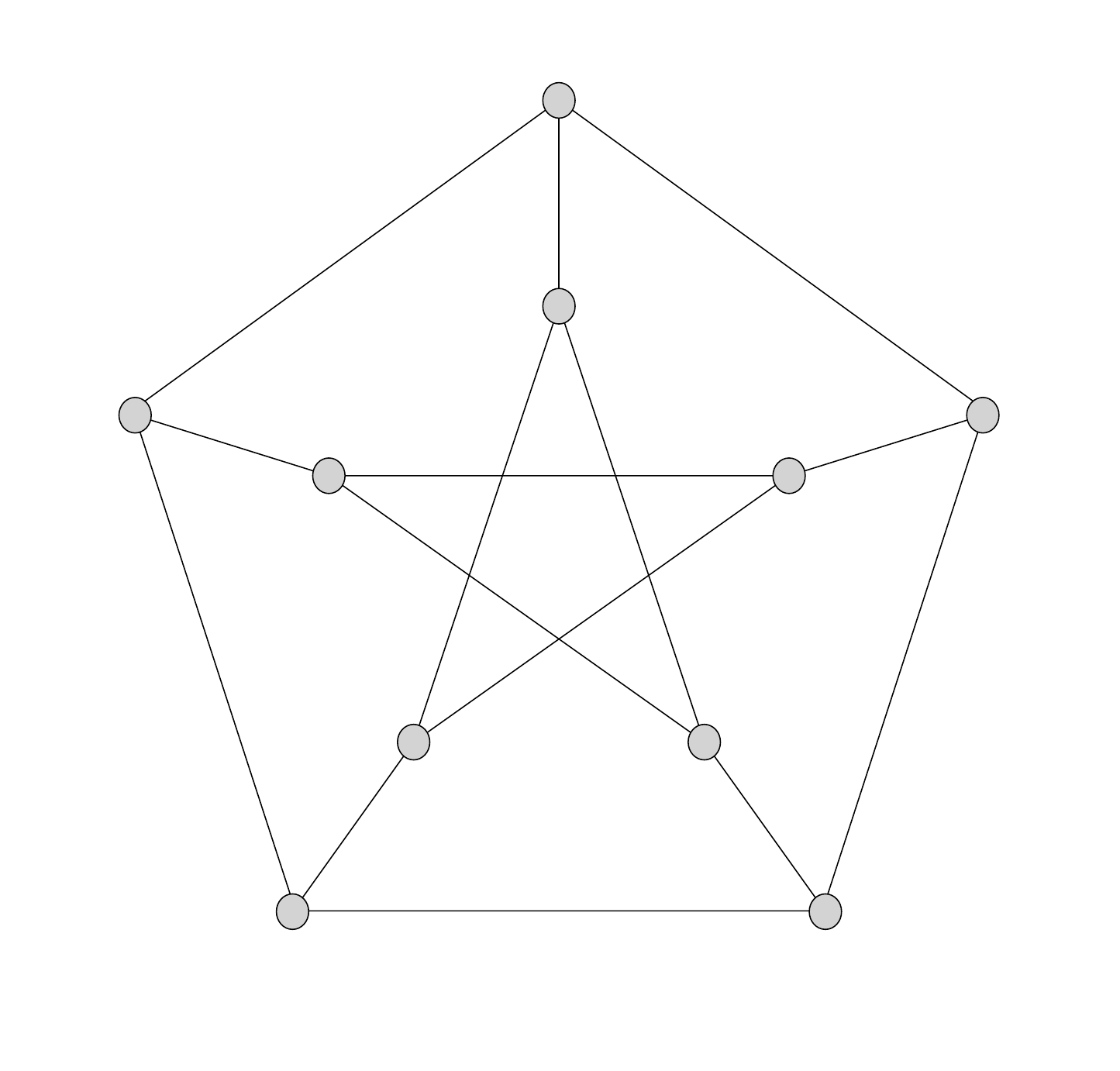}
 \scalebox{0.26}{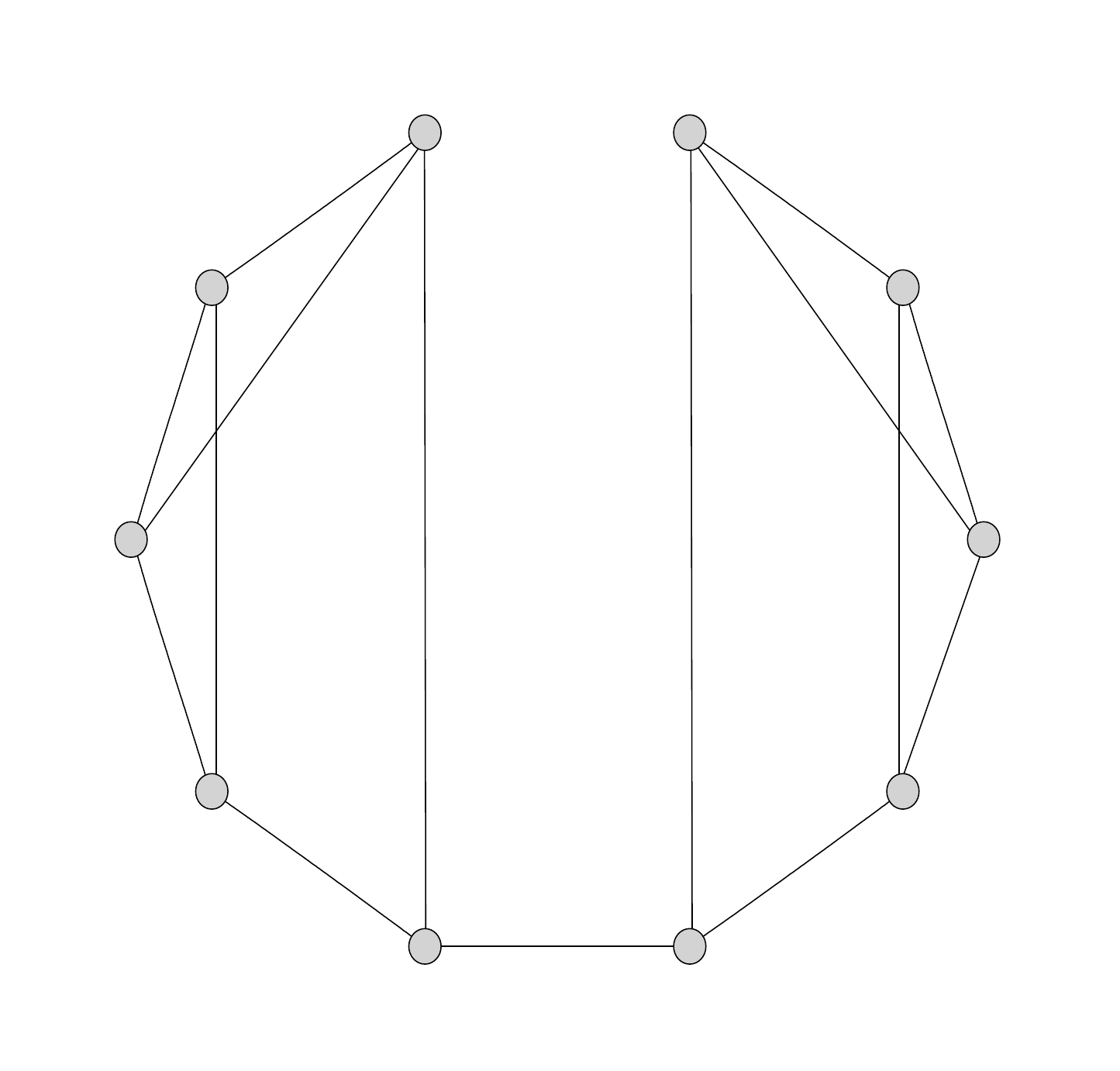} \\
\scalebox{0.26}{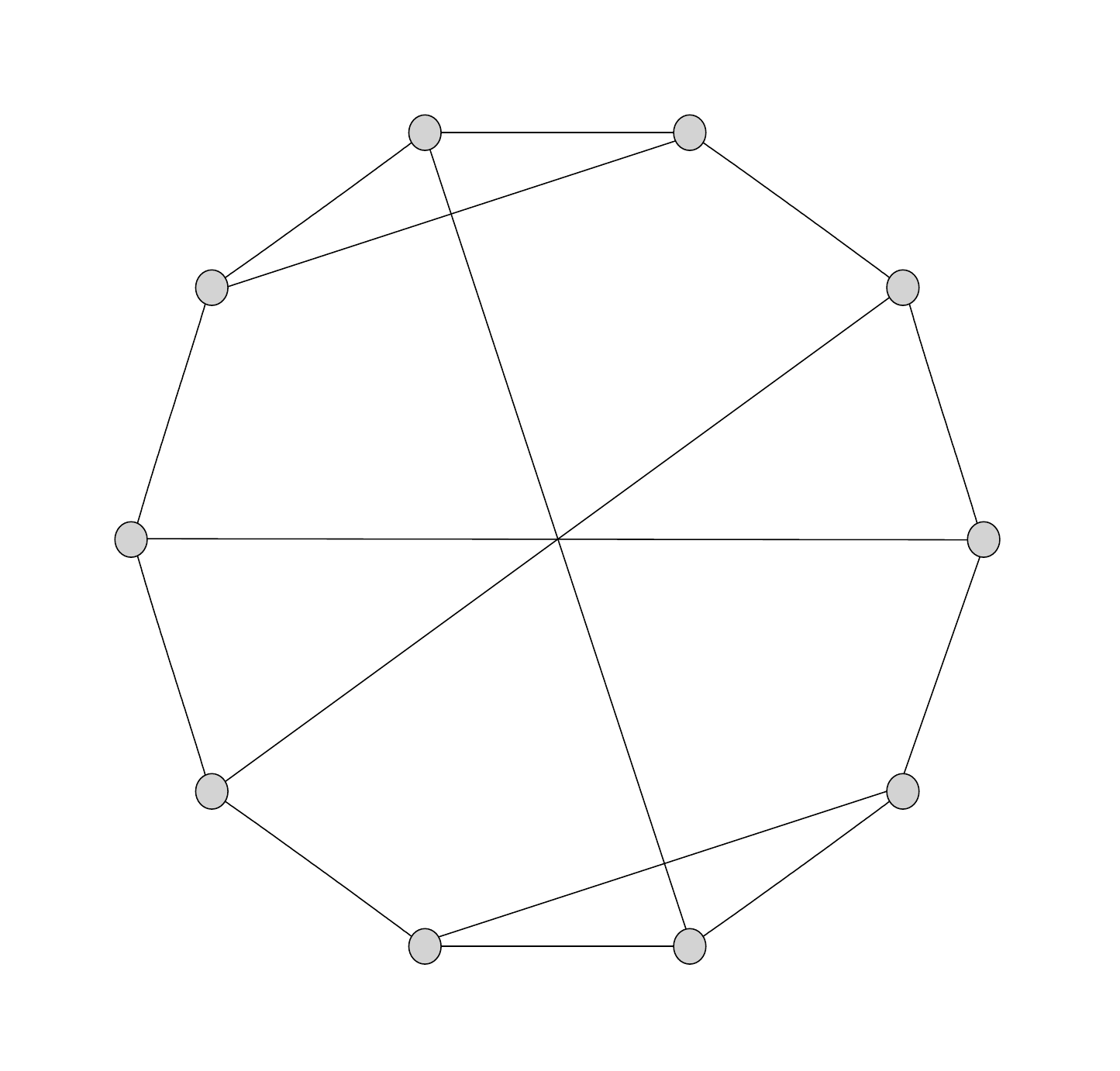}
 \scalebox{0.26}{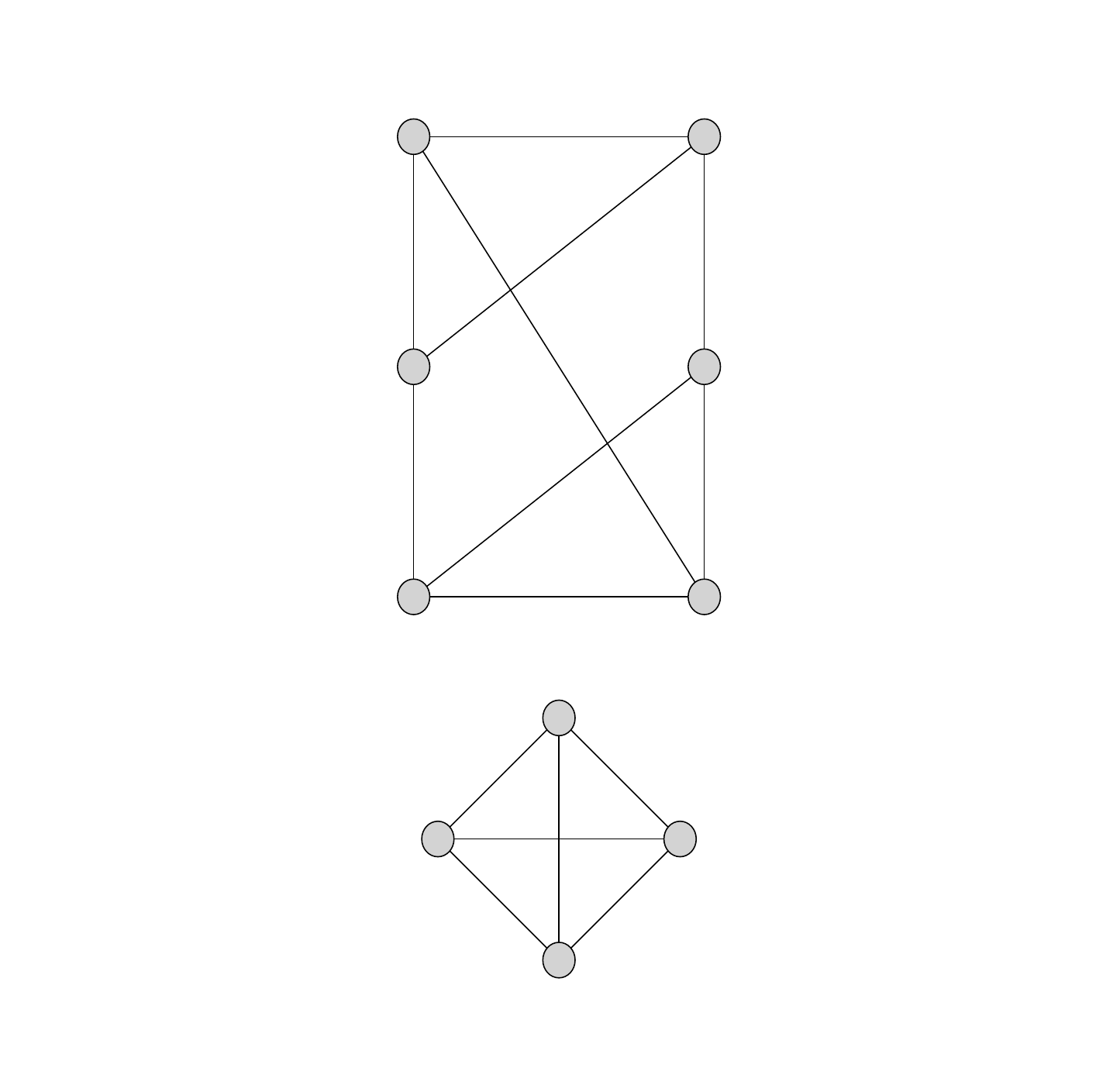}
 \scalebox{0.26}{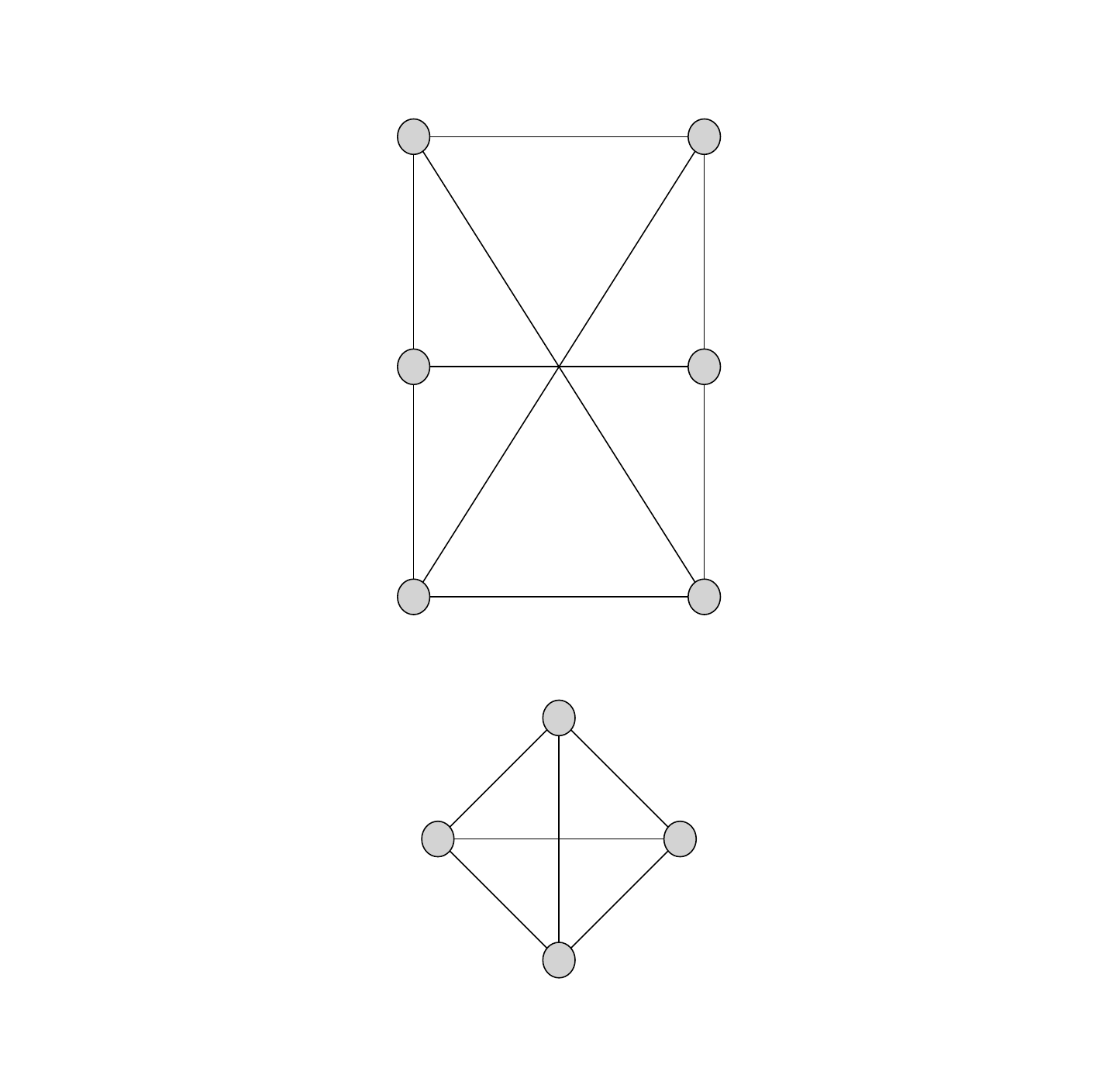} \\
\end{center}
\section*{Appendix B}
This appendix depicts $23$ graphs $J_i$ ($i \in [23]$). The vertices
of each graph $J_i$ ($i \in [23]$) are labelled with the integers $1$
to $7$ such that the vertices labelled $j \in [7]$ constitute the
$j$th branch set of a $K_{7}^{-}$- or $K_7$ minor. If the branch sets
only constitute a $K_{7}^{-}$ minor, then it is because there is no
edge between the branch sets of vertices labelled $1$ and $7$,
respectively. 
\begin{center}
\scalebox{0.26}{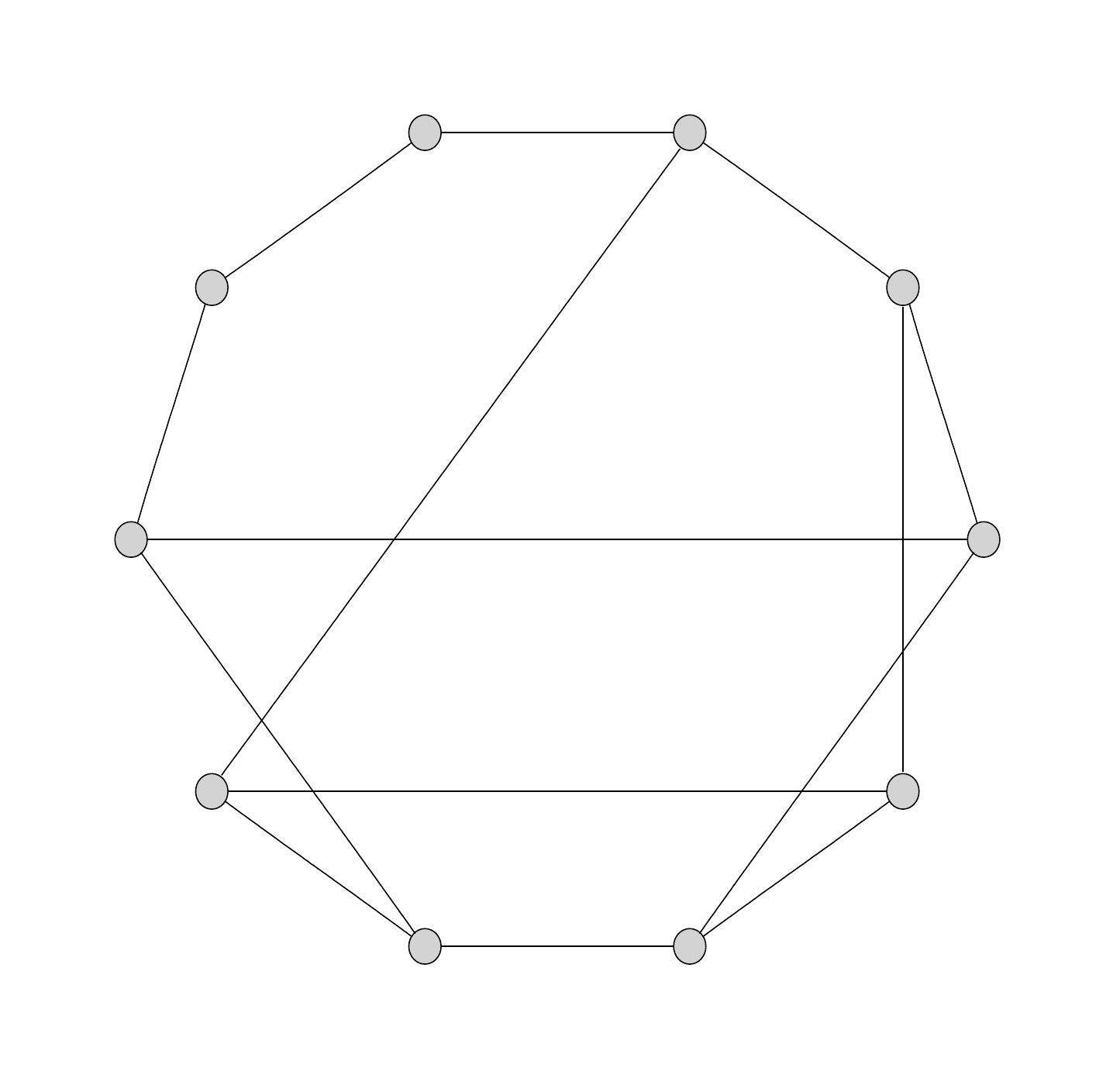}
 \scalebox{0.26}{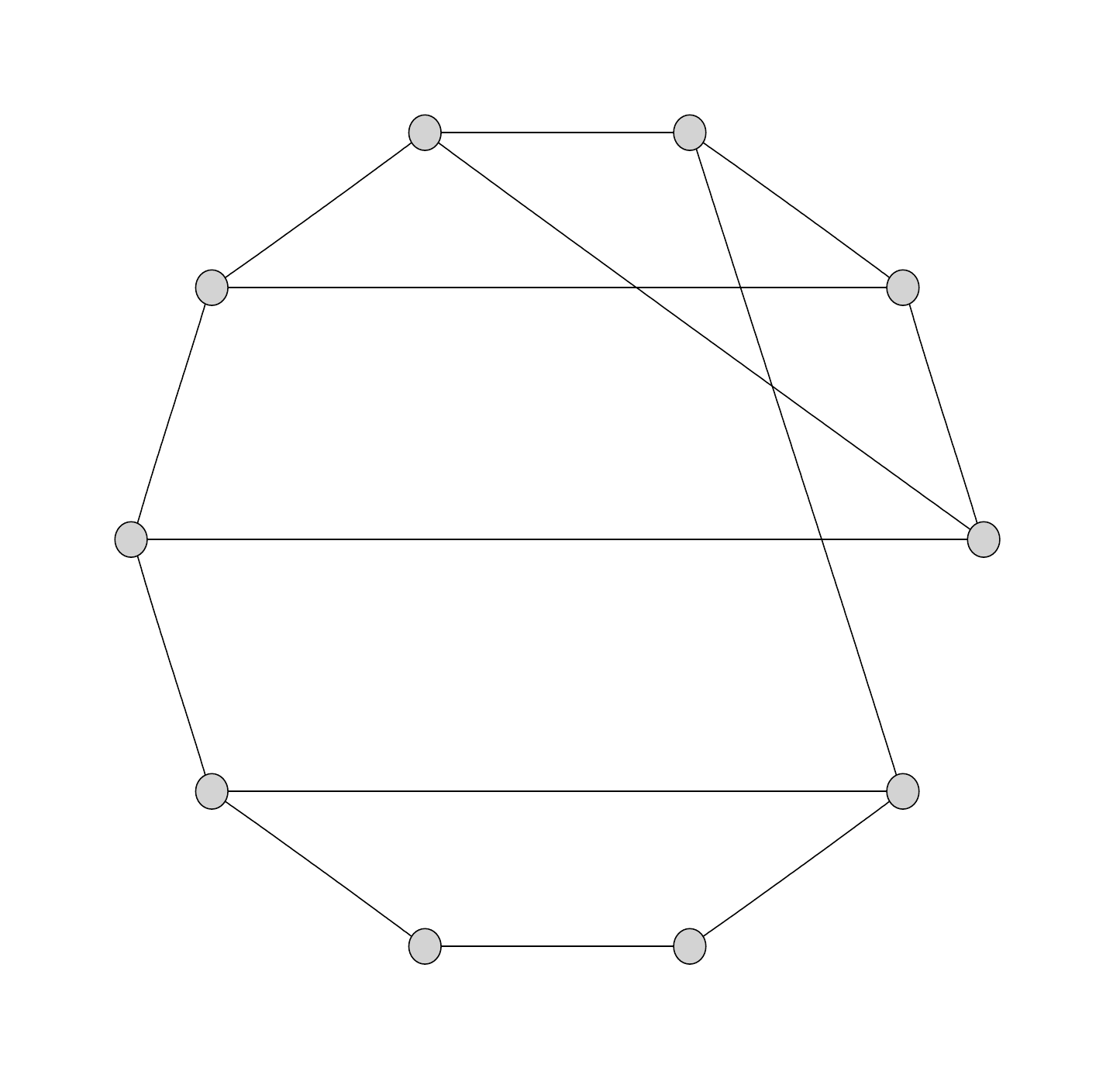}
 \scalebox{0.26}{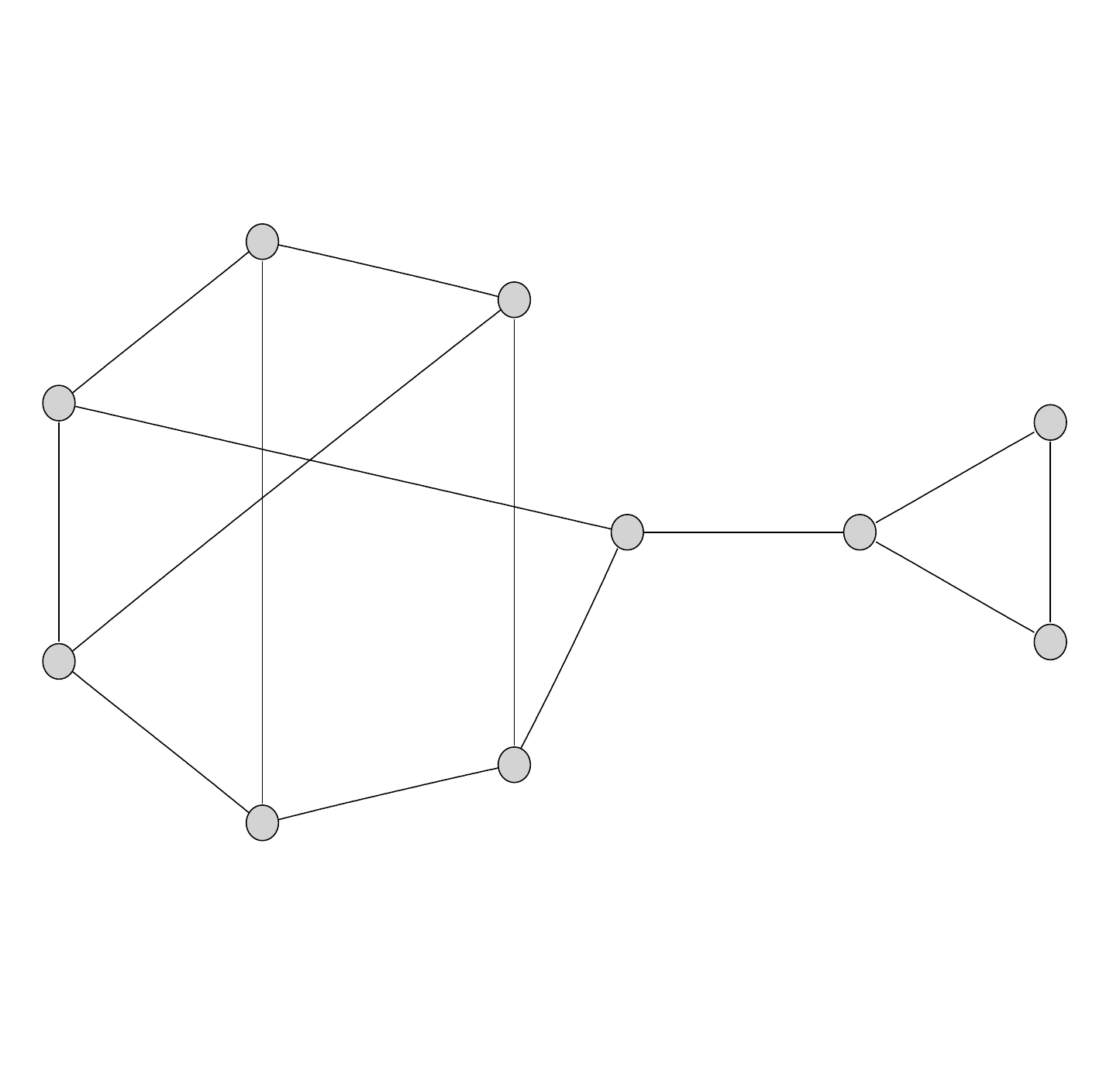} \\
\scalebox{0.26}{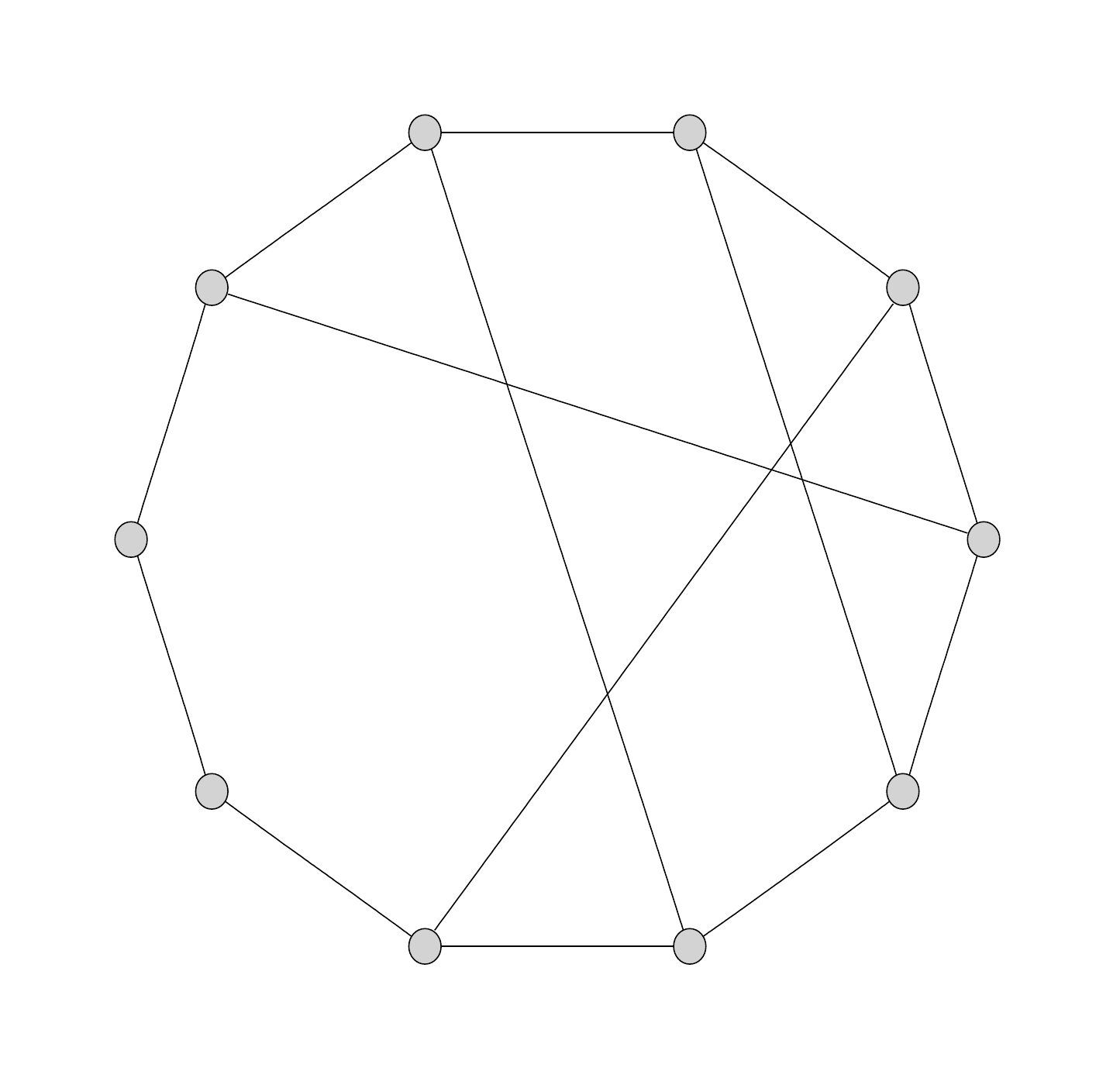}
 \scalebox{0.26}{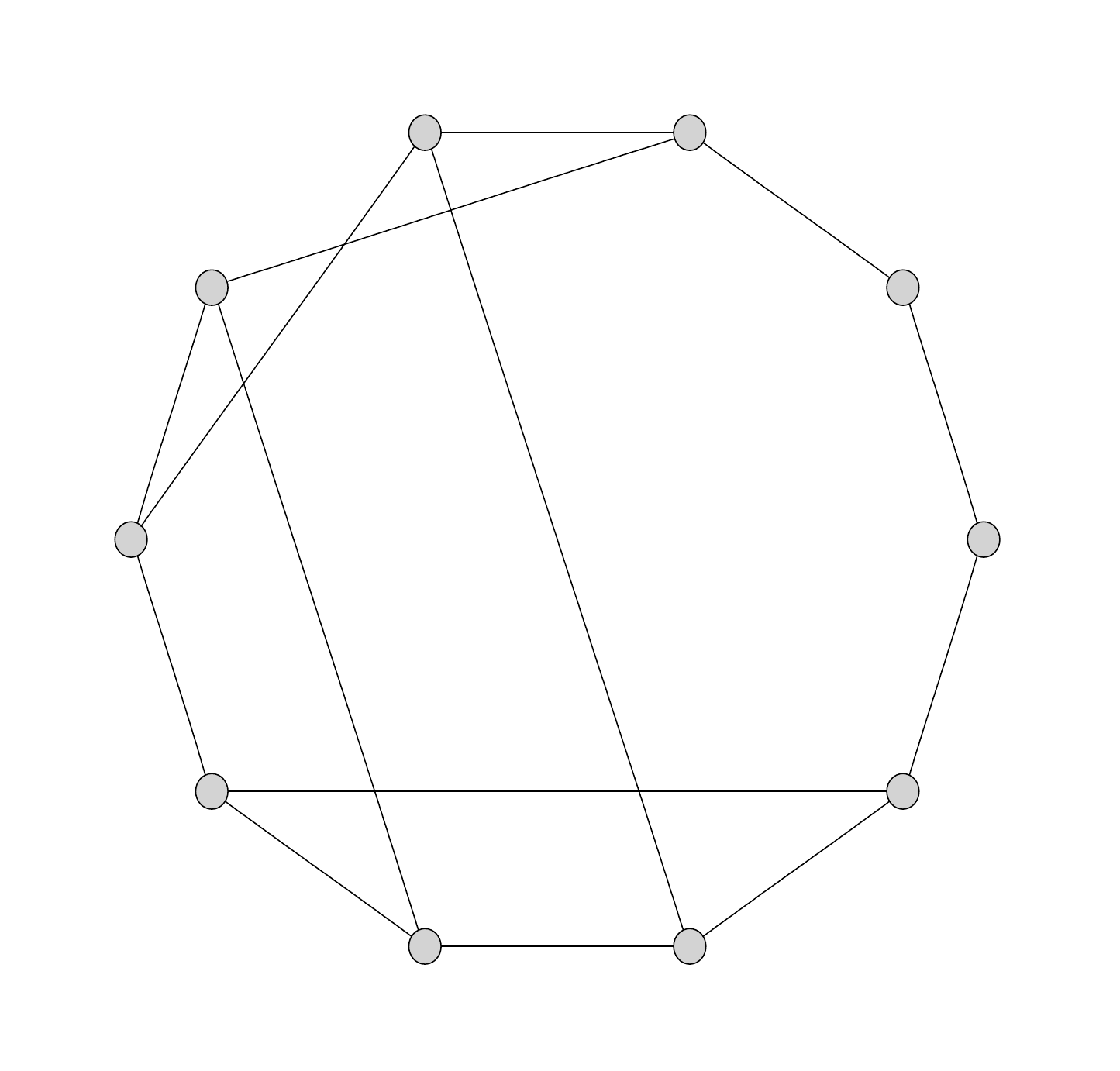}
 \scalebox{0.26}{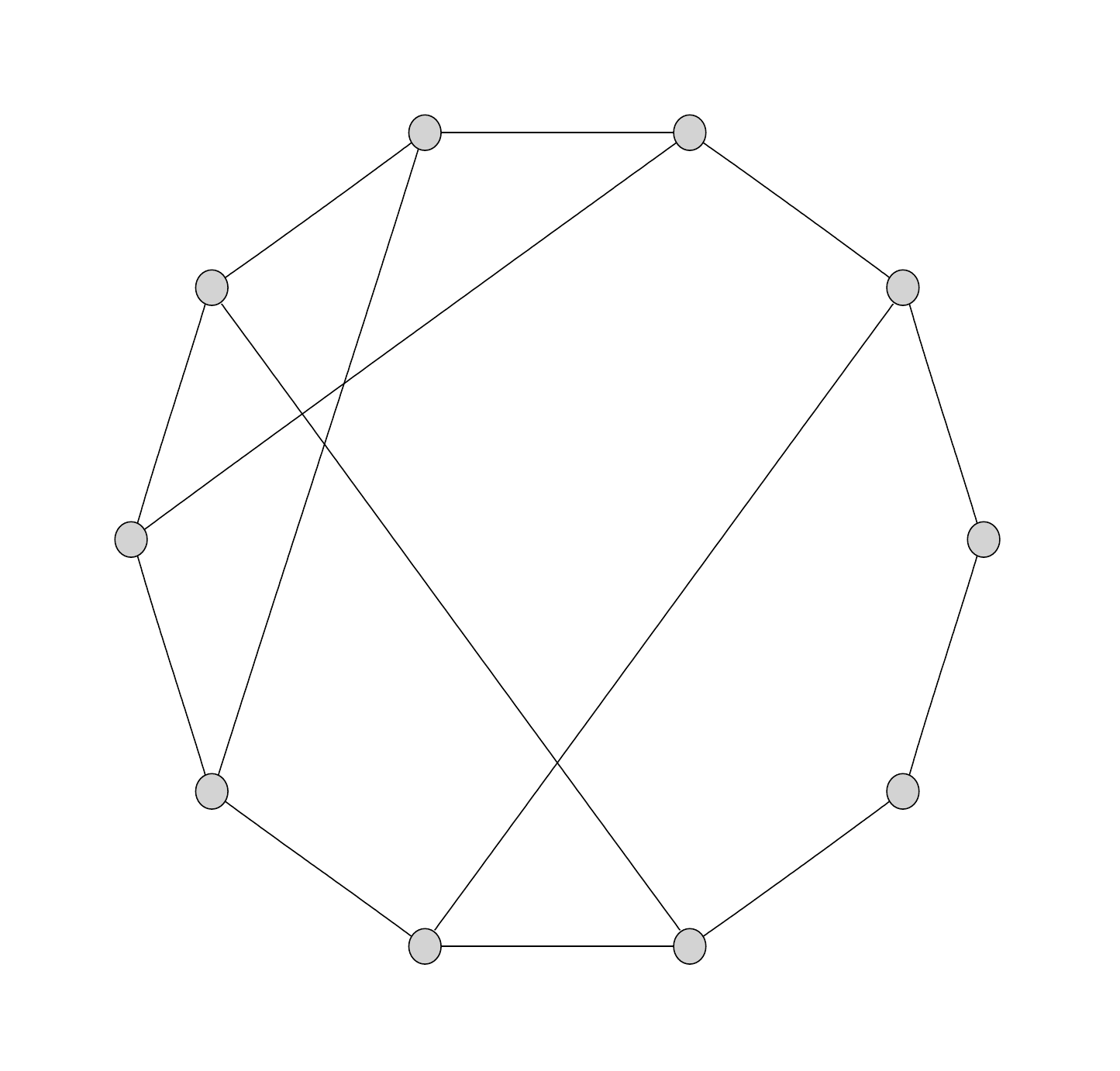} \\
\scalebox{0.26}{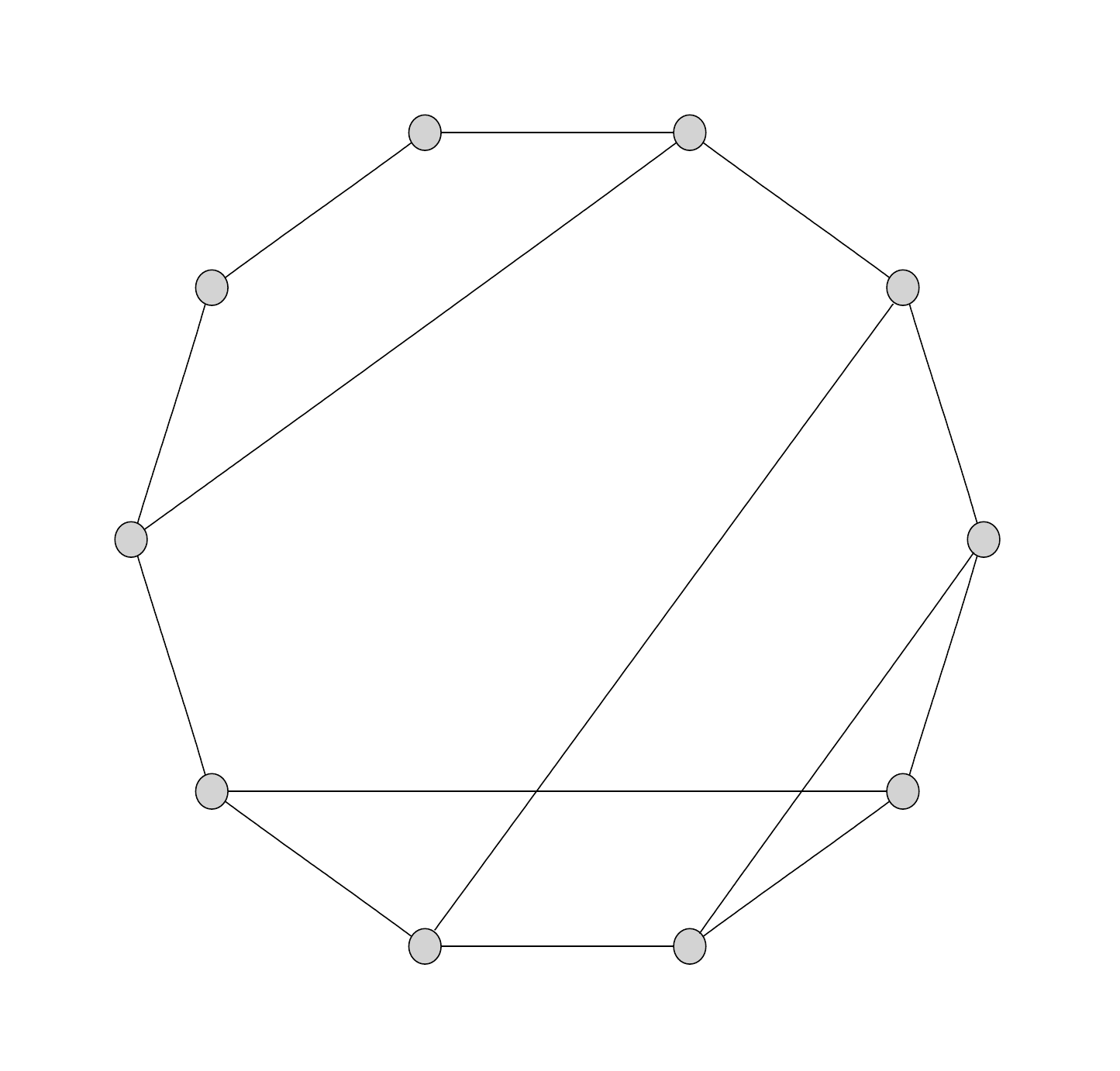}
 \scalebox{0.26}{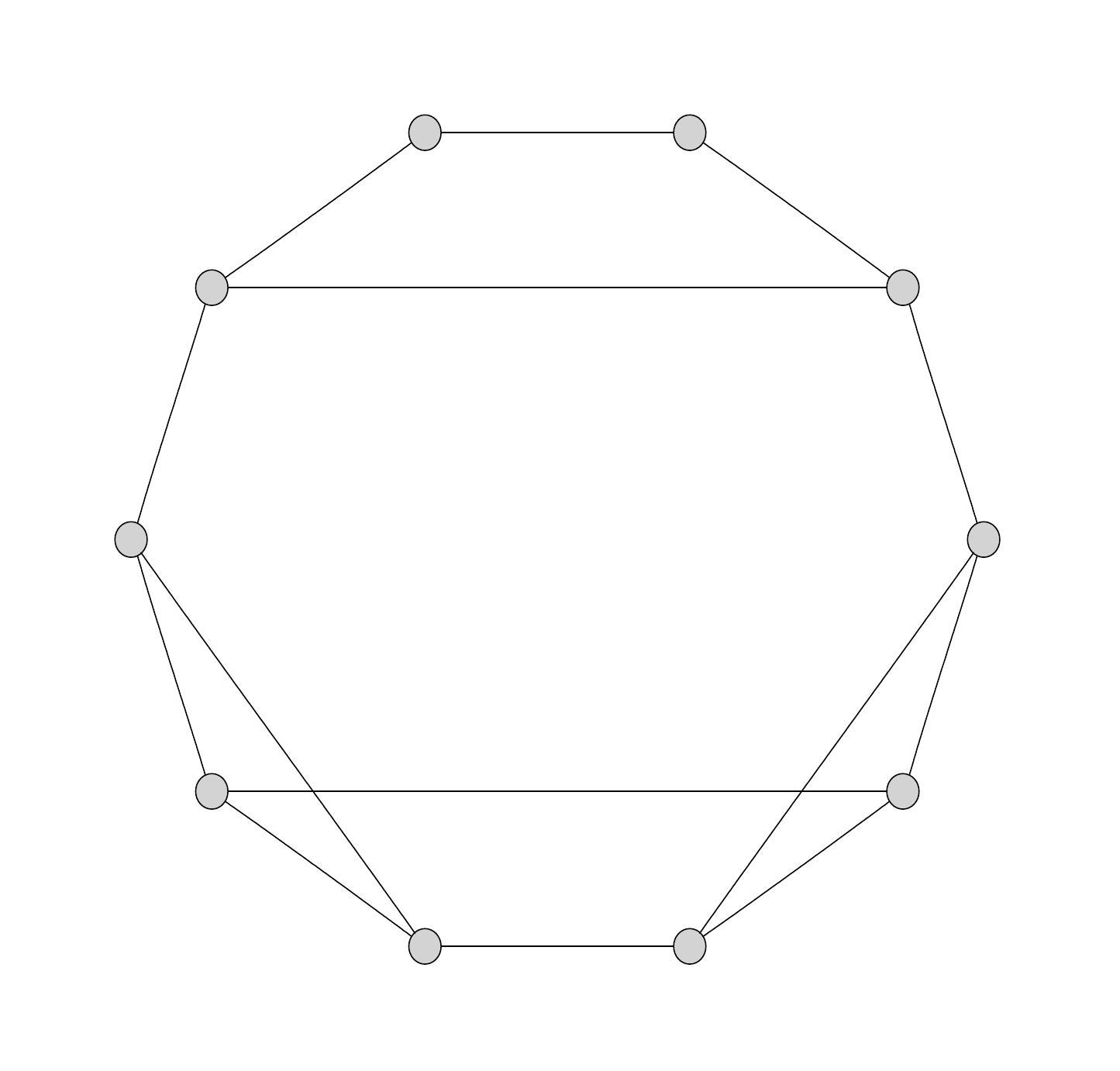}
 \scalebox{0.26}{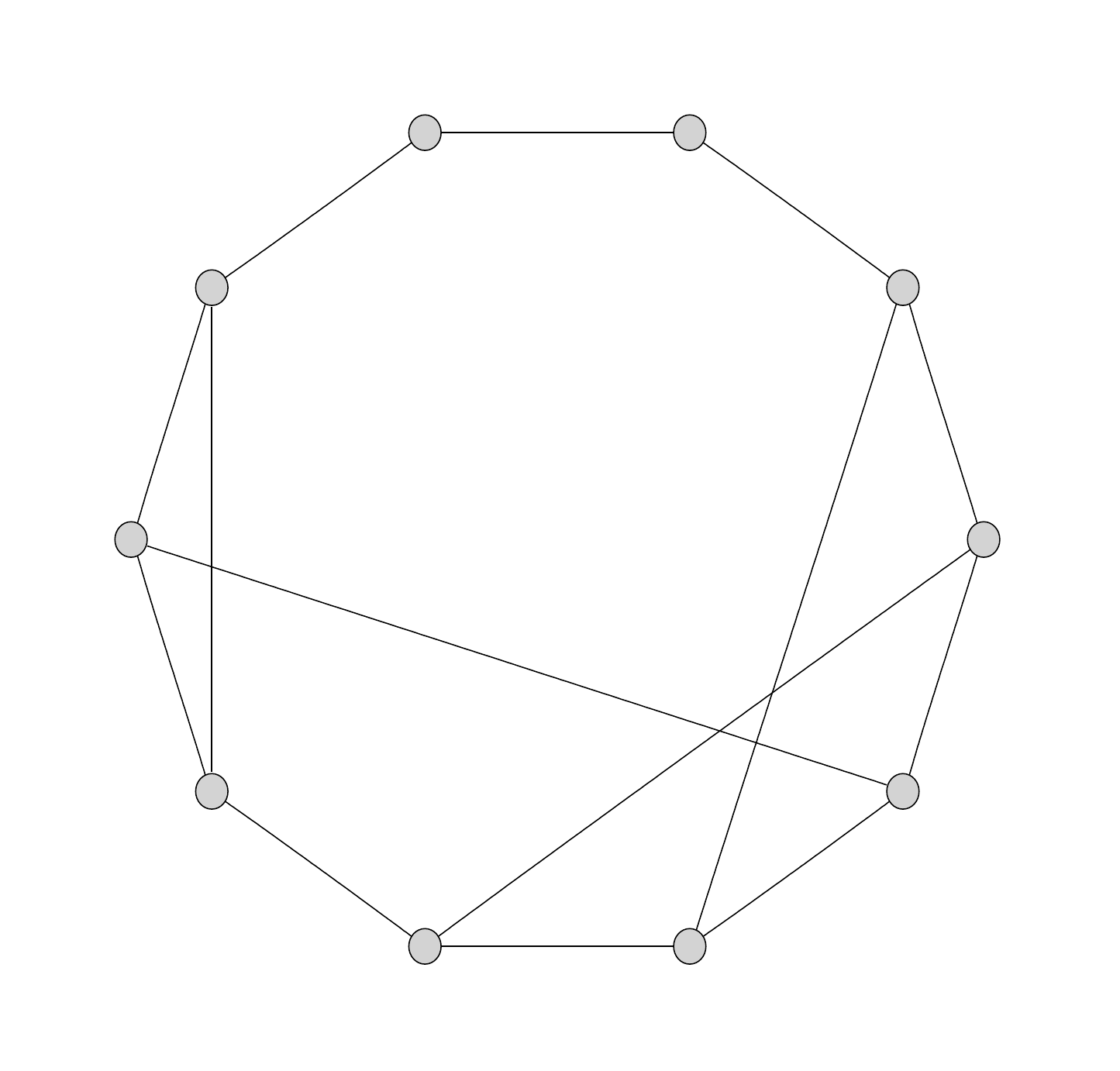} \\
 \scalebox{0.26}{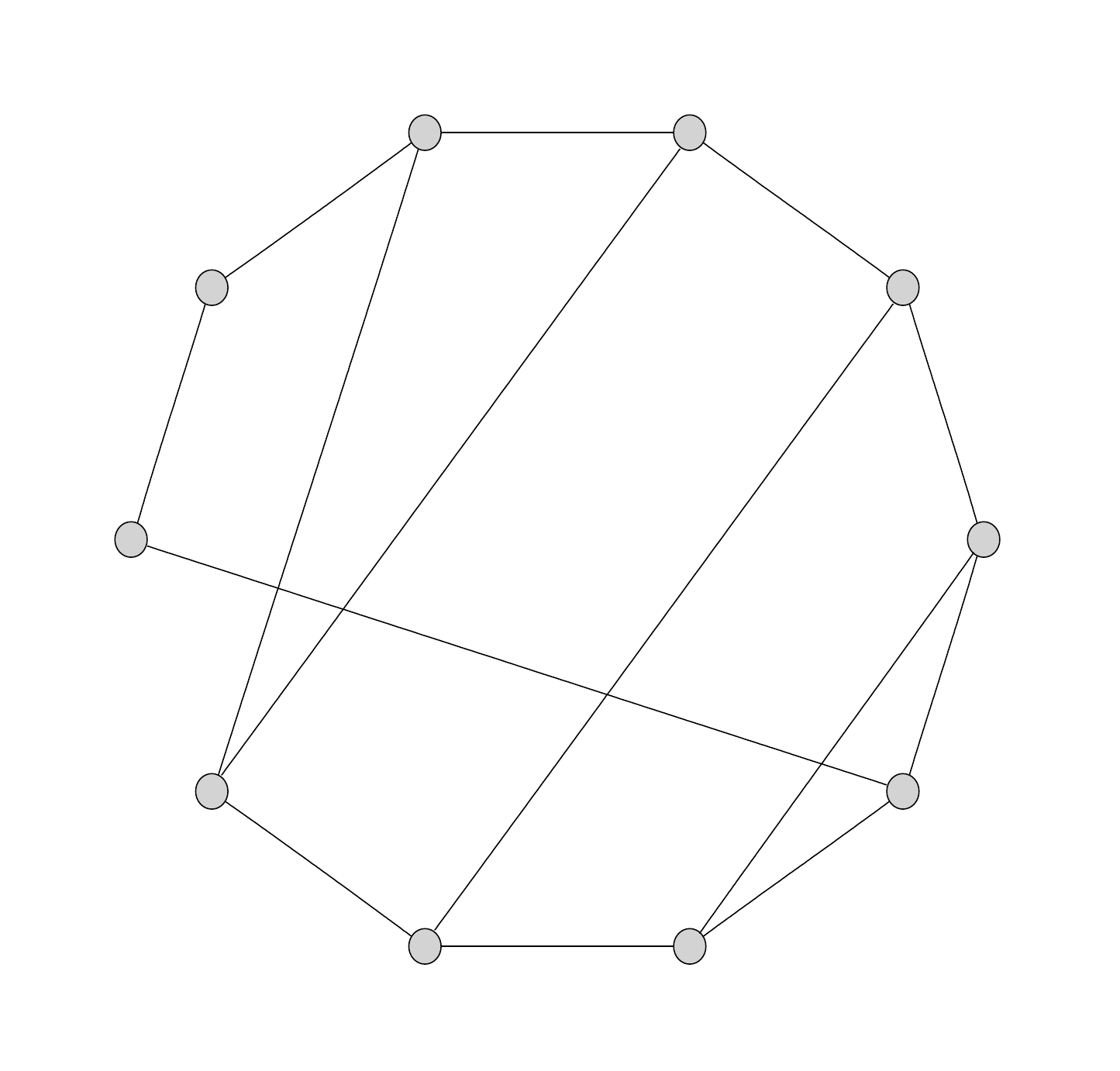}
 \scalebox{0.26}{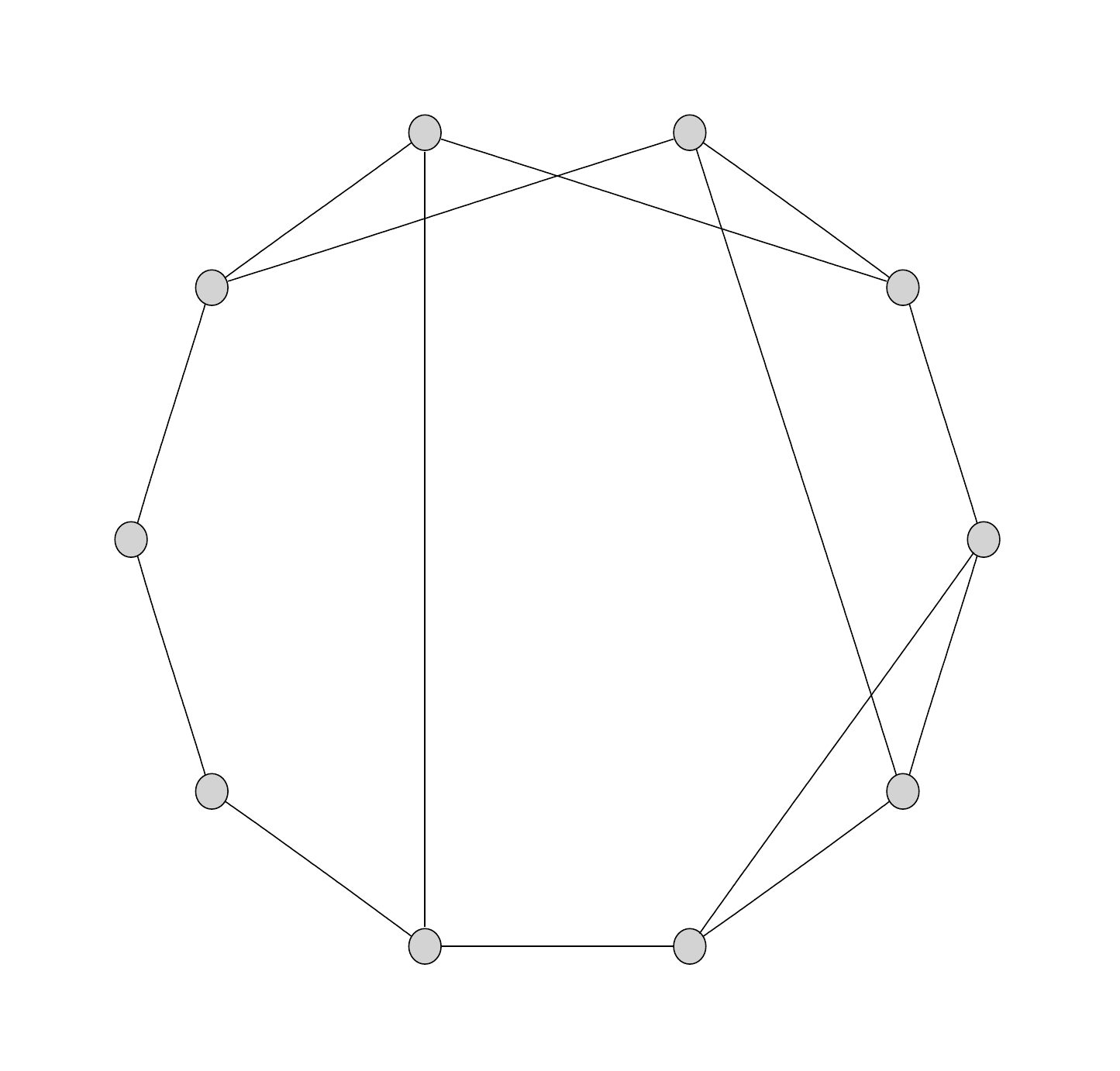}
 \scalebox{0.26}{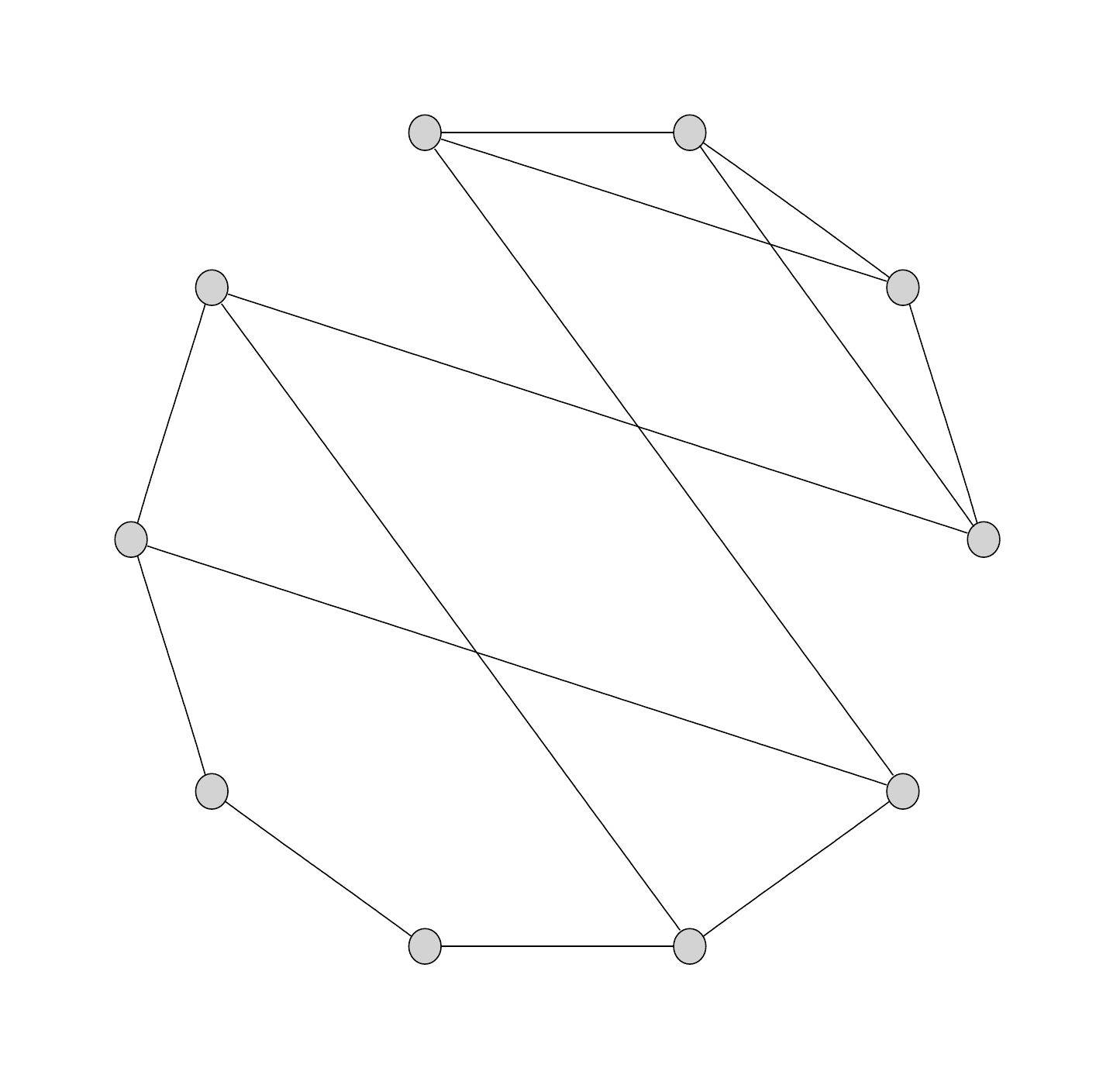} \\
\scalebox{0.26}{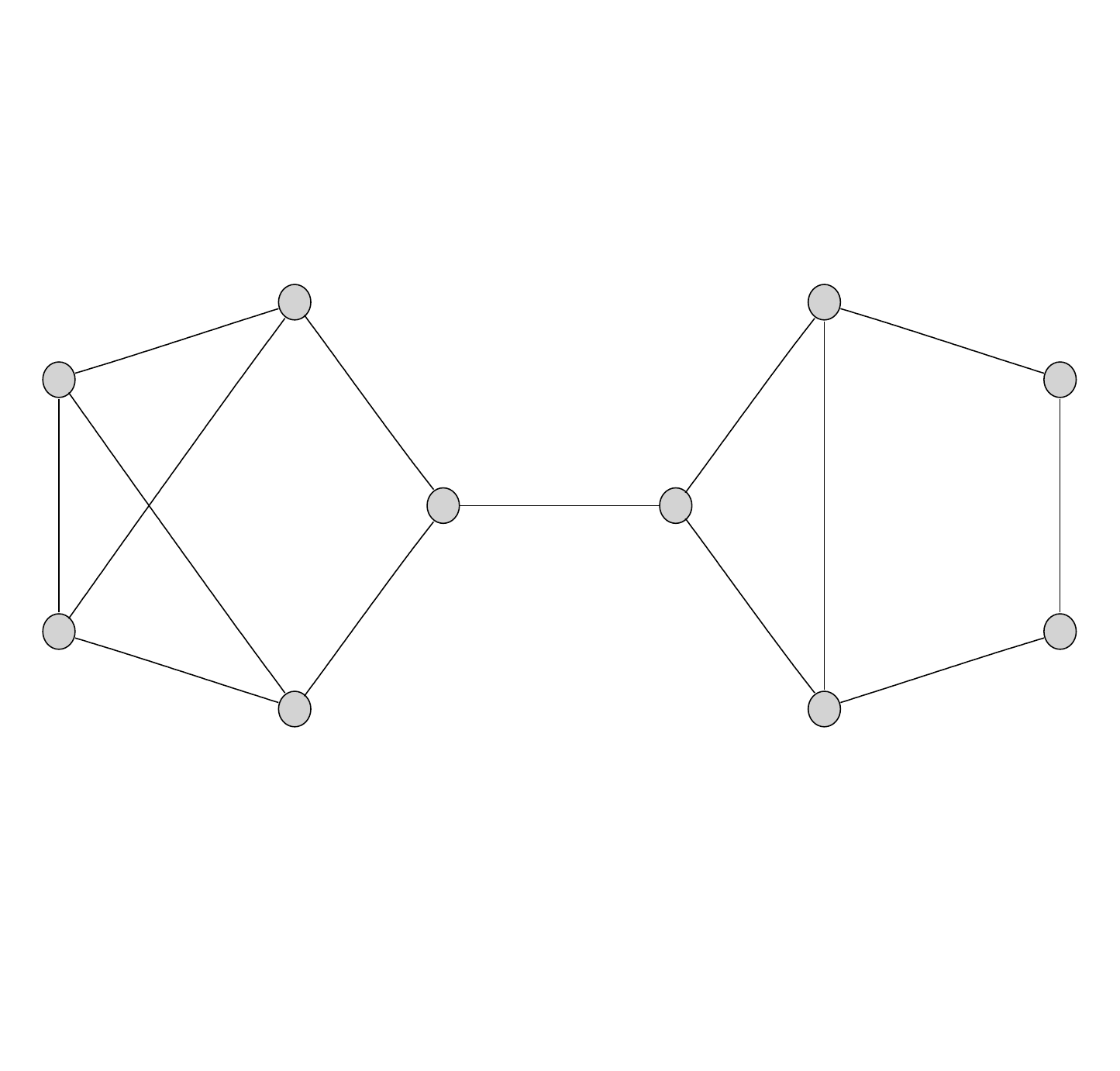}
 \scalebox{0.26}{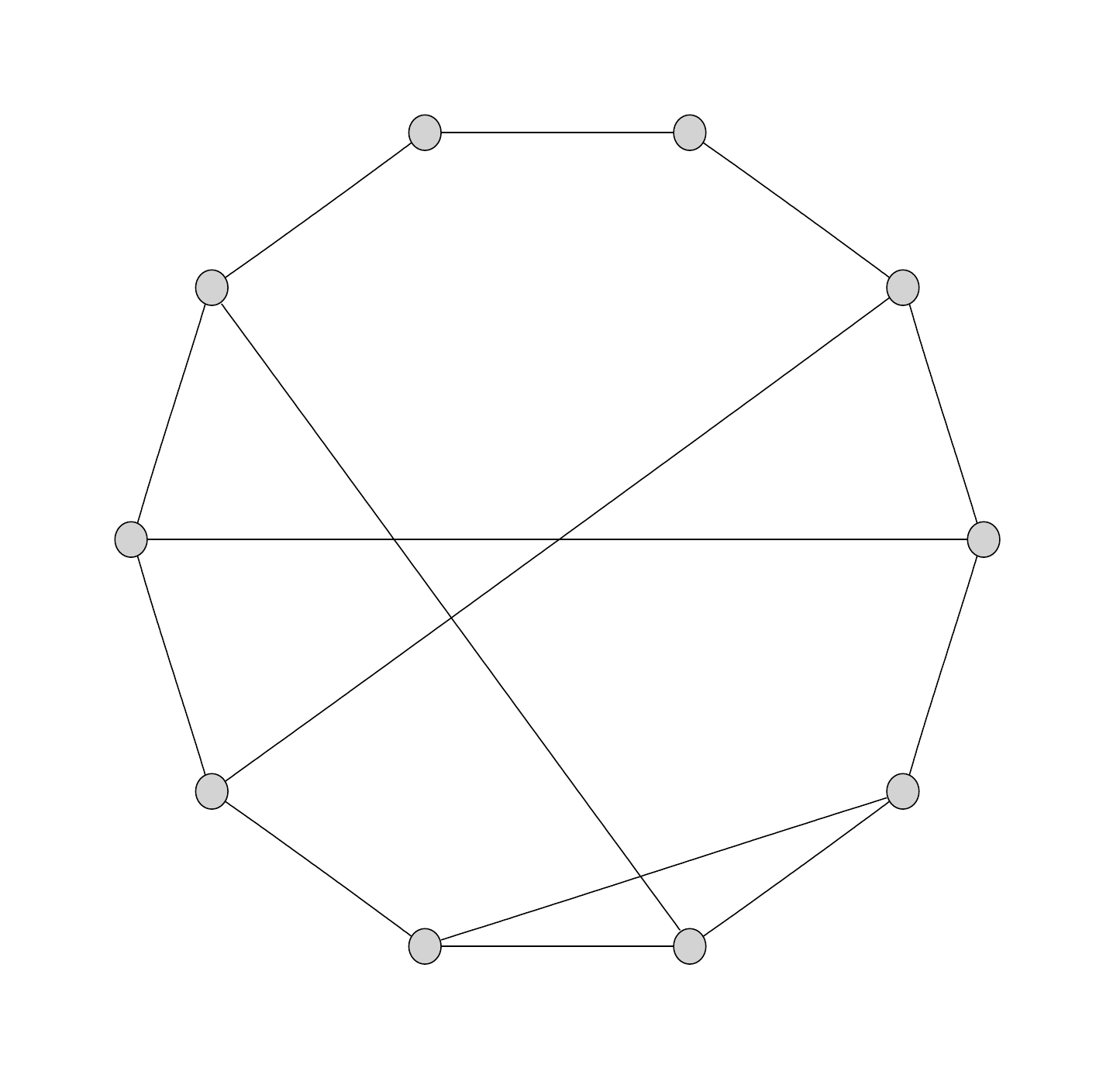}
 \scalebox{0.26}{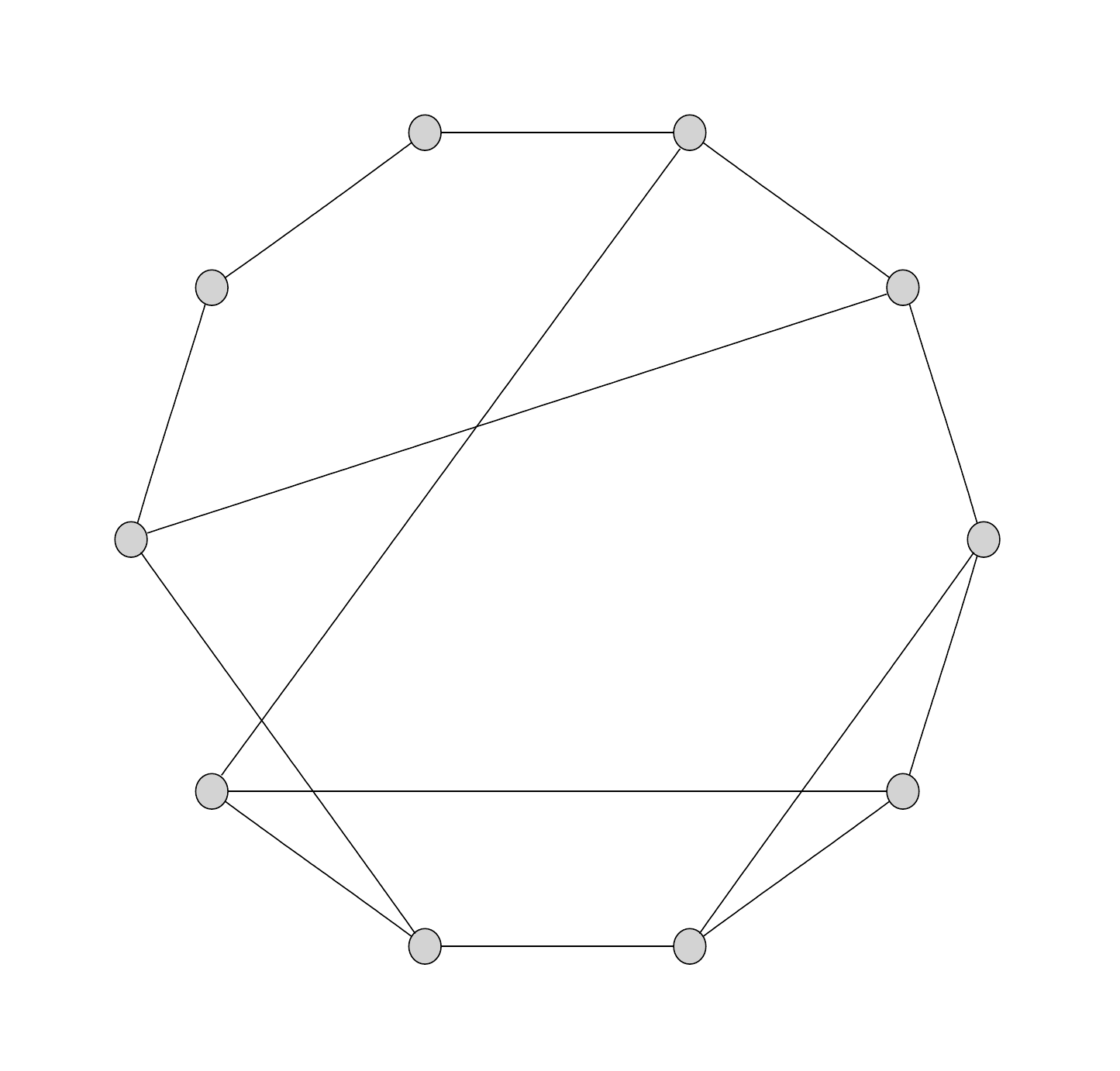} \\
\scalebox{0.26}{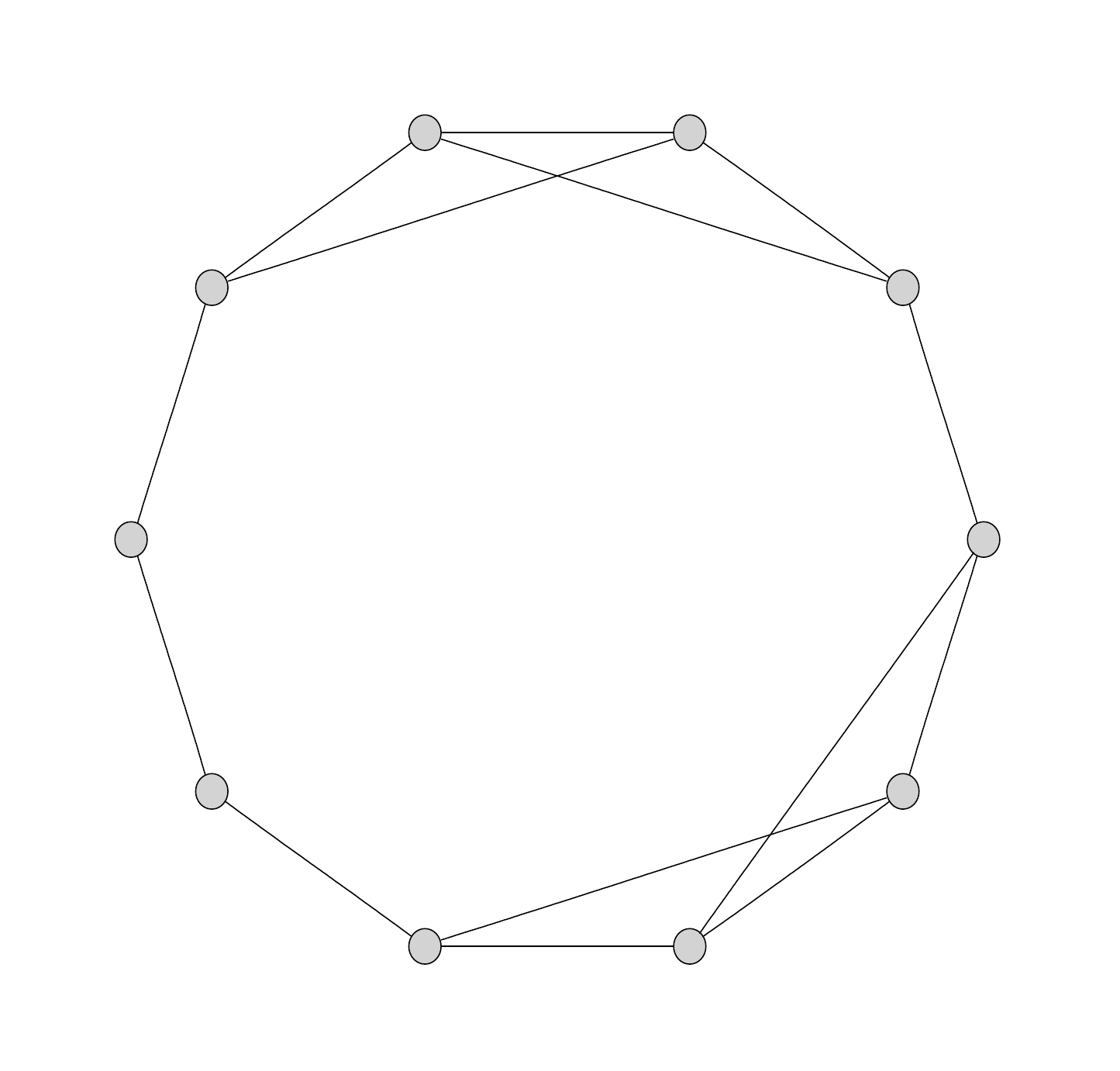} 
 \scalebox{0.26}{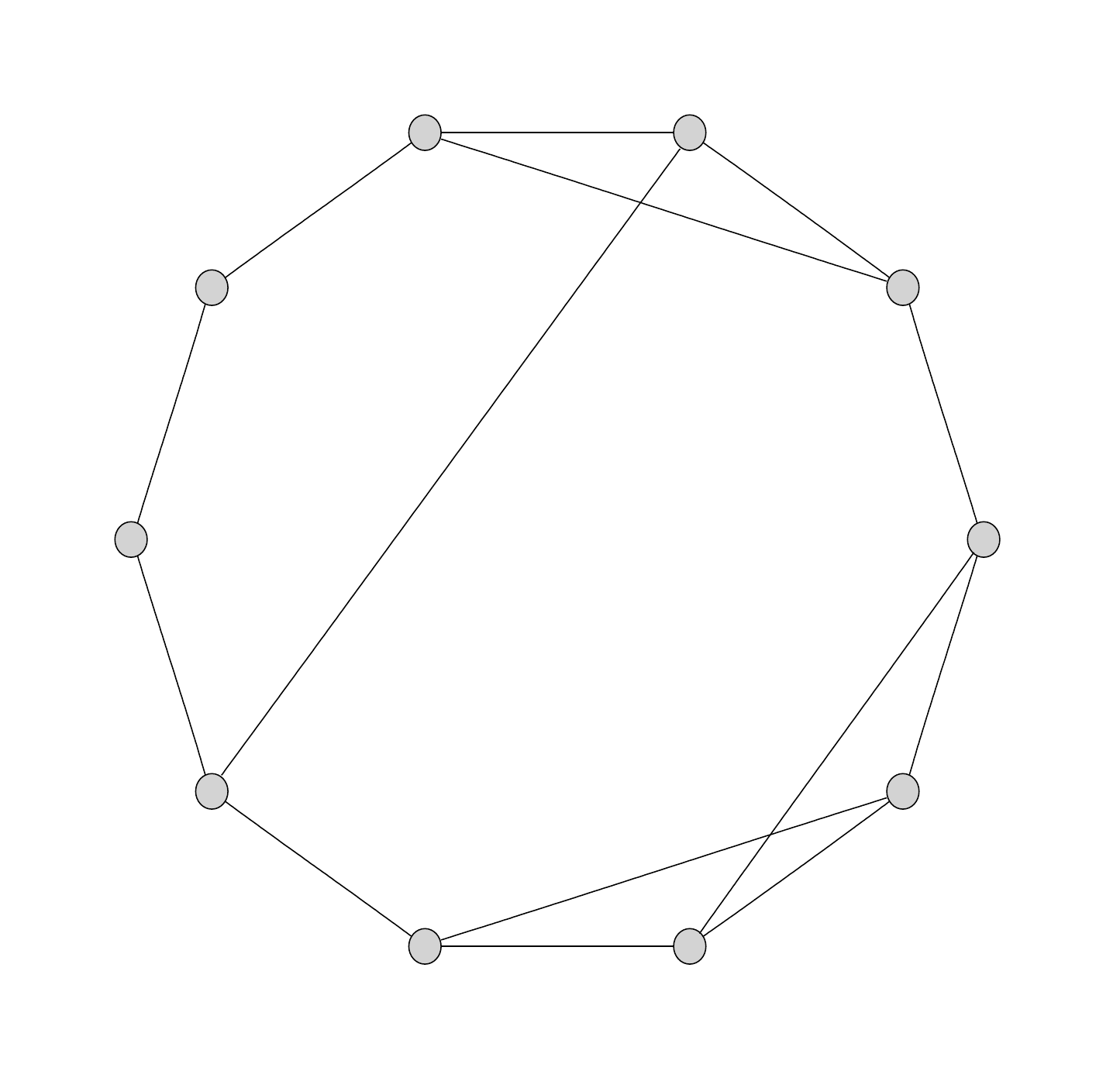}
 \scalebox{0.26}{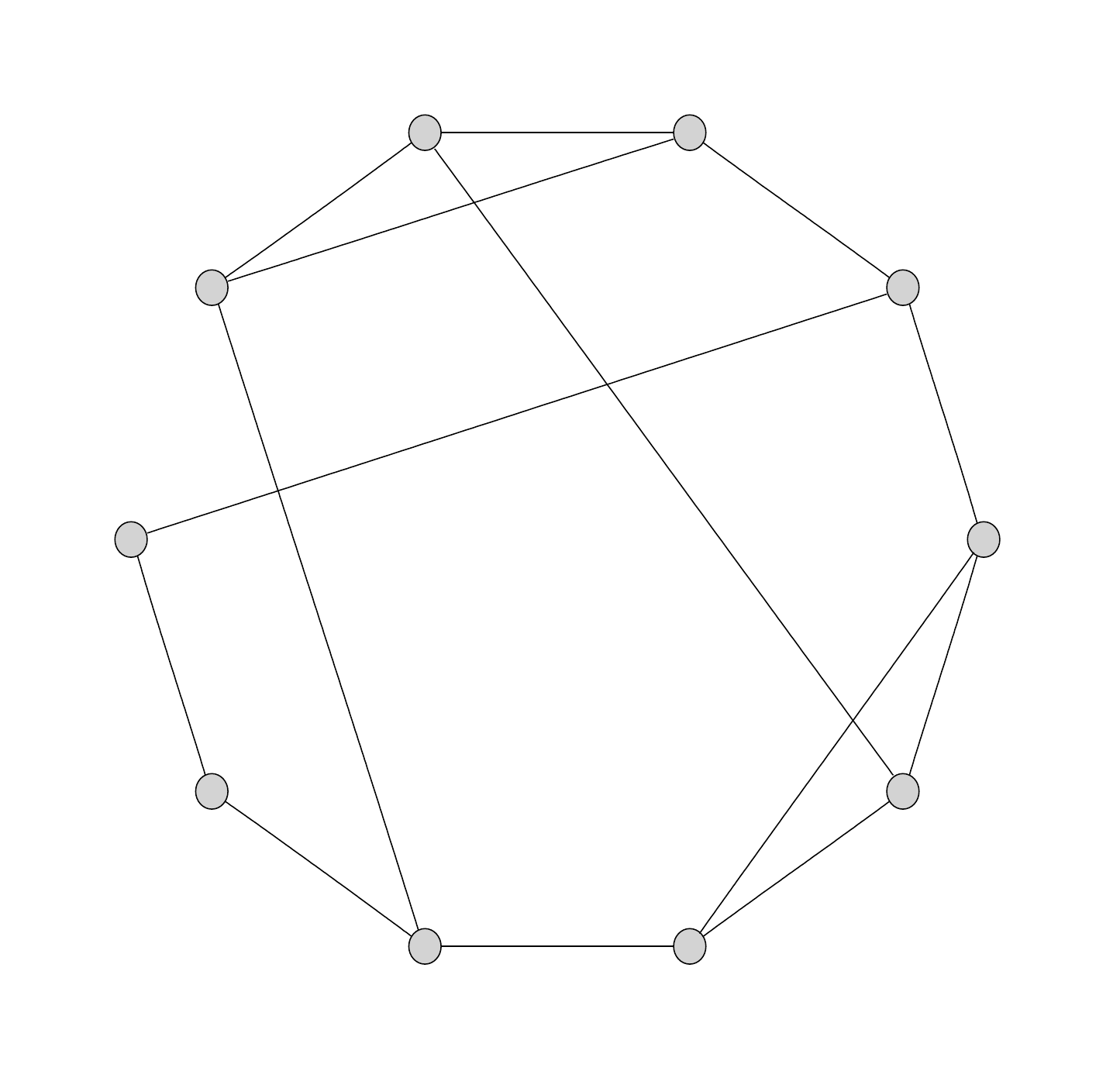} \\
\scalebox{0.26}{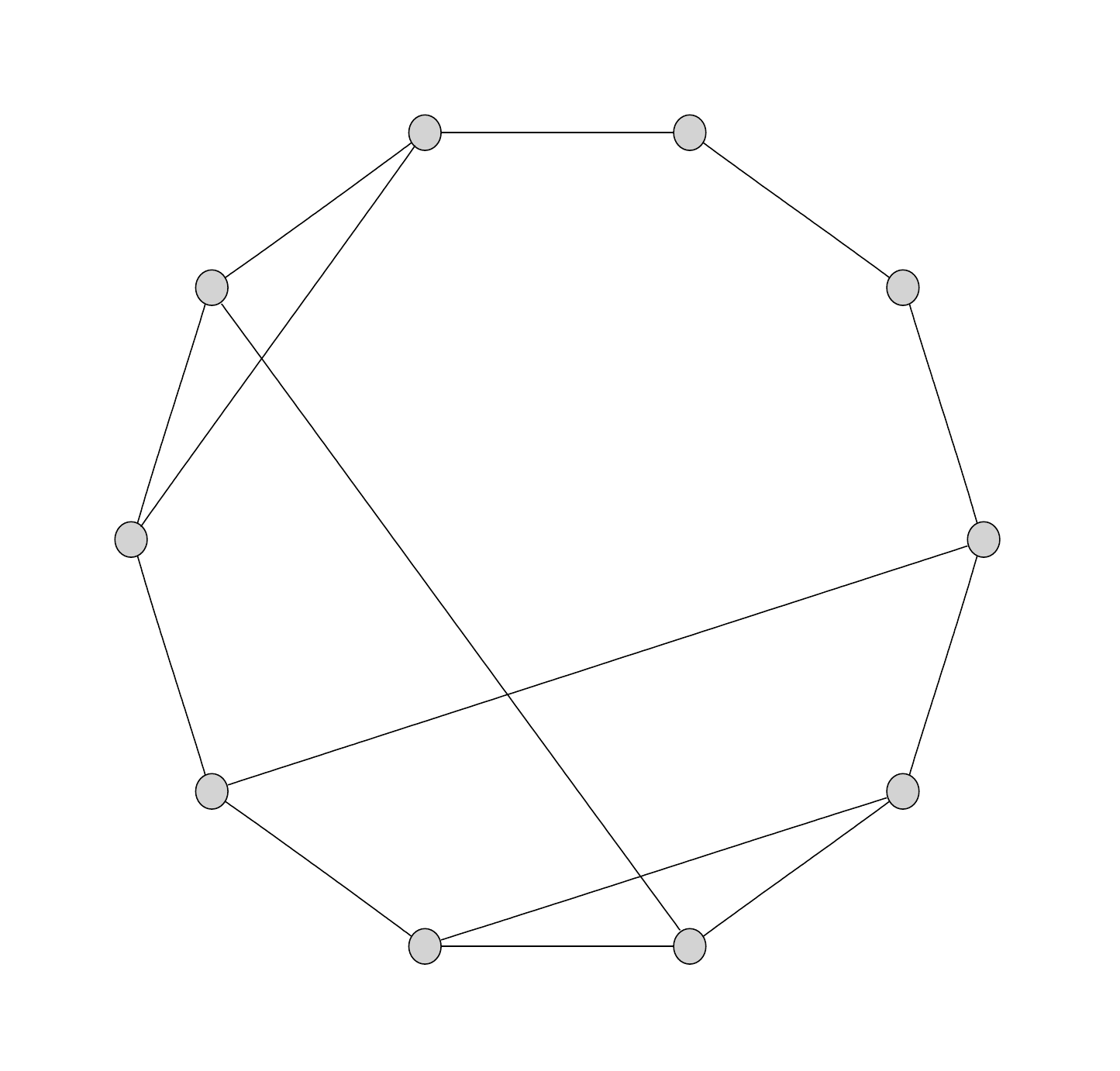}
\scalebox{0.26}{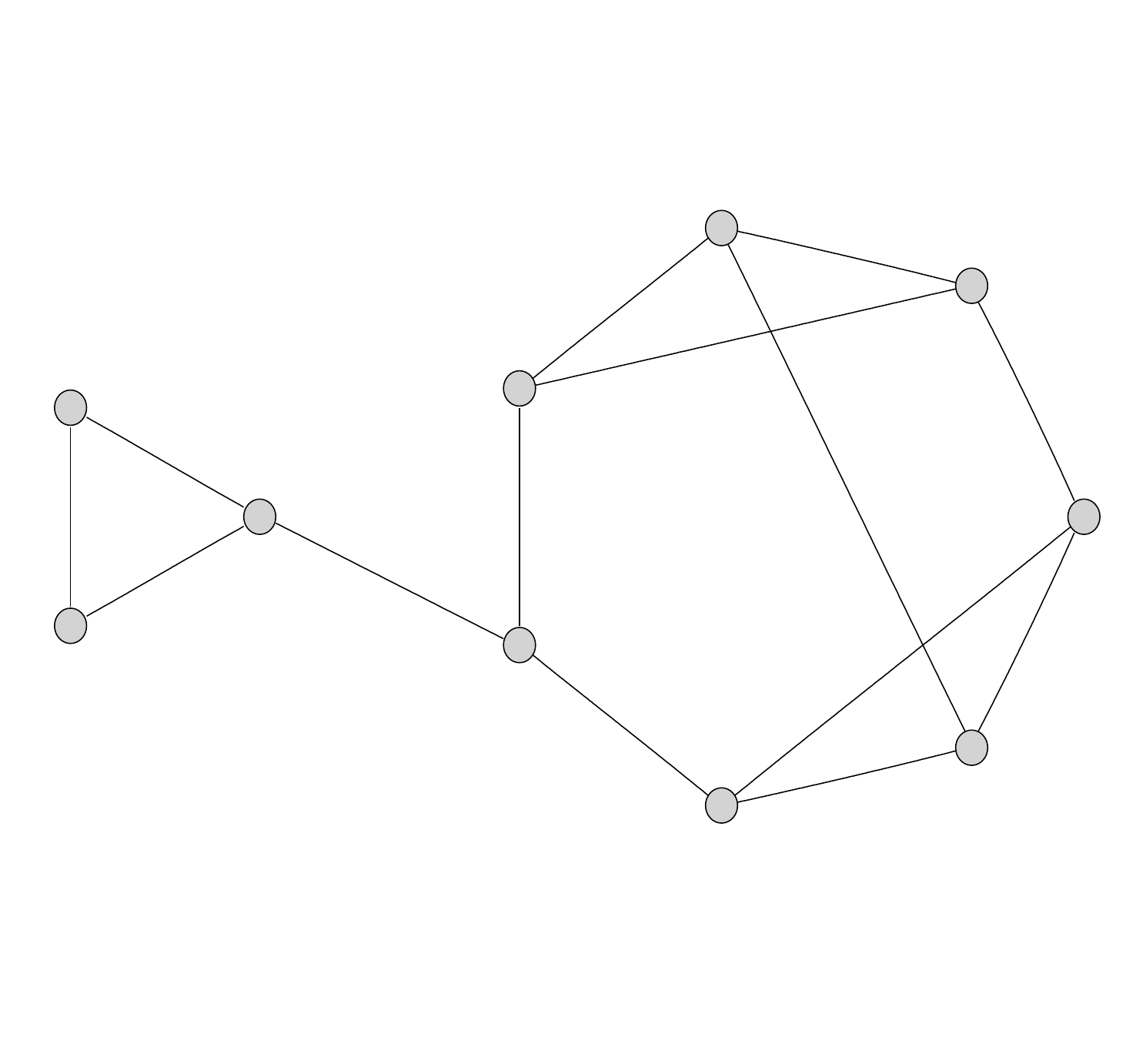}
\hspace{2mm} \scalebox{0.26}{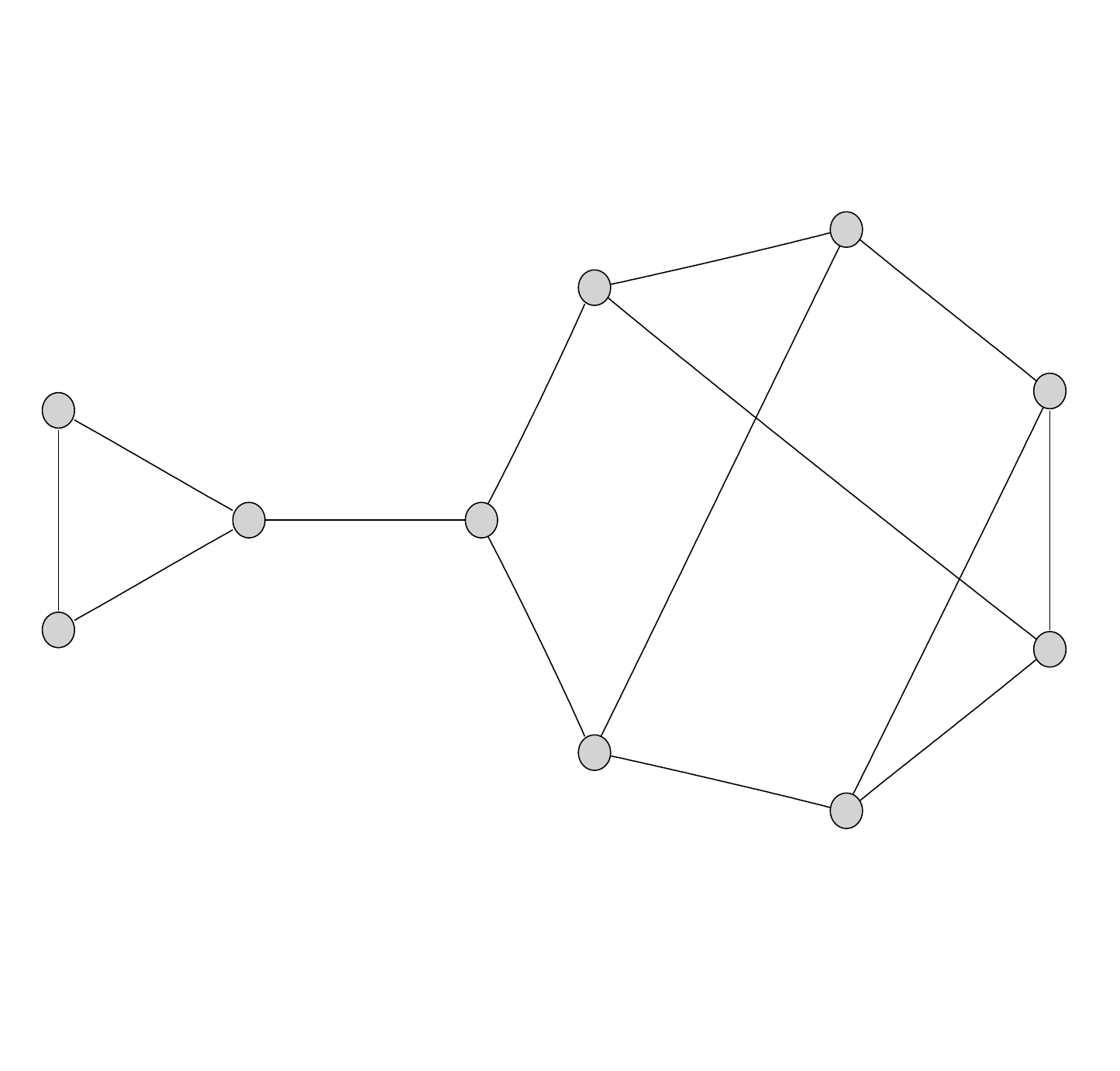} \\
\scalebox{0.26}{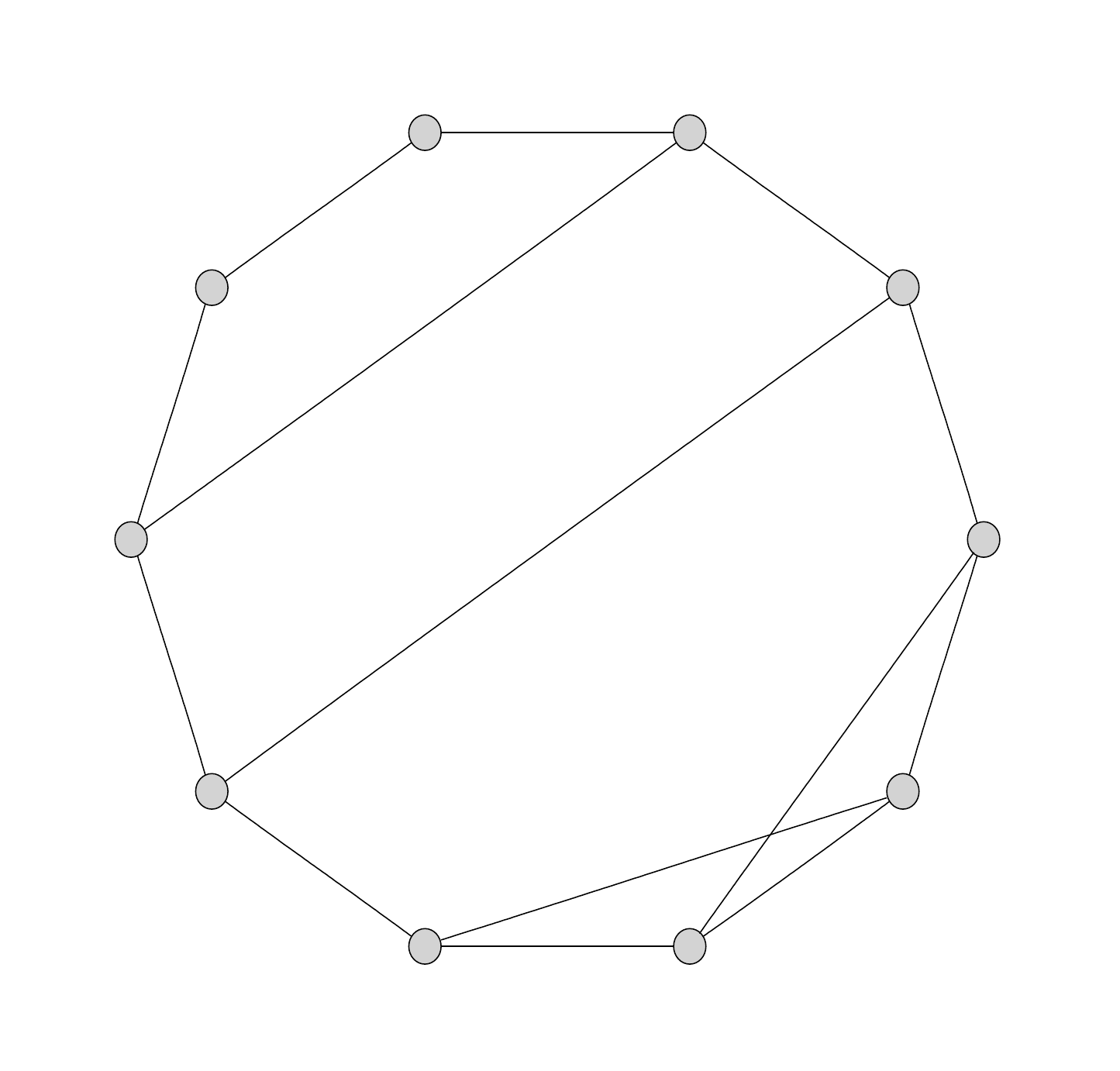}
 \scalebox{0.26}{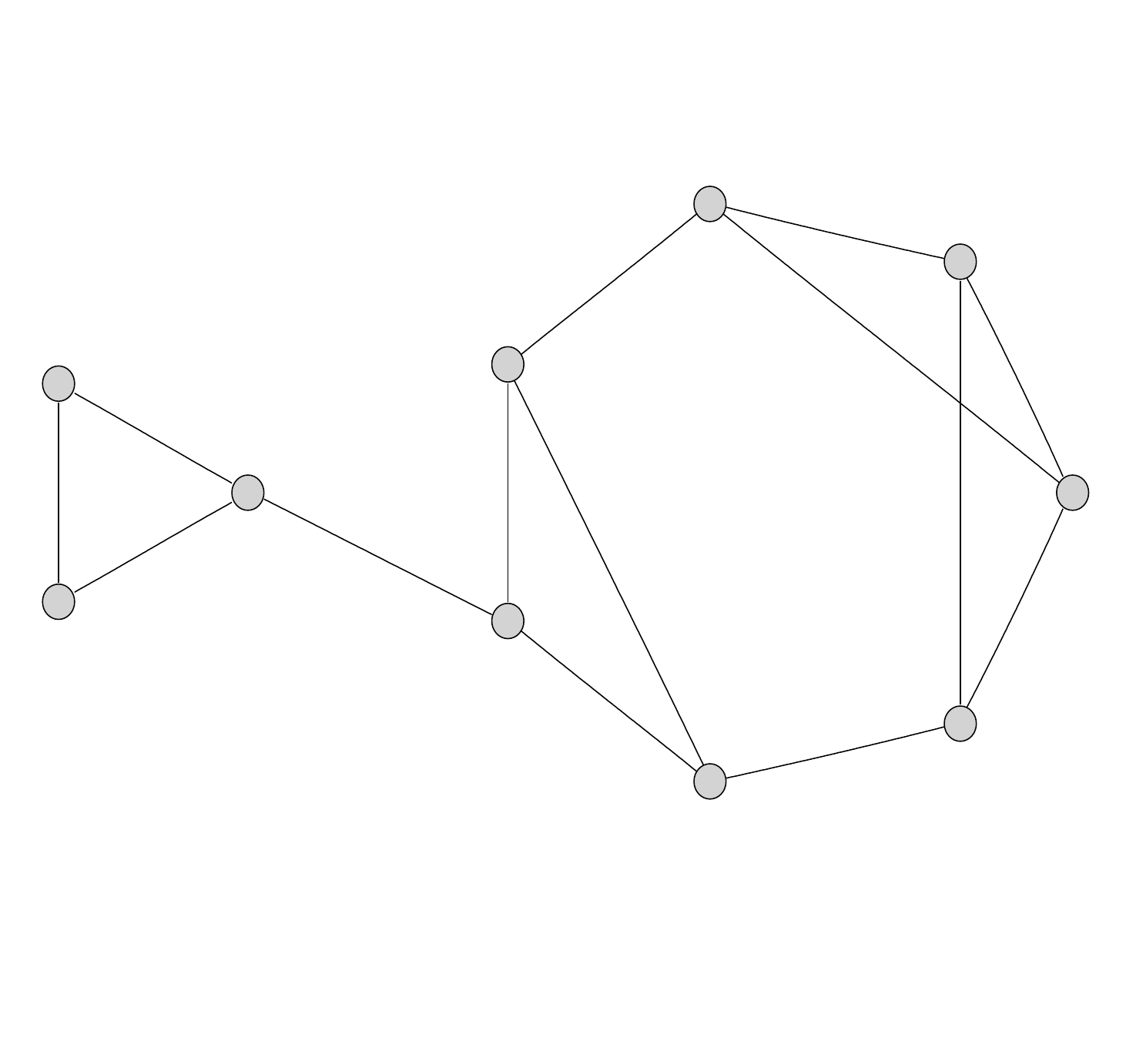}
\end{center}
\bibliographystyle{plainnat} 
\bibliography{natbibFull}
\end{document}

%% file: G7base.pdf_t
\begin{picture}(0,0)%
\includegraphics{G7base.pdf}%
\end{picture}%
%
%  Generated by Graphviz version 2.16 (Fri Feb  8 12:52:03 UTC 2008) 
%  For: (asp) Anders Sune Pedersen 
%  Title: test 
%  Pages: 1 
%
\setlength{\unitlength}{3947sp}%
\begingroup\makeatletter\ifx\SetFigFont\undefined%
\gdef\SetFigFont#1#2#3#4#5{%
  \reset@font\fontsize{#1}{#2pt}%
  \fontfamily{#3}\fontseries{#4}\fontshape{#5}%
  \selectfont}%
\fi\endgroup%
\begin{picture}(6924,6624)(-611,-5173)
\put(5251,-4636){\makebox(0,0)[b]{\smash{{\SetFigFont{25}{30.0}{\rmdefault}{\mddefault}{\updefault}{\color[rgb]{0,0,0}$G_7$}%
}}}}
\end{picture}%

%% file: G8base.pdf_t
\begin{picture}(0,0)%
\includegraphics{G8base.pdf}%
\end{picture}%
%
%  Generated by Graphviz version 2.16 (Fri Feb  8 12:52:03 UTC 2008) 
%  For: (asp) Anders Sune Pedersen 
%  Title: test 
%  Pages: 1 
%
\setlength{\unitlength}{3947sp}%
\begingroup\makeatletter\ifx\SetFigFont\undefined%
\gdef\SetFigFont#1#2#3#4#5{%
  \reset@font\fontsize{#1}{#2pt}%
  \fontfamily{#3}\fontseries{#4}\fontshape{#5}%
  \selectfont}%
\fi\endgroup%
\begin{picture}(6924,6624)(-611,-5173)
\put(5251,-4636){\makebox(0,0)[b]{\smash{{\SetFigFont{25}{30.0}{\rmdefault}{\mddefault}{\updefault}{\color[rgb]{0,0,0}$G_8$}%
}}}}
\end{picture}%

%% file: G9base.pdf_t
\begin{picture}(0,0)%
\includegraphics{G9base.pdf}%
\end{picture}%
%
%  Generated by Graphviz version 2.16 (Fri Feb  8 12:52:03 UTC 2008) 
%  For: (asp) Anders Sune Pedersen 
%  Title: test 
%  Pages: 1 
%
\setlength{\unitlength}{3947sp}%
\begingroup\makeatletter\ifx\SetFigFont\undefined%
\gdef\SetFigFont#1#2#3#4#5{%
  \reset@font\fontsize{#1}{#2pt}%
  \fontfamily{#3}\fontseries{#4}\fontshape{#5}%
  \selectfont}%
\fi\endgroup%
\begin{picture}(6924,6624)(-611,-5173)
\put(5251,-4636){\makebox(0,0)[b]{\smash{{\SetFigFont{25}{30.0}{\rmdefault}{\mddefault}{\updefault}{\color[rgb]{0,0,0}$G_9$}%
}}}}
\end{picture}%

%% file: G12base.pdf_t
\begin{picture}(0,0)%
\includegraphics{G12base.pdf}%
\end{picture}%
%
%  Generated by Graphviz version 2.16 (Fri Feb  8 12:52:03 UTC 2008) 
%  For: (asp) Anders Sune Pedersen 
%  Title: test 
%  Pages: 1 
%
\setlength{\unitlength}{3947sp}%
\begingroup\makeatletter\ifx\SetFigFont\undefined%
\gdef\SetFigFont#1#2#3#4#5{%
  \reset@font\fontsize{#1}{#2pt}%
  \fontfamily{#3}\fontseries{#4}\fontshape{#5}%
  \selectfont}%
\fi\endgroup%
\begin{picture}(6924,6624)(-611,-5173)
\put(5476,-4786){\makebox(0,0)[b]{\smash{{\SetFigFont{25}{30.0}{\rmdefault}{\mddefault}{\updefault}{\color[rgb]{0,0,0}$G_{12}$}%
}}}}
\end{picture}%

%% file: G17base.pdf_t
\begin{picture}(0,0)%
\includegraphics{G17base.pdf}%
\end{picture}%
%
%  Generated by Graphviz version 2.16 (Fri Feb  8 12:52:03 UTC 2008) 
%  For: (asp) Anders Sune Pedersen 
%  Title: test 
%  Pages: 1 
%
\setlength{\unitlength}{3947sp}%
\begingroup\makeatletter\ifx\SetFigFont\undefined%
\gdef\SetFigFont#1#2#3#4#5{%
  \reset@font\fontsize{#1}{#2pt}%
  \fontfamily{#3}\fontseries{#4}\fontshape{#5}%
  \selectfont}%
\fi\endgroup%
\begin{picture}(6924,6624)(-611,-5173)
\put(2701,-4861){\makebox(0,0)[b]{\smash{{\SetFigFont{29}{34.8}{\rmdefault}{\mddefault}{\updefault}{\color[rgb]{0,0,0}$G_{17}$}%
}}}}
\end{picture}%

%% file: G20.pdf_t
\begin{picture}(0,0)%
\includegraphics{G20.pdf}%
\end{picture}%
%
%  Generated by Graphviz version 2.16 (Fri Feb  8 12:52:03 UTC 2008) 
%  For: (asp) Anders Sune Pedersen 
%  Title: test 
%  Pages: 1 
%
\setlength{\unitlength}{3947sp}%
\begingroup\makeatletter\ifx\SetFigFont\undefined%
\gdef\SetFigFont#1#2#3#4#5{%
  \reset@font\fontsize{#1}{#2pt}%
  \fontfamily{#3}\fontseries{#4}\fontshape{#5}%
  \selectfont}%
\fi\endgroup%
\begin{picture}(6924,6624)(-611,-5173)
\put(4501,-3061){\makebox(0,0)[b]{\smash{{\SetFigFont{29}{34.8}{\rmdefault}{\mddefault}{\updefault}{\color[rgb]{0,0,0}$G_{20}$}%
}}}}
\end{picture}%

%% file: G21.pdf_t
\begin{picture}(0,0)%
\includegraphics{G21.pdf}%
\end{picture}%
%
%  Generated by Graphviz version 2.16 (Fri Feb  8 12:52:03 UTC 2008) 
%  For: (asp) Anders Sune Pedersen 
%  Title: test 
%  Pages: 1 
%
\setlength{\unitlength}{3947sp}%
\begingroup\makeatletter\ifx\SetFigFont\undefined%
\gdef\SetFigFont#1#2#3#4#5{%
  \reset@font\fontsize{#1}{#2pt}%
  \fontfamily{#3}\fontseries{#4}\fontshape{#5}%
  \selectfont}%
\fi\endgroup%
\begin{picture}(6924,6624)(-611,-5173)
\put(4501,-3061){\makebox(0,0)[b]{\smash{{\SetFigFont{29}{34.8}{\rmdefault}{\mddefault}{\updefault}{\color[rgb]{0,0,0}$G_{21}$}%
}}}}
\end{picture}%